\documentclass[11pt, reqno]{amsart}
\usepackage{wrapfig, lipsum,booktabs}
\usepackage{graphicx}
\usepackage{amssymb}
\usepackage{epstopdf}
\usepackage{verbatim}
\usepackage{bm}
\usepackage{multicol}
\usepackage{multirow}
\usepackage{subfigure}
\usepackage{float}
\usepackage{color}
\usepackage{soul}
\usepackage{tikz}
\usepackage{tikz-3dplot}
\usepackage{pgfplots}
\newtheorem{theorem}{Theorem}[section]

\newtheorem{remark}{Remark}[section]

\usepackage{listings}

\usepackage{multirow}
\usepackage{amsmath,amssymb,eucal}
\usepackage{graphicx,subfigure,epsfig}
\usepackage{url}
\usepackage[top=1in, bottom=1.in, left=1in, right=1in]{geometry}
\newcommand{\bld}[1]{\hbox{\boldmath$#1$}}
\newcommand{\Th}{\mathcal{T}_h}

\newcommand{\Eh}{\mathcal{E}_h}

\setulcolor{red}

\begin{document}
\title[Hybridized discontinuous Galerkin]{A hybridizable discontinuous Galerkin method
on unfitted meshes for single-phase Darcy flow in fractured porous media}
\author{Guosheng Fu}
\address{Department of Applied and Computational Mathematics and
Statistics, University of Notre Dame, USA.}
\email{gfu@nd.edu}
\author{Yang Yang}
\address{Department of Mathematical Sciences, Michigan Technological University, USA.}
\email{yyang7@mtu.edu}
 \thanks{
 G. Fu was partially supported by the NSF grant DMS-2012031. Y. Yang was supported by the Simon Foundation 961585.
 }

\begin{abstract}
  We present a novel hybridizable discontinuous Galerkin (HDG) method on unfitted meshes for single-phase Darcy flow in a fractured porous media.
  In particular we apply the HDG methodology to the recently introduced 
  reinterpreted discrete
fracture model (RDFM) \cite{Xusubmit} that use Dirac-$\delta$ functions to model both conductive and blocking fractures. 
Due to the use of Dirac-$\delta$ function approach for the fractures, our numerical scheme naturally allows for unfitted meshes with respect to the fractures, which is the major novelty of the proposed scheme.
Moreover, the scheme is locally mass conservative and is relatively easy to implement comparing with existing work on the subject. In particular, our scheme is a simple modification of an existing regular Darcy flow HDG solver by adding the following two components: (i) locate the co-dimension one fractures in the background mesh and adding the appropriate surface integrals associated with these fractures into the stiffness matrix, (ii) adjust the penalty parameters on cells cut through conductive and blocking fractures (fractured cells).

Despite the simplicity of the proposed scheme, it performs extremely well for various benchmark test cases in both two- and three-dimensions.
This is the first time that a truly unfitted finite element scheme been applied to complex fractured porous media flow problems
in 3D with both blocking and conductive fractures without any restrictions on the meshes.
\end{abstract}
\maketitle

\section{Introduction}
\label{sec:intro}

Many applications in contaminant transportation, petroleum engineering and radioactive waste deposit can be modeled by single- and multi-phase flows in porous media. A typical porous media may contain conductive fractures with tiny thickness but high permeability. Some fractures may be filled with minerals and debris, forming blocking fractures with low permeability. Mathematical modeling and numerical simulation for flows in fractured porous media are challenging due to the highly heterogeneity of the porous media.

Several effective mathematical models have been developed in the literature for simulating flows in porous media with conductive fractures, such as the dual porosity model \cite{DualPoro1,DualPoro2,geiger}, single porosity model \cite{SinglePoro1}, traditional discrete fracture model (DFM) \cite{FEMDFM1,FEMDFM2,SuperposeStiffnessMat,FEMDFM,DFMpaper1,FEMDFM3,FEMDFM4}, embedded DFM (EDFM) \cite{firstEDFM,secondEDFM,EDFM3,pEDFM,EDFM4,CrossShaped,EDFM5}, the interface models \cite{Alboin99,Interfaces2,Interfaces3,benchmark2} and extended finite element DFM (XDFM) based on the interface models \cite{XFEMDFM1,XFEMDFM2,ThesisXFEM,XFEMDFM3,XFEMDFM4}, finite element method based on Lagrange multipliers \cite{LMFEM2D,LMFEM2D2, LMFEM3D}, etc. Among the above mentioned works, the interface model \cite{Interfaces5,Interfaces6,Boon2018,EG20} and the projection-based EDFM (pEDFM) \cite{pEDFM,Jiang2017} can also be used for problems containing blocking fractures. The interface model is to explicitly represent the fracture as the interface of a porous media, and the governing equations in the porous media and fractures can be constructed. In the interface model, the matrix and fractures are considered as two different system and the mass transfer between them is given by the jump of the velocity. Therefore, the interface model requires fitted mesh, i.e. the fracture is located at the cell skeletons. Though hanging nodes are allowed, numerical methods based on fitted meshes may suffer from low quality meshes. To fix this gap, XDFM \cite{XFEMDFM1,XFEMDFM2,ThesisXFEM,XFEMDFM3,XFEMDFM4} was introduced. However, such treatment may significantly increase the degrees of freedom (DoFs), hence is not of practical use, especially for problems with high geometrical complexity \cite{FLEMISCH2018239}. Another possibility to extend the interface model to unfitted mesh is to use the CutFEM \cite{CutFEM}. However, this method may not work for media with complicated fractures, as the fractures have to separate the domain into completely disjoint subdomains. The pEDFM \cite{pEDFM,Jiang2017} is another way to simulate flow in porous media with blocking fractures. The basic idea is to reduce the effective flow area between the blocking fracture and the adjacent matrix cells based on the property of the blocking fracture. However, most of previous works in this direction work for rectangular meshes, and the extension to general triangular or arbitrary meshes seems to be complicated.

In \cite{Xu2020}, one of the authors introduced the reinterpreted discrete fracture model (RDFM) for single-phase flow in porous media with conductive fractures. Different from the interface model and pEDFM, the RDFM couples the fracture and matrix in one system and use one equation to model the flows in both matrix and fractures. The basic idea is to use Dirac-$\delta$ functions to represent the lower dimensional fractures in the system containing higher dimensional matrix. The effect of the Dirac-$\delta$ functions is to increase the permeability at the location of the conductive fractures. Later, the RDFM was successfully applied to simulate contaminant transportation in \cite{Feng2021}. As an extension, the RDFM was further developed to simulate flows in porous media with both conductive and blocking fractures in \cite{Xusubmit}. Similar to the RDFM for conductive fractures \cite{Xu2020}, the blocking fractures were also described as Dirac-$\delta$ functions, and they are used to increase the flow resistance. The RDFM incorporates the information of the fractures into the equation, hence it works for arbitrary meshes without any restrictions. In \cite{Xusubmit}, the local discontinuous Galerkin (LDG) methods were applied to RDFM. An extremely large penalty of order $\mathcal{O}(h^{-3})$ was added to the pressure on cell interfaces without blocking fractures, while a moderate large penalty of order $\mathcal{O}(h^{-2})$ was applied to the normal direction of the velocity on cell interfaces with blocking fractures. The penalty is used to intimate the continuity requirement of the target variables. Unfortunately, this LDG method 
leads to a fully coupled saddle-point linear system for the velocity and pressure, hence its practical application in three dimension is limited.
Moreover, the lowest order scheme therein use (discontinuous) piecewise linear functions as the piecewise constant version did not lead to a convergent algorithm. 
Furthermore, the effect of the penalty and well-posedness of the method was not clear. Besides the above works, the RDFM for blocking fractures was also combined with the interface model for conductive fractures in \cite{FuYang22} where fitted meshes for conductive fractures were required.

In this paper, we apply the hybridizable discontinuous Galerkin (HDG) methods
for RDFM-based single-phase flows in porous media. Similar to the idea given in \cite{Xusubmit}, the proposed method (1) produces locally conservative velocity approximations; (2) works for problems containing both conductive and blocking fractures; (3) can be applied to arbitrary meshes without any restrictions. In additional to the above, there are several advantages of the proposed method that were not enjoyed by the one given in \cite{Xusubmit}. First of all, the HDG method can be efficiently solved via static condensation, which leads to a symmetric positive definite (SPD) linear system 
for the pressure degrees of freedom (DOFs) on the mesh skeletons only. Hence they can be implemented very efficiently comparing with the LDG scheme \cite{Xusubmit}.
As a result, three-dimensional simulation for complex fracture networks are now possible. 
% This is the main advantage compared with the XDFM, pEDFM, CutFEM and RDFM with LDG methods. Therefore, the proposed method works for problems in three dimensions. 
Secondly, the penalty parameters in the HDG scheme is only adjusted on cells contain fractures, with extra (pressure) stabilization on conductive fractures and reduced (pressure) stabilization on blocking fractures.
With a judicious choice of the penalty parameters, the lowest order HDG scheme with piecewise constant approximations can now  yield satisfactory numerical results. 
% This adaptive choice of penalty parameters makes the lowest order HDG scheme to be convergent.
% We numerically verified that the penalty parameters chosen in \cite{Xusubmit} are optimal.
% Moreover, with the new choice of penalty locations, the piecewise constant numerical approximations may yield satisfactory numerical results. 
Finally, the well-posedness of the proposed method can be guaranteed theoretically, which was completely missing for the method given in \cite{Xusubmit}. As an application, it is straightforward to couple the proposed flow equation with the transport equations and construct the locally conservative numerical methods for the transport equations. However, that is not the main target of this paper, so we will discuss the applications in the future.
The combination of these properties for the HDG scheme on unfitted meshes makes it highly competitive comparing with existing works for fractured porous media that can simultaneously handle blocking and conductive fractures both in terms of algorithmic complexity and numerical accuracy.

The rest of the paper is organized as follows. In Section \ref{sec:model}, we present the RDFM and the HDG methods to be used. Numerical results for various benchmark test cases are presented in Section \ref{sec:numerics}. Some concluding remarks will be given in Section \ref{sec:conclude}.

\section{The HDG scheme}
\label{sec:model}
\subsection{The model}
We consider the following RDFM proposed in \cite{Xusubmit}:
\begin{subequations}
\label{rdfm}
\begin{align}
\label{rdfmX}
    \left(\bld I+\bld K_m \sum_{i=1}^M \frac{\epsilon_i}{k_{i}}\delta_{\Gamma_i}
    \bld n_i\bld n_i^T
    \right)\bld u = &\; 
    -\left(\bld K_m +\sum_{i=M+1}^{M+N} {\epsilon_i}{k_{i}}\delta_{\Gamma_i}
    (\bld I -\bld n_i\bld n_i^T)
    \right)\nabla p,\\
    % \bld s = &\; \nabla p,\\
    \nabla \cdot \bld u = &\; f,
    \end{align}
    on a $d$-dimensional domain $\Omega$ with $d=2,3$. 
Here $\bld u$ is the total Darcy velocity, $p$ the pressure, $\bld K_m$ is the matrix permeability, 
% $\bld s=\nabla p$ is the pressure gradient, 
$\bld I$ is the identity tensor, $f$ is the source term, and
$\Gamma_i$ is the location of the $i$-th $(d-1)$-dimensional fracture with thickness $\epsilon_i$, permeability 
$k_{i}$ and normal direction $\bld n_i$ for $1\le i\le M+N$, where we assume the first $M$ fractures are blocking while the last $N$ fractures are conductive, i.e., $k_i\ll \bld K_m$ for $i\le M$ and 
$k_i\gg \bld K_m$ for $i\ge M+1$. 
\end{subequations}
Moreover, $\delta_{\Gamma_i}$ is the Dirac-$\delta$ function such that $\delta_{\Gamma_i}(\bld x) = \infty$ if $\bld x\in \Gamma_i$, 
$\delta_{\Gamma_i}(\bld x) = 0$ if $\bld x\not \in \Gamma_i$ and $\int_{\Omega}\delta_{\Gamma_i}\mathrm{dx} = 1$.
For simplicity, we assume the model \eqref{rdfm} is equiped with the homogeneous Dirichlet boundary condition $p=0$ on $\partial \Omega$. Other boundary conditions will be used in 
the numerical experiments.

\begin{remark}[On RDFM]
In \eqref{rdfm}, we apply Dirac-$\delta$ functions to bridge the difference of the dimensions between the matrix and fractures as the Dirac-$\delta$ functions concentrate all the information at its concentration. The first $M$ Dirac-$\delta$ functions are used to increase the flow resistance and are for blocking fractures. Due to the small thickness of the blocking fractures, the effect of the blocking fractures in the tangential direction is negligible. Therefore, we include the tensor associate with the normal direction in the model. Similarly, the last $N$ Dirac-$\delta$ functions are used to increase the permeability and are for conductive fractures. Moreover, due to the small thickness of the blocking fractures, the effect of the conductive fractures in the normal direction is negligible, and we include the tensor associate with the tangential directions in the model. We refer to \cite{Xu2020} and \cite{Xusubmit} for more discussion on the RDFM.
\end{remark}

The Dirac-$\delta$ function approach in the above model avoids a direct modeling of the fractures using lower dimensional Darcy flows as typical in the interface models \cite{Interfaces5,Interfaces6,Boon2018,EG20}. Hence, unfitted mesh discretizations can be naturally applied. In the original work \cite{Xusubmit}, an LDG scheme on unfitted meshes was devised for \eqref{rdfm} in two dimensions with satisfactory numerical results.  However, well-posedness of the LDG scheme was not established as no energy identity exists, and its computational cost is relatively large comparing with existing works on (partially) fitted meshes.

Here we will devise a well-posed HDG scheme on unfitted meshes for the above model \eqref{rdfm}, which is not only computationally cheaper than the LDG scheme \cite{Xusubmit} but also more accurate and has an energy identity.
To this end, we introduce the Darcy velocity in the matrix 
$\widetilde{\bld u}:=-\bld K_m\nabla p$ as a new unknown and rewrite the model \eqref{rdfm} into the following 
three-field formulation:
\begin{subequations}
\label{rdfm-1} 
\begin{align}
\label{rdfmX1}
 \left(\bld K_m^{-1} +\sum_{i=1}^M \frac{\epsilon_i}{k_i}\delta_{\Gamma_i}
    \bld n_i\bld n_i^T
    \right)\bld u = &\; 
    \left(\bld I+\sum_{i=M+1}^{M+N} \bld K_m^{-1}{\epsilon_i}k_i\delta_{\Gamma_i}
    (\bld I -\bld n_i\bld n_i^T)
    \right)\bld K_m^{-1}\widetilde{\bld u},\\
\bld K_m^{-1}    \widetilde{\bld u}+\nabla p = &\;0,\\
    \nabla \cdot \bld u = &\; f,
    \end{align}
\end{subequations}
% where $\gamma_i:=\bld K_m^{-1}k_i$ denotes the permeability 
% ratio. 
Note that we multiplied equation \eqref{rdfmX} with $\bld K_m^{-1}$
on the left to obtain the equation \eqref{rdfmX1}.
We emphasis that here $\bld u$ is the total Darcy velocity that contains information about the fractures, whilst   $\widetilde{\bld u}$
is the Darcy velocity on the matrix without fracture contributions.

% \begin{remark}[Nondimensionalization]
% In practice, we usually work with nondimensionalized form of the model \eqref{rdfm-1} to avoid extreme values. 
% Denoting the following dimensionless variables:
% \[
% \bar x= x/L, \quad \bar {\bld K}_m = \bld K_m/\kappa, \quad 
% \bar p= p/P, \quad 
% \bar{\bld u} = \bld u/(\frac{\kappa P}{L}),\quad 
% \bar{\widetilde{\bld u}} = {\widetilde{\bld u}}/(\frac{\kappa P}{L}),
% \text{ and }
% \bar f = f/(\frac{\kappa P}{L^2}),
% \]
% where $L$ is the characteristic length of the domain $\Omega$, 
% $\kappa$ is the effective permeability, and $P$ is the maximum pressure. 
% Then the nondimensionalized form of \eqref{rdfm-1} is simply 
% to adding bars to the variables in \eqref{rdfm-1}. Without loss of generality, we assume the equation \eqref{rdfm-1} has already been nondimensionalized.
% \end{remark}

\newcommand{\Vh}{\bld V_h}
\newcommand{\Wh}{W_h}
\newcommand{\Mh}{M_h}

\subsection{The HDG scheme}
% \subsubsection{Preliminaries}
Let $\Th:=\{K\}$ be a triangulation
 of the domain $\Omega$ that is unfitted to the location of the fractures. Let $\Eh$ be the collections of $(d-1)$-dimensional facets
 (edges for $d=2$, faces for $d=3$)
 of $\Th$.
 We use level set functions to represent the fractures $\Gamma_i$. 
 In particular, 
 \begin{itemize}
     \item 
      if $\Gamma_i$ is a closed curve/surface without boundaries, it is simply approximated by the zero level set of a continuous piecewise linear function 
 $\phi_i\in W_h^1\cap H^1(\Omega)$ 
 on the mesh
 $\Th$:
     \[
     \Gamma_{i,h}:= \{x\in \Omega: \phi_i(x)=0\}.
     \]
     \item 
      if $\Gamma_i$ is a closed curve/surface with $(d-2)$-dimensional boundary 
      $\partial \Gamma_i$, it is approximated by 
      a main level set function $\phi_i\in W_h^1\cap H^1(\Omega)$ for the (extended) surface $\Gamma_i$ and additional level set functions $\psi_i^j\in W_h^1\cap H^1(\Omega)$ for $j=1, \cdots L$ to take care of the boundary $\partial\Gamma_i$:
     \[
     \Gamma_{i,h}:= \{x\in \Omega: \phi_i(x)=0\}
     \cap_{j=1}^L \{x\in \Omega: \psi_i^j(x)< 0\}.
     \]
     In practice, usually two level set functions are sufficient to provide a good approximation of
     $\Gamma_i$, i.e., 
          \[
     \Gamma_{i,h}:= \{x\in \Omega: \phi_i(x)=0\}
     \cap \{x\in \Omega: \psi_i^1(x)< 0\}.
     \]
 \end{itemize}
 Hence the discrete fractures $\Gamma_{i,h}$ on each element 
 $K\in\Th$ is always a line segment in 2D or a polygon in 3D due to the use of piecewise linear functions as level sets. 
With $\Gamma_{i,h}$ ready, we split the cells in $\Th$
into disjoint three groups
\[
\Th:= \Th^r\cup \Th^b\cup \Th^c,
\]
where $\Th^r$ contains regular cells without fractures,
$\Th^b$ contains cells with blocking fractures, and 
 $\Th^c$ contains cells with conductive fractures defined as follows:
 \begin{alignat*}{2}
     \Th^b:=&\;\{K\in\Th: &&\;
     \exists i\in [1, M] \text{ such that  }
     K\cap \Gamma_{i,h} \not= \emptyset
     \},\\
     \Th^c:=&\;\{K\in\Th\backslash \Th^b: &&\;
     \exists i\in [M+1, M+N] \text{ such that  }
     K\cap \Gamma_{i,h} \not= \emptyset
     \},\\
     \Th^r:=&\;\Th\backslash \{\Th^b\cup \Th^c\}.
 \end{alignat*}
Note that when both blocking and conductive fractures appear
in a cell, we always treat it as a blocking cell in $\Th^b$ and we will ignore the conductive fractures within that cell in the discretization for stability considerations.
Note also that we allow fractures to be intersecting with each other in an arbitrary fashion within a single cell as long as a discrete characterization of the fracture $\Gamma_{i,h}$ using (multi-)level sets is possible. 

 Given a polynomial degree $k\ge 0$, we 
 consider the following finite element spaces:
 \begin{subequations}
  \label{space}
 \begin{align}
   \label{space-u}
   \Vh^k :=&\; \{\bld v\in [L^2(\Th)]^d:\;
   \bld v|_K\in [P_k(K)]^d,\quad \forall K\in\Th\},\\
   \Wh^k :=&\; \{w\in L^2(\Th):\;
   w|_K\in P_k(K),\quad \forall K\in\Th\},\\
   \Mh^k :=&\; \{\mu\in L^2(\Eh):\;
   \mu|_F\in P_k(F),\quad \forall F\in\Eh, \quad 
   \mu|_F = 0 \text{ on }\partial\Omega\},
 \end{align}
\end{subequations}
where $P_k(S)$ is the polynomial space of degree $k$ on $S$.
 We further denote the following inner products to simplify notation:
\begin{alignat*}{2}
  (\phi, \psi)_{\Th}: =&\;\sum_{K\in \Th}\int_{K}\phi\,\psi\,\mathrm{dx},&&\quad
  \quad
  \langle\phi, \psi\rangle_{\partial\Th}: =\;\sum_{K\in
\Th}\int_{\partial K}\phi\,\psi\,\mathrm{ds}.
\end{alignat*}

% \subsubsection{The HDG scheme}
The HDG scheme for \eqref{rdfm-1} is now given as follows:
Find $(\bld u_h, \widetilde{\bld u}_h, p_h, \widehat p_h)\in \Vh^k\times \Vh^k\times \Wh^k\times \Mh^k$ such that
\begin{subequations}
 \label{hdg}
 \begin{align}
 \label{hdg-1}
  (\bld K_m^{-1}\bld u_h,\widetilde{\bld  v}_h)_{\Th}
  +\Phi_{b}(\bld u_h, \widetilde{\bld  v}_h)
= \; 
  (\bld K_m^{-1}\widetilde{\bld u}_h,\widetilde{\bld  v}_h)_{\Th}
  +\Phi_{c}(\widetilde{\bld u}_h, \widetilde{\bld  v}_h),\\
 \label{hdg-2}
(\bld K_m^{-1}\widetilde{\bld  u}_h, \bld  v_h)_{\Th}
-(p_h, \nabla\cdot\bld  v_h)_{\Th}
+\langle\widehat p_h, \bld  v_h\cdot\bld n\rangle_{\partial\Th}=&\;0,\\
 \label{hdg-3}
-(\bld  u_h, \nabla q_h)_{\Th}
+\langle\widehat{\bld u}_h\cdot\bld n, q_h\rangle_{\partial\Th}=&\;(f, q_h)_{\Th}
,\\
 \label{hdg-4}
\langle\widehat{\bld u}_h\cdot\bld n, \widehat q_h\rangle_{\partial\Th}=&\;0,
 \end{align}
for all 
$(\bld v_h, \widetilde{\bld v}_h, q_h, \widehat q_h)\in \Vh^k\times \Vh^k\times \Wh^k\times \Mh^k$, where 
$\Phi_b$/$\Phi_c$ contain the following blocking/conductive fracture surface integrals (taking into account the property of the Dirac-$\delta$ functions):
\begin{align}
 \Phi_b(\bld u, \bld v):=&\;
 \sum_{K\in\Th^b}\sum_{i=1}^M\int_{K\cap \Gamma_{i,h}}
 \frac{\epsilon_i}{k_i}(\bld u\cdot\bld n_i)(\bld v\cdot\bld n_i)\,\mathrm{ds},\\
 \Phi_c(\bld u, \bld v):=&\;
 \sum_{K\in\Th^c}\sum_{i=M+1}^{M+N}\int_{K\cap \Gamma_{i,h}}\epsilon_ik_i
   (\bld K_m^{-1}\bld u)_{t,i}
   \cdot 
      (\bld K_m^{-1}\bld v)_{t, i}
%   - (\bld K_m^{-1}\bld u\cdot\bld n_i)\bld n_i\Big)
%  \Big(\bld K_m^{-1}\bld u - (\bld K_m^{-1}\bld u\cdot\bld n_i)\bld n_i\Big)\cdot\Big(\bld K_m^{-1}\bld v-(\bld K_m^{-1}\bld v\cdot\bld n_i)\bld n_i\Big)
 \,\mathrm{ds},
\end{align}
where $(\bld w)_{t,i}:=\bld w- (\bld w\cdot\bld n_i)\bld n_i$
denotes the tangential component of a vector $\bld w$ on $\Gamma_{i,h}$, 
and the numerical flux 
$\widehat{\bld u}_h\cdot\bld n$ takes the following form:
\begin{align}
    \label{flux}
    \widehat{\bld u}_h\cdot\bld n:=\bld u_h\cdot\bld n
    +\alpha_h (p_h-\widehat p_h),
\end{align}
with $\alpha_h>0$ being the stabilization function defined element-wise as follows:
\begin{align}
\label{stab}
    \alpha_h|_K=\left\{
    \begin{tabular}{ll}
    $\bld K_m$ & if $K\in \Th^r$,\\[.5ex]
    $C_b(h_K/L)^{s_b}\bld K_m$ & if $K\in \Th^r$,\\[.5ex]
    $C_c\,(h_K/L)^{-s_c}\bld K_m$ & if $K\in \Th^c$,
    % $(h_K/L)^{-s}\max_{i\in I_K}\epsilon_i k_i$ & if $K\in \Th^c$
    % with $I_K=\{i:\; K\cap \Gamma_{i,h}\not=\emptyset\}$,
    \end{tabular}
    \right.
\end{align}
where $h_K$ is the local mesh size, 
$L$ is the characteristic length of the domain $\Omega$, and $C_b,C_c>0$ and $s_b,s_c>0$ are penalty parameters to be tuned.
We note that proper tuning of these penalty parameters 
is important for the accuracy of the scheme \eqref{hdg}; see 
Remark~\ref{rk-stab} below.
\end{subequations}

\begin{remark}[Connection with LDG-H for regular porous media flow]
In the absence of fractures ($\Phi_b=\Phi_c=0$), we have 
$\widetilde{\bld u}_h=\bld u_h$ for the scheme \eqref{hdg}.
Hence, the scheme \eqref{hdg} reduces to the so-called LDG-H scheme introduced in \cite{CGL09} and analyzed in \cite{CGS10}. In particular, the LDG-H scheme produces an optimal $L^2$-convergence rate of order $h^{k+1}$ for the velocity approximation, and a superconvergent $L^2$-convergence rate of order $h^{k+2}$ (under the usual full $H^2$-elliptic regularity assumption) for a special projection error of the pressure, from which a superconvergent postprocessed pressure approximation $p_h^*\in W_h^{k+1}$ can be constructed that satisfies
\begin{subequations}
\label{postprocess}
\begin{alignat}{2}
    (\nabla p_h^*, \nabla q_h)_{K} = &\; -(\bld K_m^{-1}\widetilde{\bld u}_h, \nabla q_h), \quad 
    \forall q_h\in W_h^{k+1}, \quad &&\forall K\in \Th,\\
    (p_h^*, 1)_{K} = &\; (p_h, 1), \quad &&\forall K\in \Th.
\end{alignat}
\end{subequations}

Scheme \eqref{hdg} is a simple modification of the classical LDG-H scheme by adding the fracture surface integrals $\Phi_b$ and $\Phi_c$
and adjusting the stabilization parameter $\alpha_h$ on conductive and blocking fractured cell. 
Hence, minimal amount of work is needed to convert a regular porous media flow HDG solver to a fractured porous media flow solver on unfitted meshes. This has to be contrasted with other fractured porous media flow models that model lower dimensional fractured flows where significant code re-design is needed (and are mostly restricted to geometrically fitted meshes); see, e.g. \cite{Alboin99,Boon2018, FuYang22}.
Despite the simplicity of the scheme \eqref{hdg}, its performance for various 2D and 3D benchmark tests reveal that it is also highly accurate. 
\end{remark}

% Regardless of the choice of $s$ in the stabilization parameter, 
We have the following well-posedness of the HDG scheme \eqref{hdg}.
\begin{theorem}[Well-posedness]
The solution to the HDG scheme \eqref{hdg} exists and is unique.
\end{theorem}
\begin{proof}
Taking test functions 
$(\bld v_h, q_h, \widehat q_h):=
(\bld u_h, p_h, \widehat p_h)$
in \eqref{hdg-2}--\eqref{hdg-4} and adding, we obtain the following identity:
\begin{align}
    \label{ener}
      (\bld K_m^{-1}\widetilde{\bld u}_h,{\bld  u}_h)_{\Th}
      + \langle\alpha_h(p_h-\widehat p_h), 
      (p_h-\widehat p_h)\rangle_{\partial \Th}= (f, \bld u_h)_{\Th}.
\end{align} 
% {\color{red}Need $\Phi_b$ and $\Phi_c$ in (6) or add up over $\mathcal{T}^r$}
Taking $\widetilde{\bld v}_h$ to be supported on 
a cell $K\in\Th$ and using the definition of $\Phi_b$
and $\Phi_c$, we get 
\begin{alignat*}{2}
    (\bld K_m^{-1}{\bld u}_h,\widetilde{\bld  v}_h)_{K}=&\,(\bld K_m^{-1}\widetilde{\bld u}_h,\widetilde{\bld  v}_h)_{K}
    ,&&\quad \forall K\in\Th^r,\\
    (\bld K_m^{-1}{\bld u}_h,\widetilde{\bld  v}_h)_{K}
    +\Phi_{b,K}({\bld u}_h,\widetilde{\bld  v}_h)
    =&\,(\bld K_m^{-1}\widetilde{\bld u}_h,\widetilde{\bld  v}_h)_{K},&&\quad \forall K\in\Th^b,\\
        (\bld K_m^{-1}{\bld u}_h,\widetilde{\bld  v}_h)_{K}
    =&\,(\bld K_m^{-1}\widetilde{\bld u}_h,\widetilde{\bld  v}_h)_{K}
    +\Phi_{c,K}(\widetilde{\bld u}_h,\widetilde{\bld  v}_h)
    ,&&\quad \forall K\in\Th^c,
\end{alignat*}
where $\Phi_{b,K}$ and $\Phi_{c,K}$ are restrictions of 
$\Phi_b$ and $\Phi_c$ on the respective cell $K$.
Taking $\widetilde{\bld v}_h=\bld u_h$ on $K\in \Th^r\cup \Th^b$, and taking 
$\widetilde{\bld v}_h=\widetilde{\bld u}_h$ on $K\in \Th^c$, we get 
\begin{alignat*}{2}
    (\bld K_m^{-1}{\bld u}_h,\widetilde{\bld u}_h)_K
    =\left\{ 
    \begin{tabular}{ll}
    $(\bld K_m^{-1}{\bld u}_h,{\bld u}_h)_K$&  if $K\in \Th^r$,
    \\[.5ex]
$(\bld K_m^{-1}{\bld u}_h,{\bld u}_h)_K
+\Phi_{b,K}({\bld u}_h,{\bld u}_h)$&  if $K\in \Th^b$,\\[.5ex]
$(\bld K_m^{-1}\widetilde{\bld u}_h,\widetilde{\bld u}_h)_K
+\Phi_{c,K}(\widetilde{\bld u}_h,\widetilde{\bld u}_h)$&  if $K\in \Th^c$.
    \end{tabular}\right.
    \end{alignat*}
Combining the above equalities with identity \eqref{ener}, we yield
\begin{align*}
%     (\bld K_m^{-1}{\bld u}_h,{\bld u}_h)_{\Th^r\cup \Th^b}
% +(\bld K_m^{-1}\widetilde{\bld u}_h,\widetilde{\bld u}_h)_{\Th^c}
% +\Phi_{b}({\bld u}_h,{\bld u}_h)
% +\Phi_{c}(\widetilde{\bld u}_h,\widetilde{\bld u}_h)&\\
E(\bld u_h, \widetilde{\bld u}_h)
      + \langle\alpha_h(p_h-\widehat p_h), 
      (p_h-\widehat p_h)\rangle_{\partial \Th}&\;= (f, \bld u_h)_{\Th},
\end{align*}
where 
\[
E(\bld u_h, \widetilde{\bld u}_h):= 
    (\bld K_m^{-1}{\bld u}_h,{\bld u}_h)_{\Th^r\cup \Th^b}
+(\bld K_m^{-1}\widetilde{\bld u}_h,\widetilde{\bld u}_h)_{\Th^c}
+\Phi_{b}({\bld u}_h,{\bld u}_h)
+\Phi_{c}(\widetilde{\bld u}_h,\widetilde{\bld u}_h)
\]
is non-negative.

Now let us establish uniqueness of the solution using the above energy identity. 
Taking $f=0$, we have 
\[
E(\bld u_h, \widetilde{\bld u}_h) = 0, \text{ and }
\langle\alpha_h(p_h-\widehat p_h), 
      (p_h-\widehat p_h)\rangle_{\partial \Th}=0,
\]
which implies that
$p_h=\widehat{p}_h$ on $\partial\Th$, 
$\bld u_h=0$ on $\Th^r\cup \Th^b$, 
and 
$\widetilde{\bld u}_h=0$ on $\Th^c$. 
Equation \eqref{hdg-1} then implies that 
$\widetilde{\bld u}_h = \bld u_h=0$ on all cells.
Equation \eqref{hdg-2} and $p_h=\widehat{p}_h$ 
then implies that $\nabla p_h=0$ on all cells. Hence $p_h$
is a global constant. Using the homogeneous Dirichlet boundary condition on $\widehat{p}_h$, we conclude that 
$p_h=0$ and $\widehat{p}_h=0$. Hence we proved the uniqueness.
Existence of the solution is a direct consequence of uniqueness as the system \eqref{hdg} is a square linear system.
\end{proof}

\begin{remark}[On the stabilization function]
\label{rk-stab}
In practice (see the numerics section below), we found taking $s_c\ge 3$ in \eqref{stab} leads to a convergent scheme for polynomial degree $k\ge1$, while taking $s_c<3$ may produce large consistency errors on conductive fractures. 
This indicates the stabilization needs to be very large on  conductive fractures.  
The reason for such large choice of stabilization is to implicitly enforce the pressure continuity across the boundary of conductive fractured cells as the physical model has such continuity while the surface integral 
$\Phi_c$ on conductive fractures itself does not enforce such pressure continuity. 
Equally well, the choice of smaller stabilization (with $s_b>0$) in \eqref{stab}
on blocking fractured cells has the effect of enforcing the velocity normal continuity across 
the boundary of blocking fractured cells to be consistent with the physical model. 
Here we found that when $\epsilon_i/k_i\approx \bld K_m^{-1}$ for blocking fractures, we can simply use the
same stabilization on blocking cells as those on the regular cells, i.e., with $C_b=1$ and $s_b=0$; see Examples 1 \& 2 in Section~\ref{sec:numerics}.
On the other hand, when 
$\epsilon_i/k_i\gg \bld K_m^{-1}$, we need to reduce the blocking cell stabilization to enforce the normal velocity continuity, where 
taking $C_b=1$ and $s_b=2$ usually gives a good result; see Examples 3 \& 6 in Section~\ref{sec:numerics}.

Here we provide another heuristic argument to justify the large stabilization on conductive fractures. We assume permeability $\bld K_m$ is a constant on each cell $K$ in the following discussion.
Taking $\alpha_h|_K\rightarrow \infty$ on conductive fractures, we have  $p_h\approx \widehat{p}_h$
on $\partial K$ for all $K\in \Th^c$ under the reasonable assumption that the numerical flux 
$\widehat{\bld u}_h\cdot\bld n$ stays bounded.
Then equation \eqref{hdg-2} implies that 
$\widetilde{\bld u}_h\approx -\bld K_m\nabla\bld p_h$.
Taking $\widetilde{\bld v}_h:=\bld K_m\nabla\bld q_h$ on 
$\Th^c$ in \eqref{hdg-1}
with $q_h\in W_h$ that is continuous across interior facets of $\Th^c$, we have 
\begin{align*}
    -(\bld u_h, \nabla q_h)_{\Th^c}
= &\;
    -(\widetilde{\bld u}_h, \nabla q_h)_{\Th^c}
    -
        \Phi_c(\widetilde{\bld u}_h, \bld K_m\nabla q_h)\\
        \approx &\;
            (\bld K_m\nabla p_h, \nabla q_h)_{\Th^c}
    +
        \Phi_c(\bld K_m\nabla p_h, \bld K_m\nabla q_h),
\end{align*}
Combine the above relation with \eqref{hdg-3} and using the fact that $q_h$ is continuous across interior facets of $\Th^c$, we get 
\begin{align*}
      (\bld K_m\nabla p_h, \nabla q_h)_{\Th^c}
    +
        \Phi_c(\bld K_m\nabla p_h, \bld K_m\nabla q_h)
        \approx 
        (f, q_h)_{\Th^c}-
        \langle\widehat{\bld u}_h\cdot\bld n, q_h
        \rangle_{\Gamma_h^c},
\end{align*}
where $\Gamma_h^c$ is the boundary facets of $\Th^c$.
The above relation show that $p_h$ will be a good approximation 
to the $H^1$-conforming finite element discretization 
of the RDFM model on conductive fractures \cite{Yang2021}, which was known to provide a consist approximation with respect to the conductive fractures as long as $W_h$ contains at least piecewise linear functions. 

The lowest-order case with $k=0$ requires further attention. It is more subtle to find a good set of penalty parameters on conductive fractured cells for $k=0$. If it is taking to be too large, the strong penalty will effectively makes pressure along fractures be a global constant, leading to large consistency errors. On the other hand, if stabilization is taking to be too small, the effects of conductive fractures will not be seen by the scheme. Our numerical experiments below suggests that taking $s=2$ for $k=0$ may lead to reasonable approximations.

We will investigate more on the effects of the stabilization function on the HDG scheme in our future work.
\end{remark}

\begin{remark}[Hybrid-mixed methods]
\label{rk:hm}
We can increase the velocity space $\Vh^k$
to be a discontinuous Raviart-Thomas space of degree $k$:
\[
 \Vh^{RT,k} :=\; \{\bld v\in [L^2(\Th)]^d:\;
   \bld v|_K\in [P_k(K)]^d\oplus 
   \bld x\widetilde{P}_k(K),\quad \forall K\in\Th\},
 \]
 where $\widetilde{P}_k(K)$ is the space of homogeneous polynomials of degree $k$.
Then the velocity-pressure pair $\Vh^{RT,k}$
and $\Wh^k$ satisfy the inf-sup condition, and we can set the stabilization $\alpha_h$
on $\Th^r\cup \Th^b$ to be zero when using 
$\Vh^{RT,k}$, $\Wh^k$, and $\Mh^k$ in the scheme \eqref{hdg} (while still keep the large stabilization on conductive fractures). 
The resulting scheme is the hybrid-mixed method, whose computational cost is similar to the HDG scheme \eqref{hdg}. 
% and is more accurate in the lowest order case when $k=0$. The accuracy of the HDG scheme and the hybrid-mixed method are expected to be comparable when $k\ge 1$. 
We note that the lowest-order hybrid-mixed method for porous media containing pure blocking fractures was already introduced in our earlier work \cite{FuYang22}. 
\end{remark}

\begin{remark}[Variable polynomial degree on different cells]
\label{rk-variable}
% From Remark \ref{rk-stab}, we find that 
% the lowest order HDG scheme \eqref{hdg} for problems contain conductive fractures use polynomial degree $k=1$ as the $k=0$ version is not convergent. 
% Moreover, 
The large stabilization on conductive fractures may lead to less accurate velocity approximations therein since 
$\bld u_h\approx -\bld K_m^{-1}\nabla p_h$.
An interesting variant of the scheme is to use one degree higher on conductive fractured cells $\Th^c$
than those on regular and  blocking fractured cells, that is, replacing the spaces $\Vh^k, \Wh^k$ and $\Mh^k$ in \eqref{hdg} by the following reduced version:
\begin{subequations}
\label{space-var}
\begin{align}
\Vh^{k-1,k}:=&\; [\Wh^{k-1,k}]^d,\\
   \Wh^{k-1,k} :=&\; 
   \{w\in \Wh^{k}:\;
   w|_K\in P_{k-1}(K),\quad \forall K\in\Th^r\cup \Th^b\},
   \\
   \Mh^{k-1,k} :=&\; \{\mu\in \Mh^{k}:\;
   \mu|_F\in P_{k-1}(F),\quad \forall F\in\Eh \text{ with }F\cap 
   \partial\Th^c = \emptyset\}.
\end{align}
This reduced version is less accurate than the original version but is cheaper to solve as it has less DOFs.
\end{subequations}
\end{remark}
In this work, we present numerical results only for the original HDG scheme \eqref{hdg} with polynomial degree $k=0,1,2$.
We will explore the performance of the above mentioned hybrid-mixed method and variable-degree variants in our future work. 
% More over, the lowest-order version of this variable degree scheme may be 
% We will report some numerical results using this reduced version.

\begin{remark}[Static condensation and efficient implementation]
\label{rk-condense}
Just like the LDG-H scheme for regular porous media flow, we can solve system \eqref{hdg} efficiently using static condensation where one first locally eliminates 
the cell-wise DOFs to express the unknowns $\widetilde{\bld u}_h, \bld u_h, p_h$ as (local) functions of the global unknown $\widehat{p}_h$ and source term $f$ using \eqref{hdg-1}--\eqref{hdg-3}, and then 
solve the global transmission problem \eqref{hdg-4}
for $\widehat{p}_h$, which is a sparse and symmetric positive definite linear system whose efficient solution procedure can be designed following similar work for regular porous media flows; see, e.g., \cite{CDG14, Fu21a}. 
\end{remark}

\begin{remark}[Local mesh refinement near fractures]
\label{rk:refine}
Since the computational mesh $\Th$ is assumed to be completely independent of the fractures, the approximation quality of the scheme \eqref{hdg} on an initial coarse mesh that does not know the fracture locations may be poor. Here we propose to use local mesh refinement 
that only refine cells intersected by the fractures. In particular, given an initial mesh $\Th:= \Th^r\cup \Th^b\cup \Th^c$, we only mark cells in $\Th^b$ and $\Th^c$ for refinement
using the bisection algorithm. Multiple refinements can be performed sequentially as needed. 
This refinement procedure puts more cells around fractures and leads to significantly more efficient algorithms comparing with a na\"ive uniform refinement procedure.  
\end{remark}

\section{Numerics}
\label{sec:numerics}
In this section, we present detailed numerical results 
in two- and three-dimensions for the proposed HDG scheme \eqref{hdg}. When evaluating pressure distribution along line segments, we always evaluate the postprocessed pressure approximation  \eqref{postprocess}.
Our numerical simulations are performed using the open-source finite-element software
{\sf NGSolve} \cite{Schoberl16}, \url{https://ngsolve.org/}.
In particular, the (multi-)level set representation of  the fractures and their associated surface integrations are realized using the {\sf ngsxfem} add-on \cite{JOSS_LHPvW_2021}.
Sample code can be can be found in the git repository
\url{https://github.com/gridfunction/fracturedPorousMedia}.

\subsection*{Example 1: Cross-shaped fractures in 2D}
In this example, we test the performance of scheme \eqref{hdg} for a fractured media with simple cross-shaped fractures. Similar test was used in \cite{pEDFM}.
The computational domain is a unit square $\Omega=[0,1]\times [0, 1]$. 
Two fractures with thickness $\epsilon=10^{-3}$ and length 
$0.5$ are located in teh region given below and cross each other at the center $(0.5,0.5)$: 
\[
\Gamma_1 = \{(x,0.5): 0.25\le x\le 0.75\}, \quad
\Gamma_2 = \{(0.5,y): 0.25\le y\le 0.75\}.
\]
The matrix permeability is $\bld K_m=1$ and the fracture permeability is either (a) $k_1=k_2=10^3$ for the conductive case or (b) $k_1=k_2=10^{-3}$ for the blocking case.
Source term is $f=0$, and
the problem is closed with no flow boundary condition on the top and bottom boundaries and Dirichlet boundary condition 
$p=1$ on the left boundary and $p=0$ on the right boundary. See Figure~\ref{fig:ex1} for an illustration of the setup and
the reference solutions for the two cases. 
Here the reference solutions are obtained using a continuous $Q_1$ finite element scheme on a uniform  $2000\times 2000$ rectangular mesh where the fracture has been fully resolved.
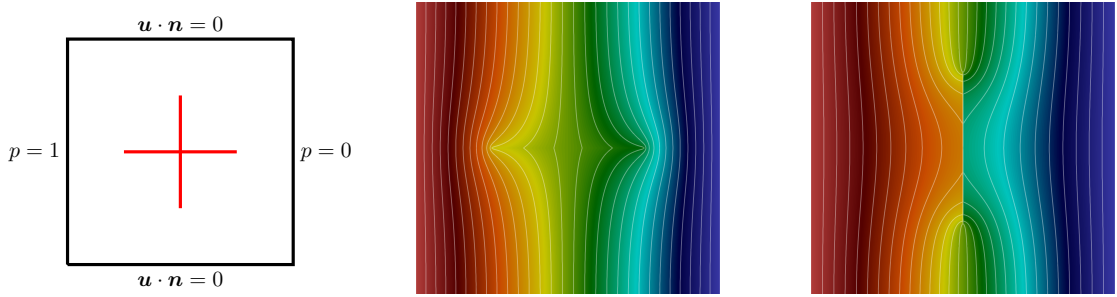
\begin{figure}[ht]
\centering
\scalebox{0.75}{
  \begin{tikzpicture}
    \draw[ultra thick,draw=black]% pattern=north west lines]
    (0, 0)
    to (4, 0)
    to (4, 4)
    to (0, 4)
    to (0, 0)
    ;
  \draw[ultra thick, draw=red]
    (1,2) to (3,2);
  \draw[ultra thick, draw=red]
    (2,1) to (2,3);
\node at (2,4)[above,scale=1] {$\bld u\cdot\bld n=0$};
  \node at (2,0)[below,scale=1] {$\bld u\cdot\bld n=0$};
  \node at (0,2)[left,scale=1] {$p=1$};
  \node at (4,2)[right,scale=1] {$p=0$};
\end{tikzpicture} }    
\includegraphics[width=0.31\textwidth]{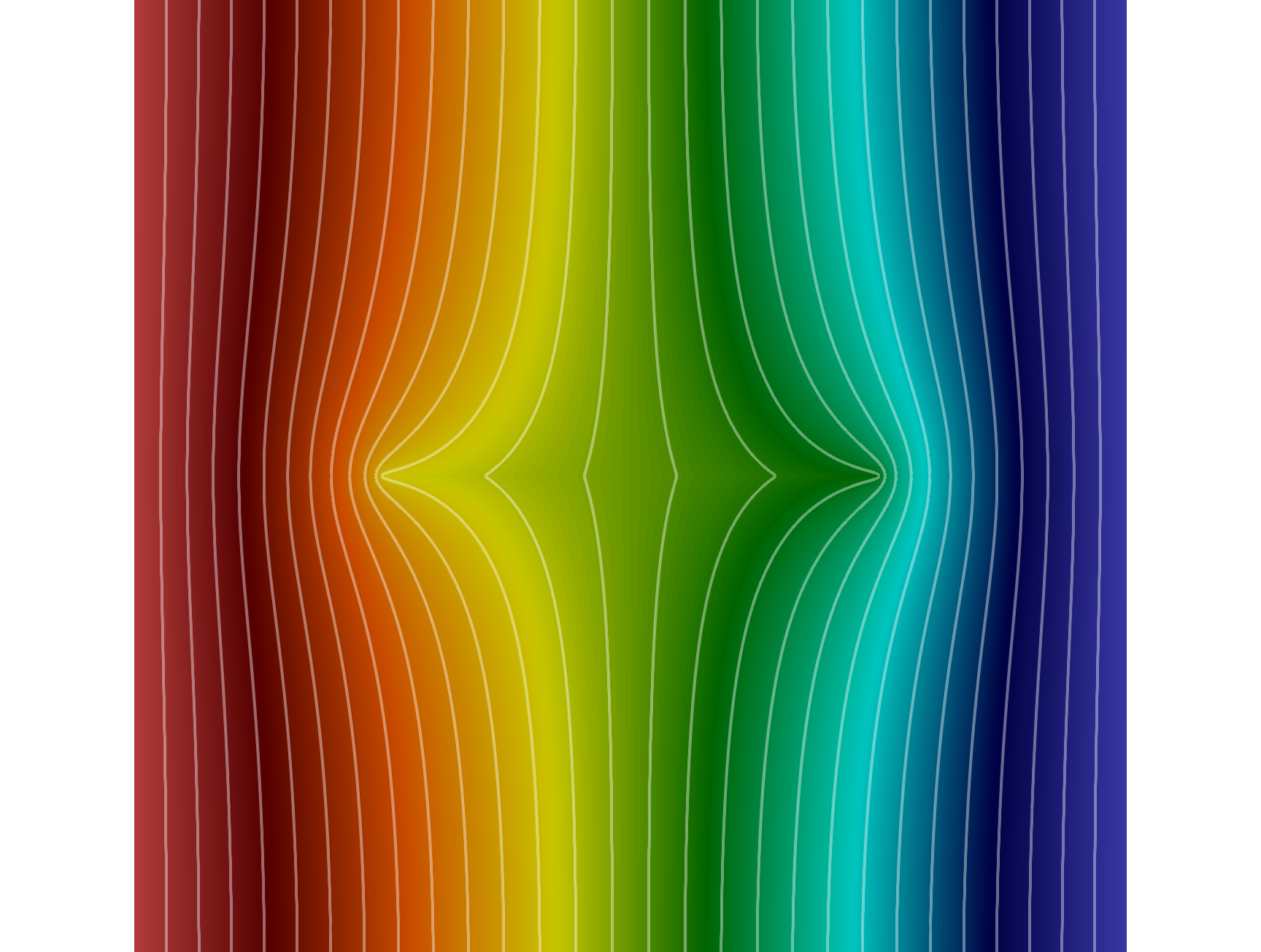}
\includegraphics[width=0.31\textwidth]{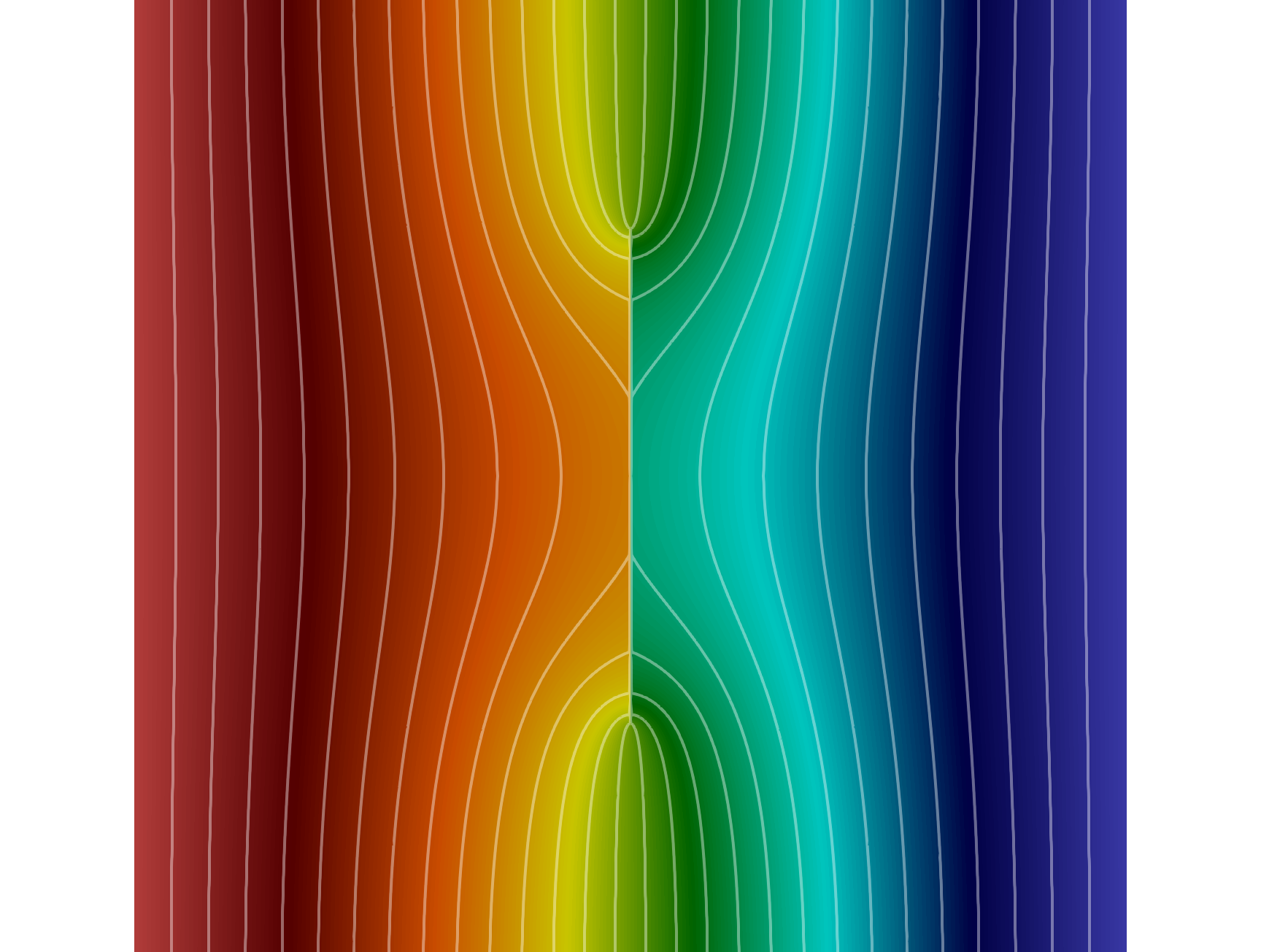}
\caption{\textbf{Example 1.}
Left: Domain and boundary conditions.
Middle: reference solution for conductive fractures (1a).
Right: reference solution for blocking fractures (1b). 
Color range: (0, 1). Thirty uniform contour lines 
from 0 to 1.}
\label{fig:ex1}
\end{figure}

We consider two meshes, see Figure~\ref{fig:ex1-msh}: a coarse unfitted triangular mesh with mesh size $h=0.1$ and a refined unfitted mesh  that performs 3 steps of local mesh refinements near the fractured cells using the procedure detailed in Remark \ref{rk:refine}. 
The coarse mesh has $230$ cells, while the fine mesh has 
$1328$ cells.
\begin{figure}[ht]
\centering
\vspace{1pt}
\includegraphics[width=0.40\textwidth]{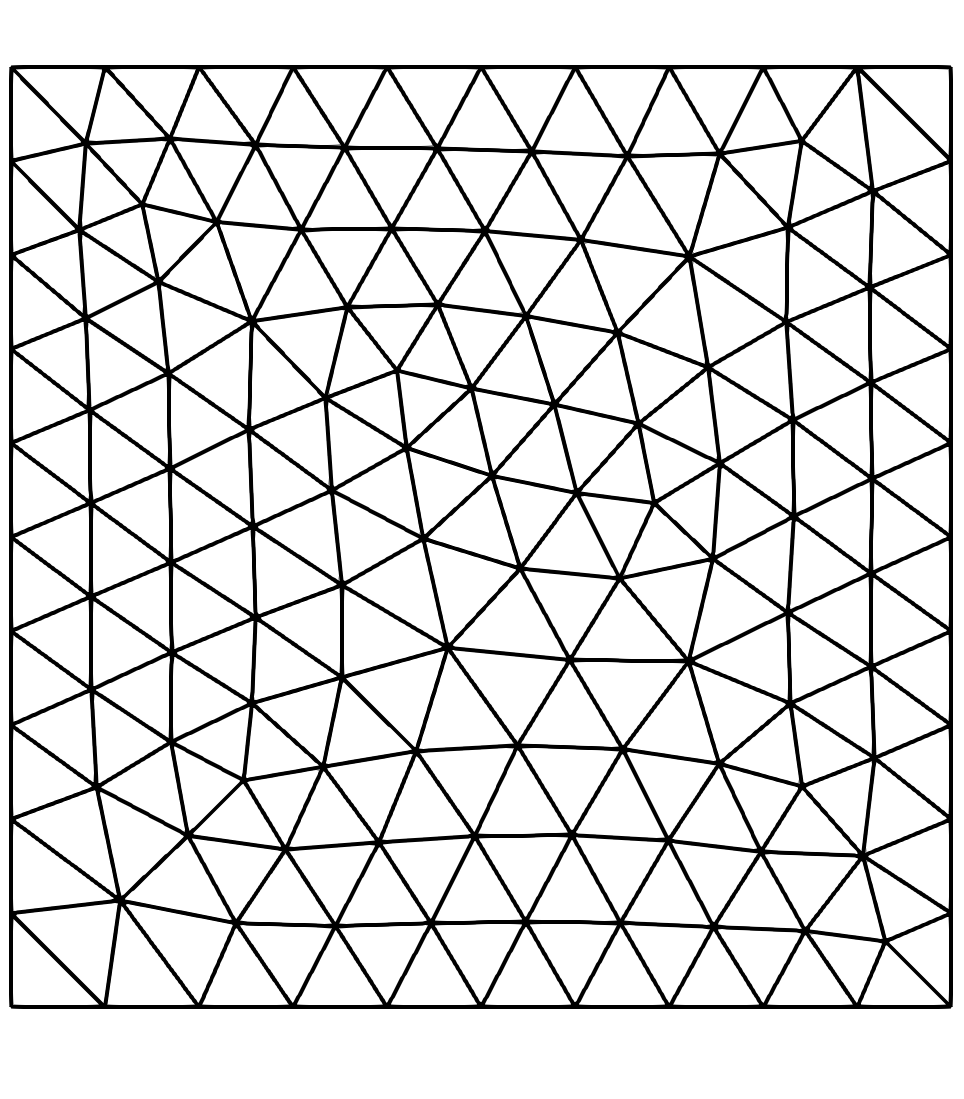}
\includegraphics[width=0.40\textwidth]{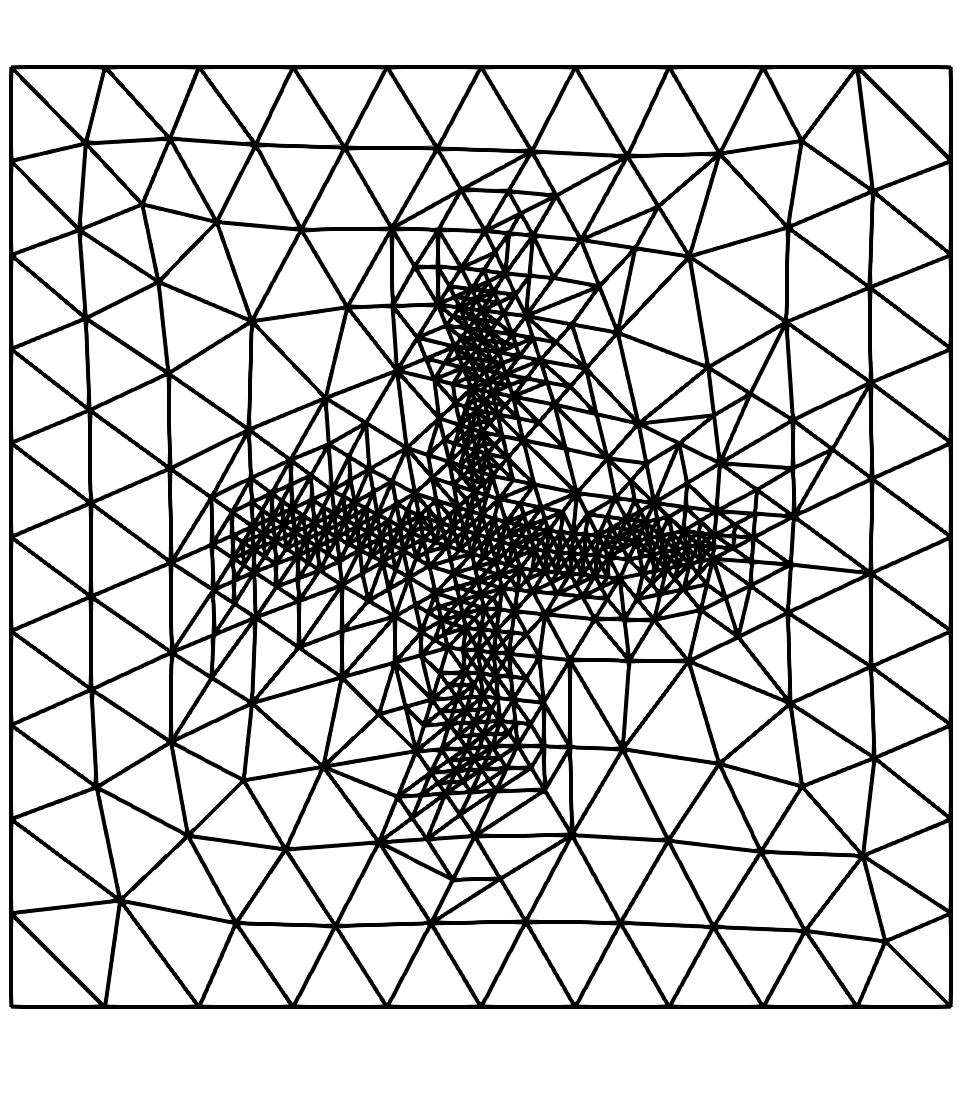}
\vspace{1pt}
\caption{\textbf{Example 1.}
Left: a coarse mesh  with size $h=0.1$.
Right: a locally refined mesh  with $h\approx 0.1/8$
near the fractures.}
\label{fig:ex1-msh}
\end{figure}

We first  study the role of the stabilization function  $\alpha_h$, in particular, the effect of $s_c$ therein, near conductive fractures on the scheme \eqref{hdg}.
We take polynomial degree $k=0,1,2$, 
and vary the scaling power $s_c\in\{1,2,3\}$ of $\alpha_h$ in \eqref{stab} with $C_c=1$.
The pressure approximations along the line $x=0.5$ are shown in Figure~\ref{fig:ex1-stab}.
From these figures, we observe that 
\begin{itemize}
    \item When polynomial degree $k=0$, a convergent result,
    comparing with reference solution, is obtained with $s=2$. For $s=1$ (smaller stabilization) the scheme does not converge as the fracture is not captured. For $s=3$ (larger stabilization) the scheme does not convergence either as it leads to a constant approximation along the fractures, which is consistent with the discussion in Remark \ref{rk-stab} as too large stabilization effectively makes the pressure within conductive fractures to be a global constant for $k=0$.
    \item When polynomial degree $k=1$ or $k=2$, a convergent result  is obtained with $s=3$. The results for  $s=1$
    and $s=2$ leads to locking phenomena as the effects of the fracture is not captured correctly.
    We note that  further increase $s$ from 3 essentially leads to similar results as those with $s=3$ for this example. 
\end{itemize}
\begin{figure}[ht]
\centering
\includegraphics[width=0.45\textwidth]{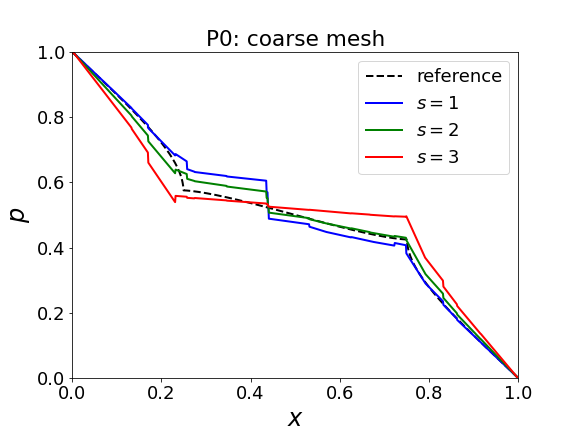}
\includegraphics[width=0.45\textwidth]{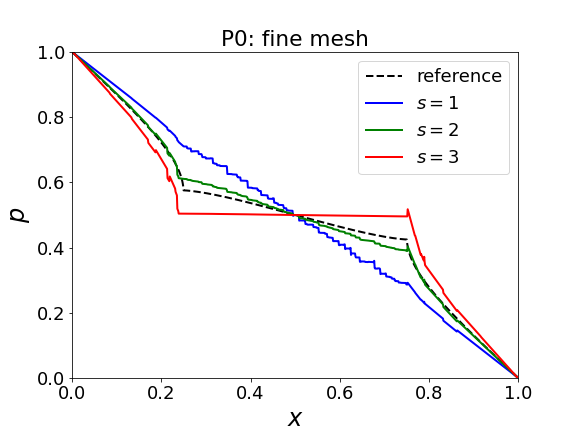}\\
\includegraphics[width=0.45\textwidth]{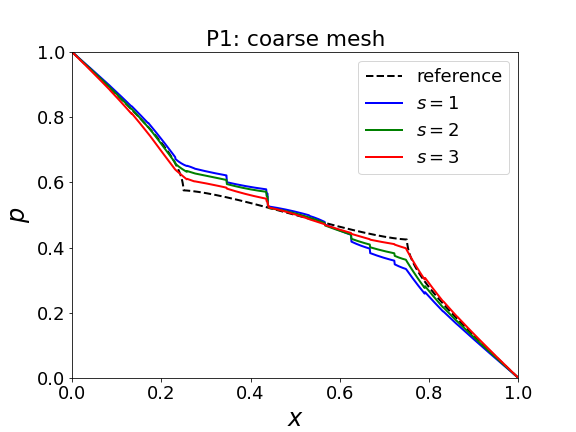}
\includegraphics[width=0.45\textwidth]{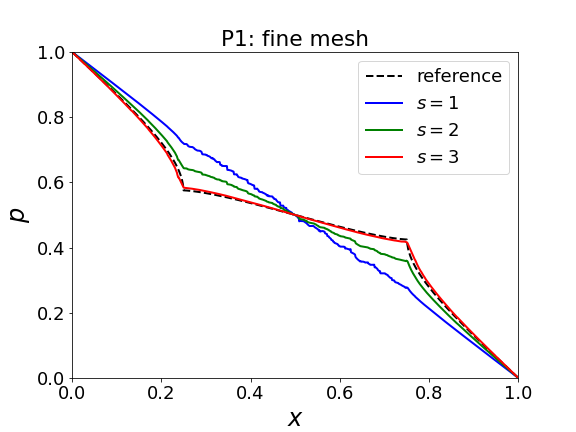}\\
\includegraphics[width=0.45\textwidth]{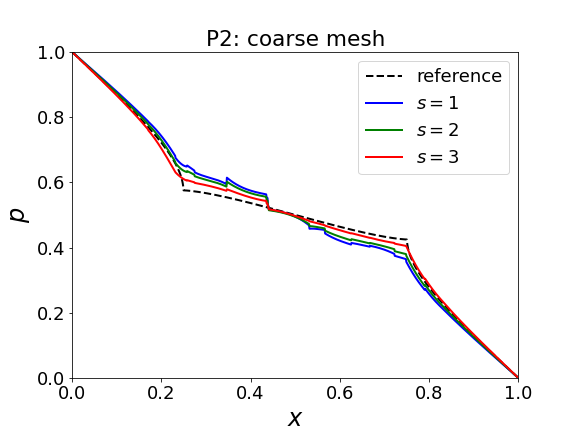}
\includegraphics[width=0.45\textwidth]{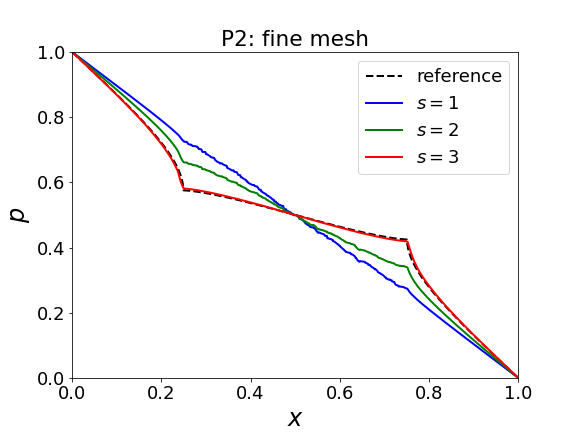}\\
\caption{\textbf{Example 1(a)}: conductive fractures. Pressure along cut line $x=0.5$ for 
the scheme \eqref{hdg} with different stabilization parameter with $s_c=s=1,2,3$ in \eqref{stab}.
% Top row: polynomial degree $k=0$; 
% Middle row: polynomial degree $k=1$; 
% Bottom row: polynomial degree $k=2$.
% Left: coarse mesh; Right: fine mesh.
}
\label{fig:ex1-stab}
\end{figure}
Contour plots of the pressure on the fine mesh for $k=0$ with $s_c=2$ and for $k=1, 2$ with $s_c=3$ are shown in 
Figure~\ref{fig:ex1-cont}. We observe these results are qualitatively similar to the reference solution in the middle of Figure~\ref{fig:ex1-msh}.
\begin{figure}[ht]
\centering
\includegraphics[width=0.31\textwidth]{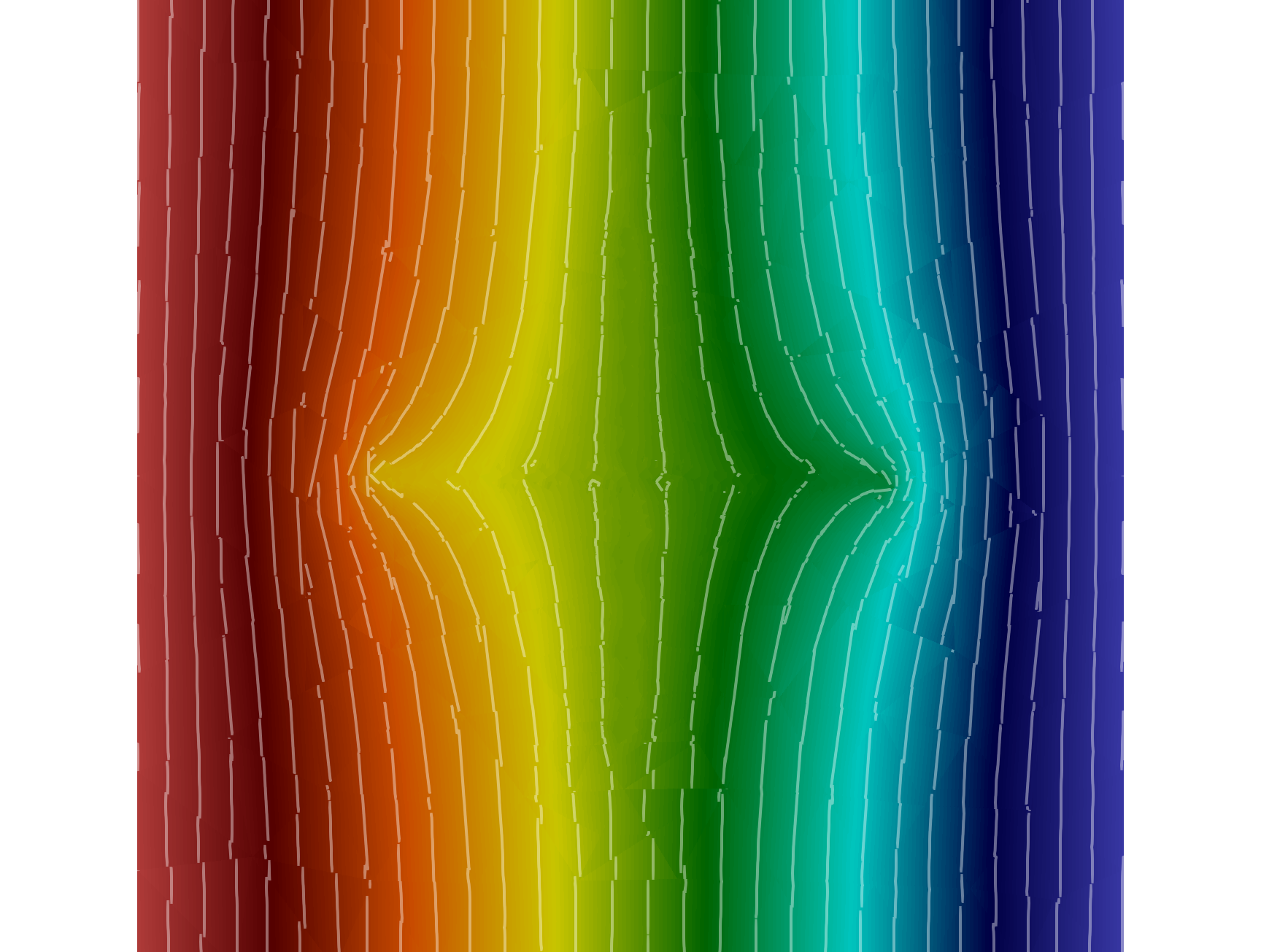}
\includegraphics[width=0.31\textwidth]{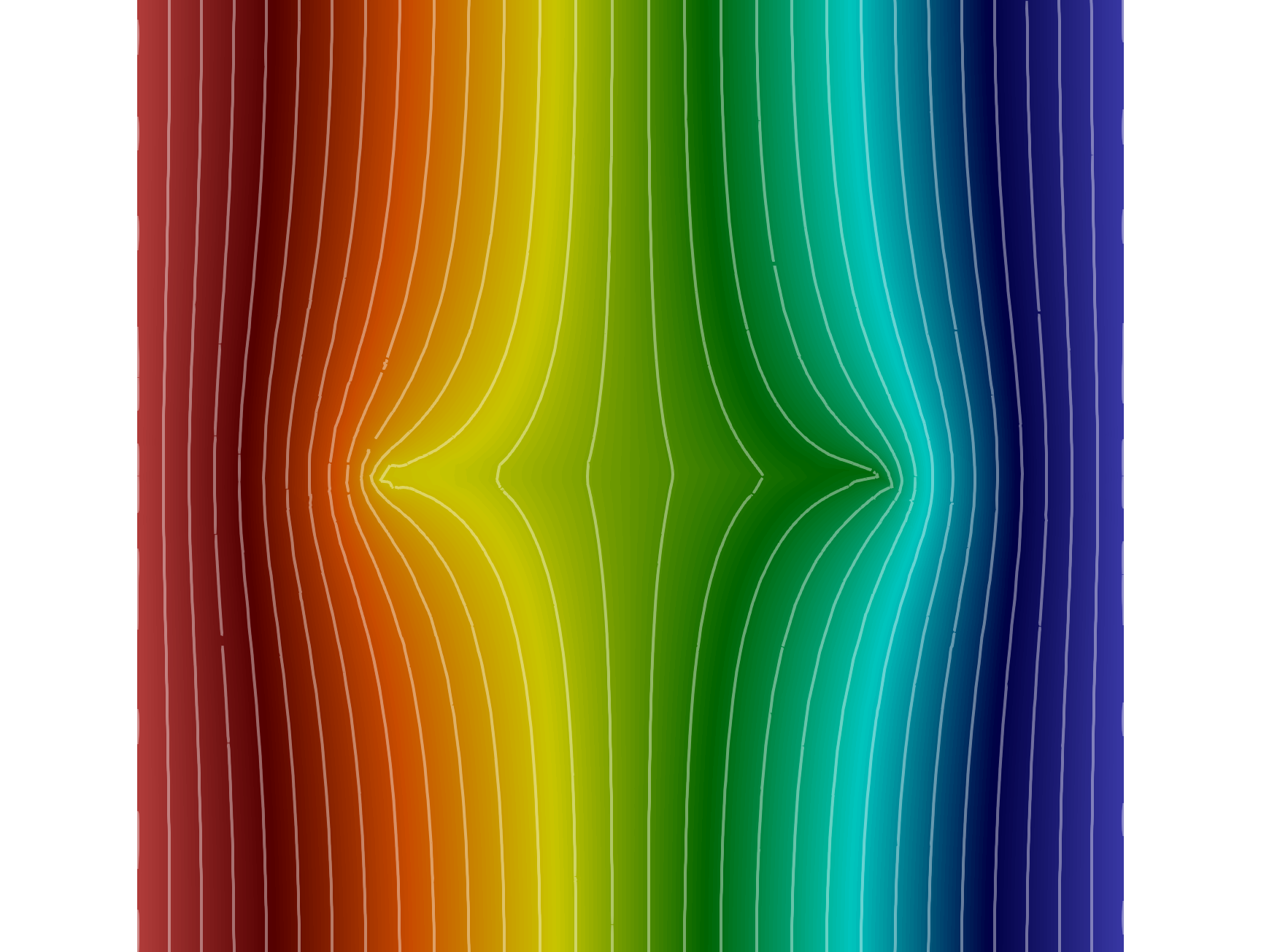}
\includegraphics[width=0.31\textwidth]{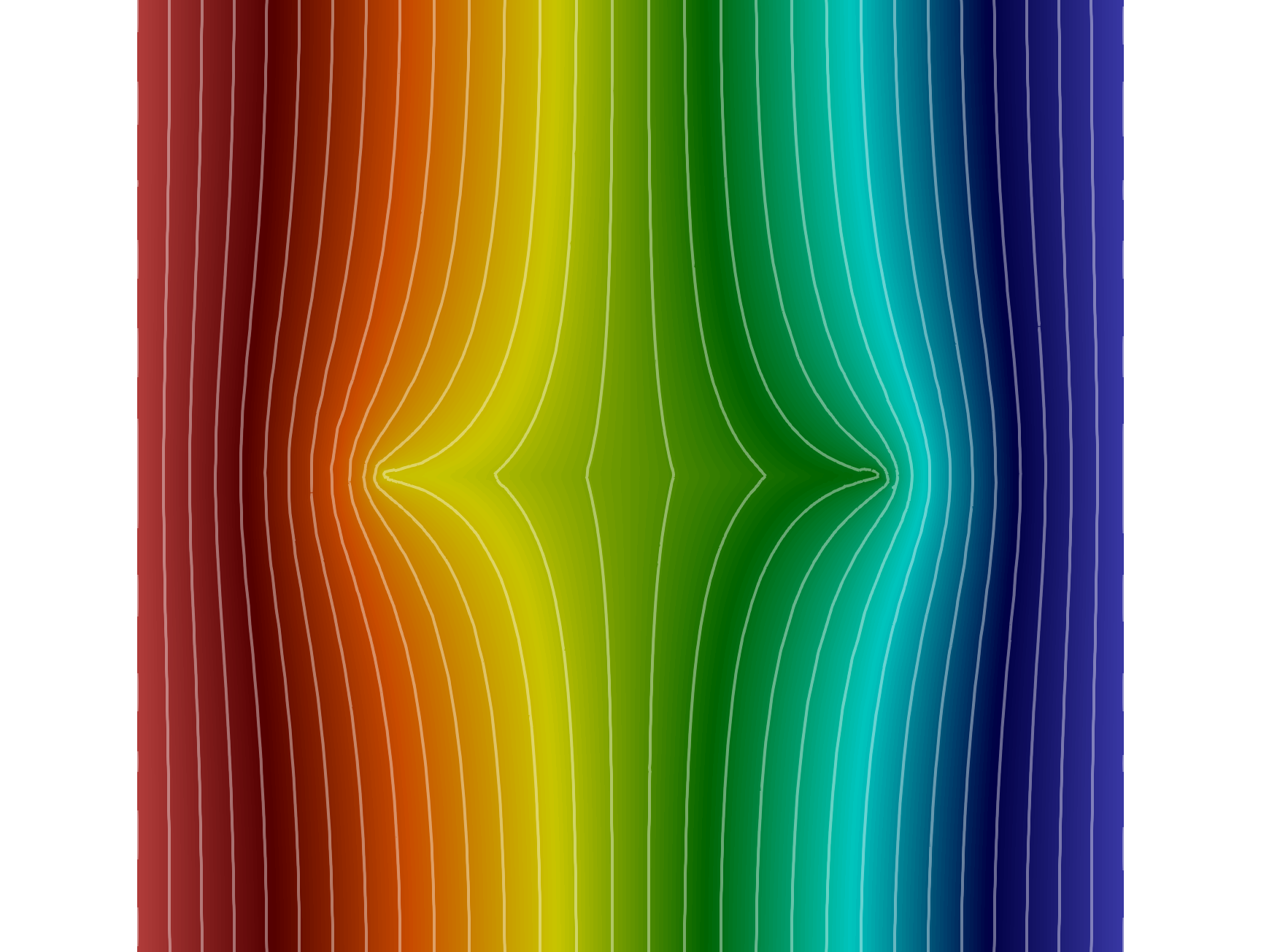}
\caption{\textbf{Example 1(a)}: conductive fractures. Pressure contour for $k=0, s=2$ (left), 
$k=1, s=3$ (middle), and $k=2, s=3$ (right).
}
\label{fig:ex1-cont}
\end{figure}

We next study the performance of our scheme for the blocking fracture case. 
For this case, we find that taking 
the penalty parameter on blocking fractured cells to be the same as regular cells (i.e., $C_b=1$ and $s_b=0$) already lead to a convergent scheme, so we report results for this choice of parameters. 
The pressure approximations along the line $x=0.5$
for $k=0,1,2$ on the coarse and fine meshes are shown in Figure~\ref{fig:ex1-stab-b}.
Convergence to the reference solution is observed for all cases as mesh refines.
\begin{figure}[ht]
\centering
\includegraphics[width=0.32\textwidth]{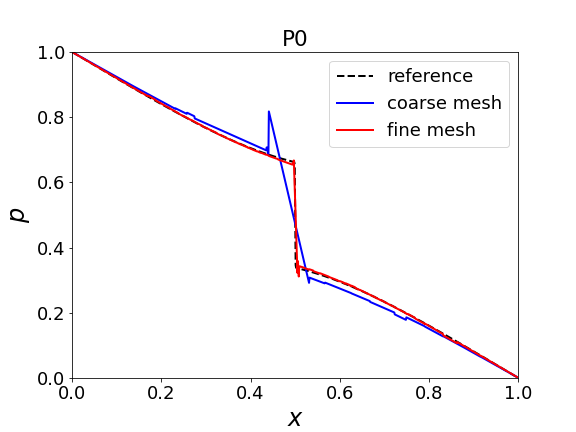}
\includegraphics[width=0.32\textwidth]{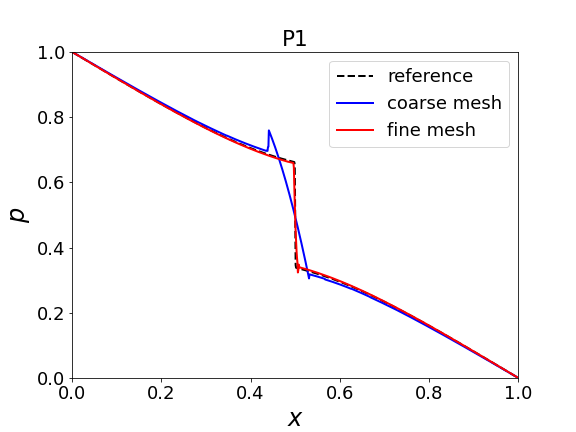}
\includegraphics[width=0.32\textwidth]{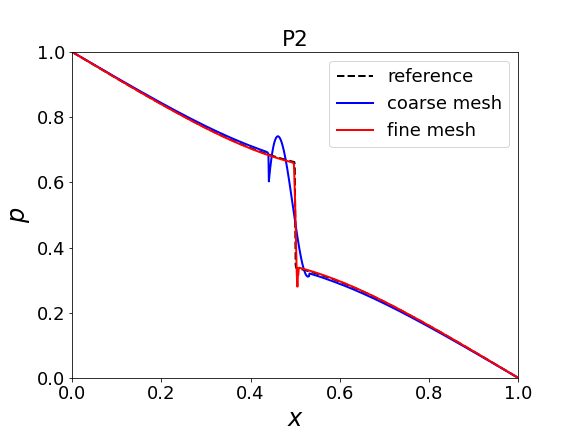}
\caption{\textbf{Example 1(b)}: Pressure along cut line $x=0.5$ for 
the scheme \eqref{hdg} for the blocking fracture case.
}
\label{fig:ex1-stab-b}
\end{figure}
Contour plots of the pressure on the fine mesh are shown in 
Figure~\ref{fig:ex1-contb}. We again observe these results are qualitatively similar to the reference solution in the right of Figure~\ref{fig:ex1-msh}.
\begin{figure}[ht]
\centering
\includegraphics[width=0.31\textwidth]{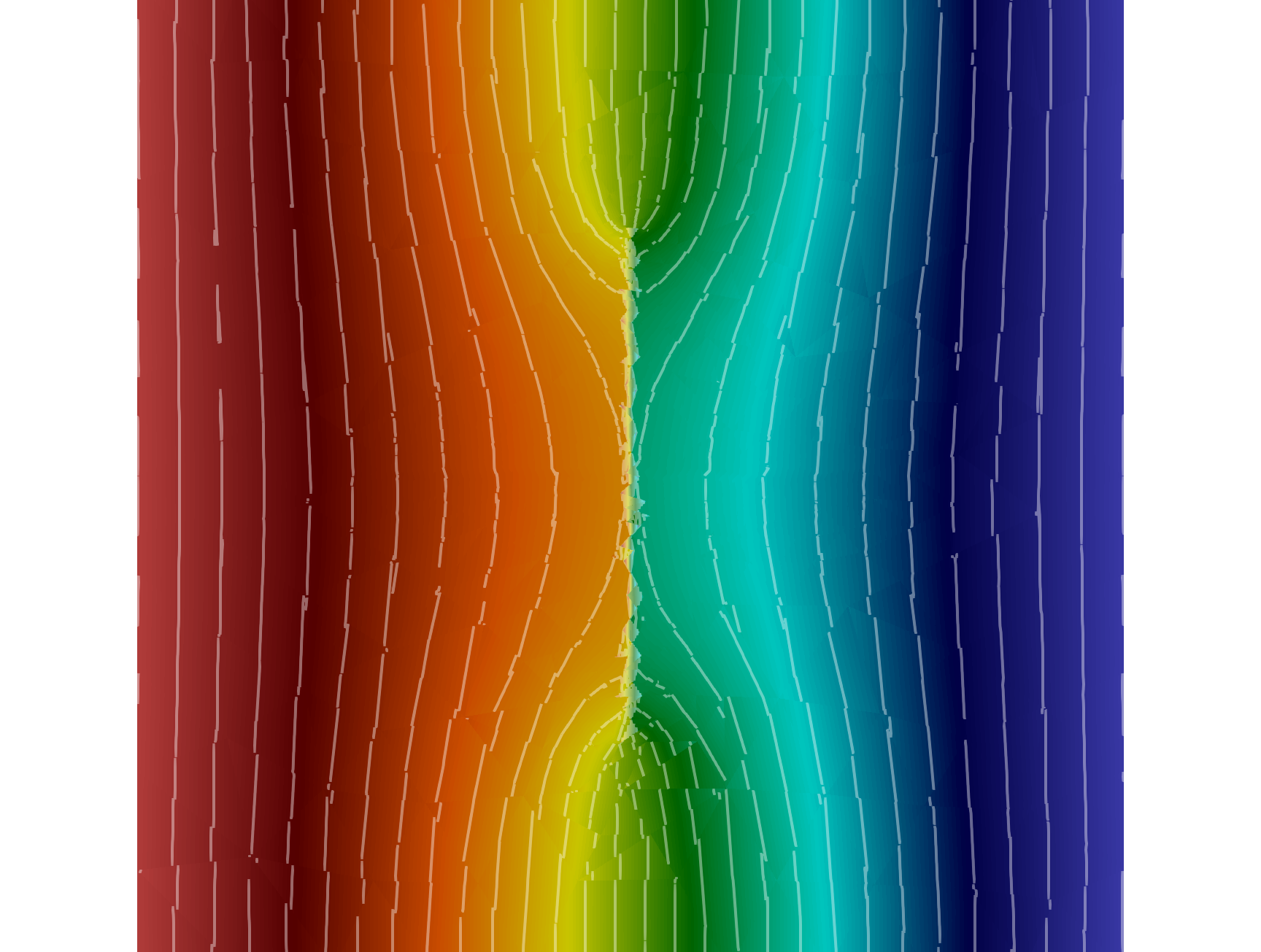}
\includegraphics[width=0.31\textwidth]{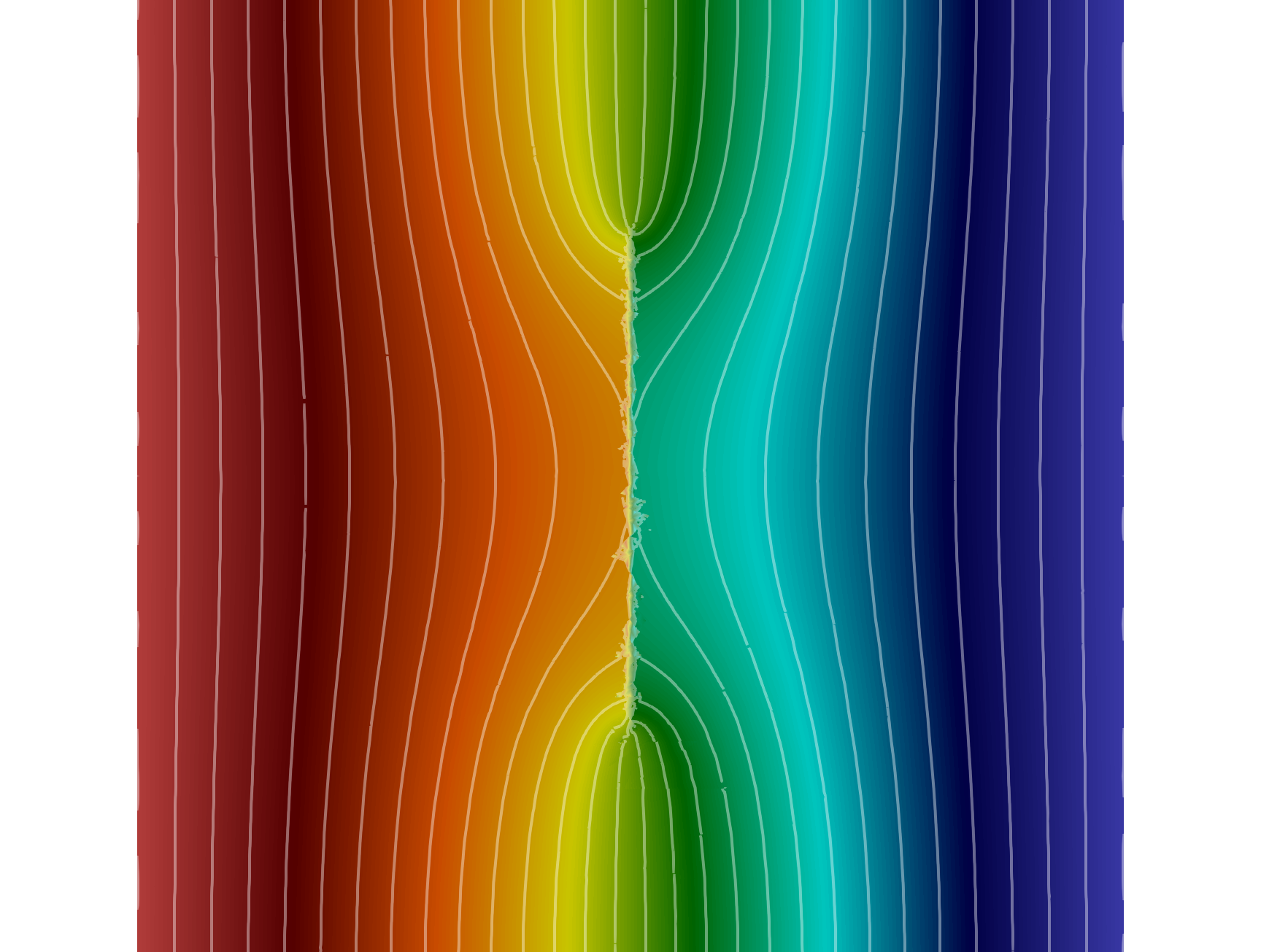}
\includegraphics[width=0.31\textwidth]{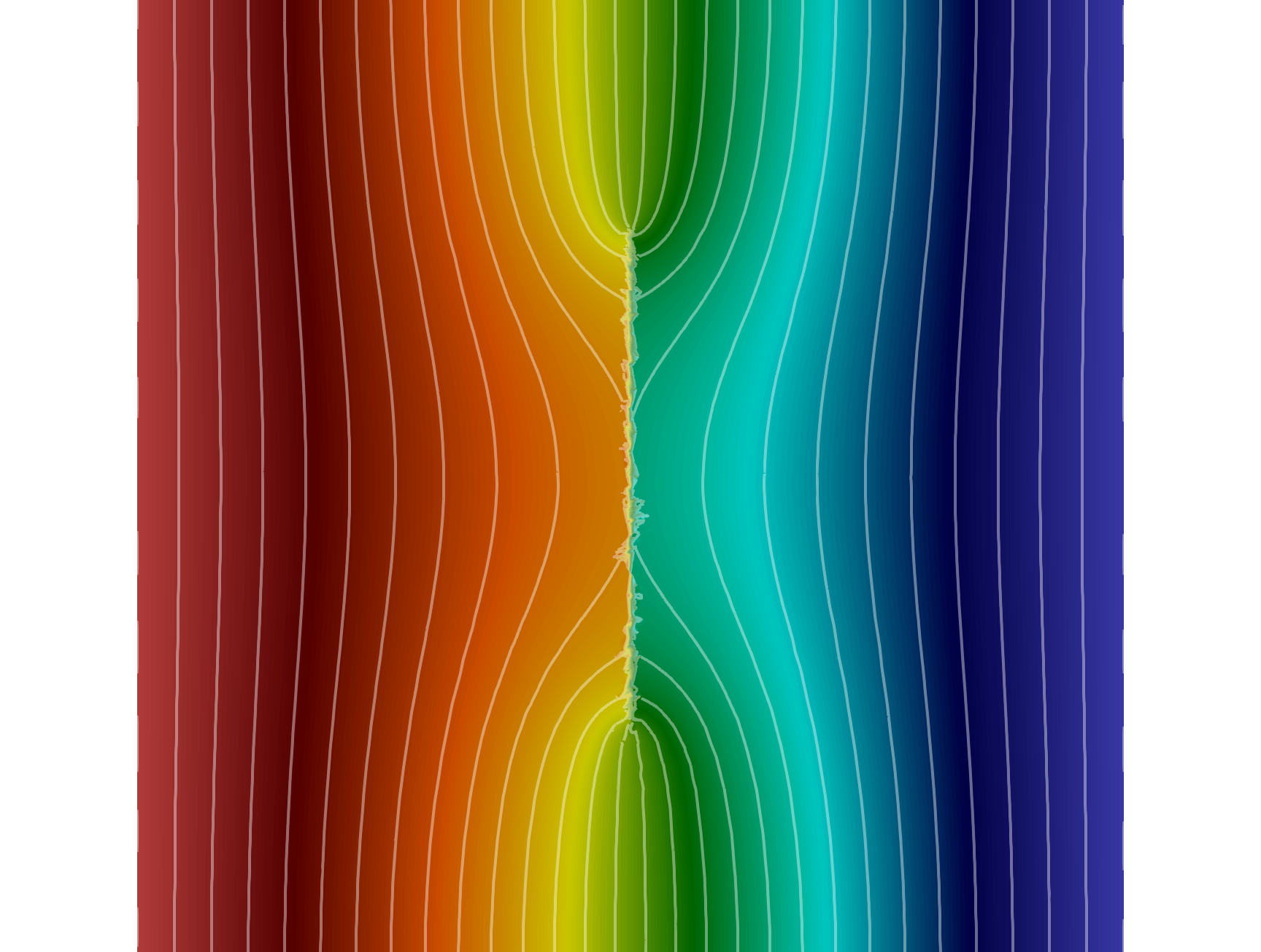}
\caption{\textbf{Example 1(b)}: Pressure contour for $k=0$ (left), 
$k=1$ (middle), and $k=2$ (right).
}
\label{fig:ex1-contb}
\end{figure}

\subsection*{Example 2: Complex fracture network in 2D}
This test case considers a small but complex fracture network that includes eight 
conductive fractures
and two blocking fractures. The domain and boundary conditions are shown in Figure
\ref{fig:complex} (red represents conductive fractures, blue represents blocking fractures). All fractures are represented by line segments, and the exact coordinates for the fracture positions can be found in
\cite[Appendix C]{FLEMISCH2018239}.
% The fracture network contains ten straight immersed fractures. 
The fracture thickness
is $\epsilon=10^{-4}$ for all fractures, and permeability is
$k_i=10^4$ for all fractures except
for fractures 4 and 5 which are blocking fractures with $k_i=10^{-4}$ .
We consider two subcases (a) and (b) with a pressure gradient which is predominantly vertical and horizontal respectively.
\begin{figure}[ht]
  \centering
    \includegraphics[width=.86\textwidth]{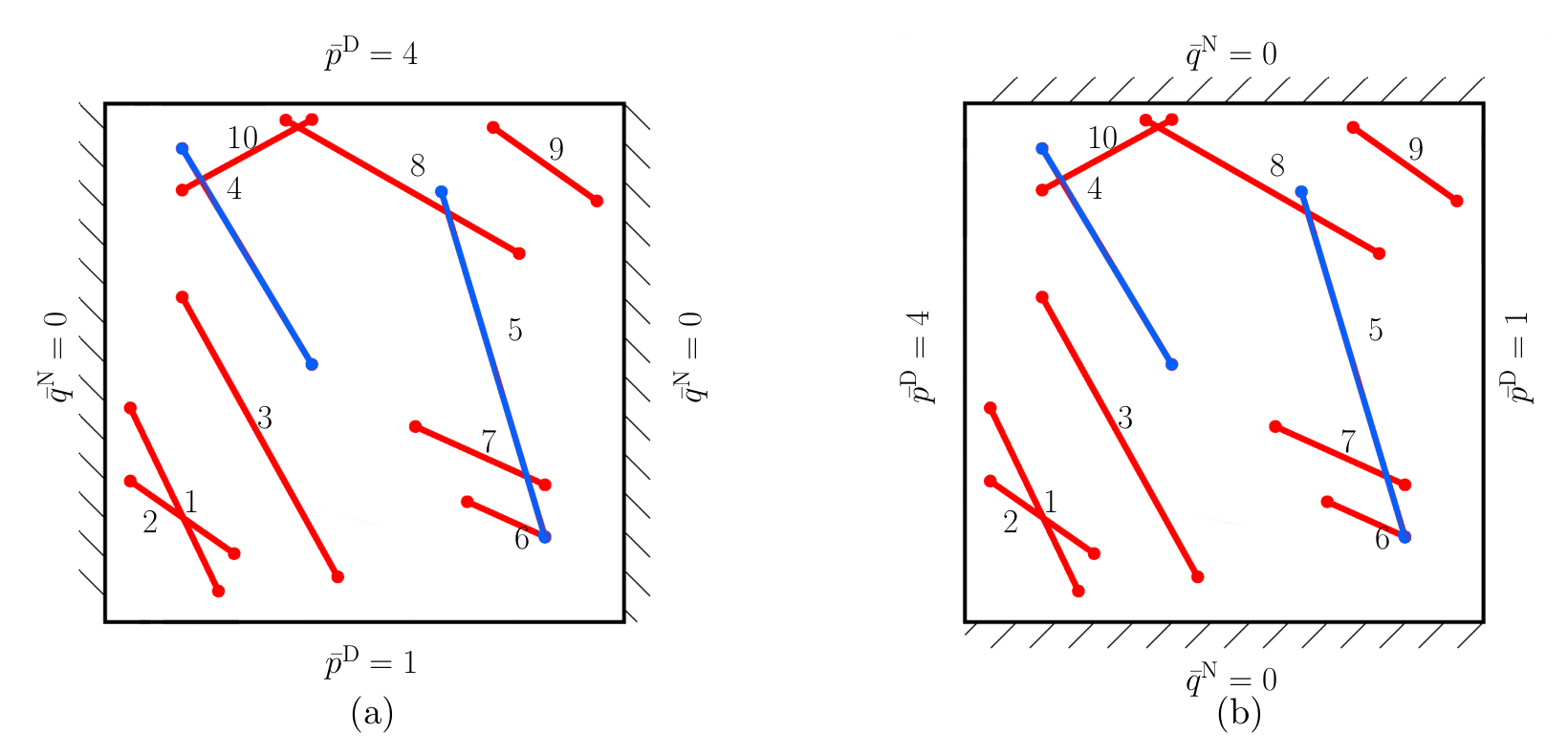}
  \caption{Benchmark 3: computational domain and boundary conditions.}
  \label{fig:complex}
\end{figure}

We consider two meshes, see Figure~\ref{fig:ex2-msh}: a coarse triangular mesh with mesh size $h=0.1$ and a refined mesh  that performs 3 steps of local mesh refinements near the fractures using the procedure detailed in Remark \ref{rk:refine}.
The coarse mesh has $230$ cells, while the fine mesh has 
$1380$ cells.
\begin{figure}[ht]
\centering
\vspace{1pt}
\includegraphics[width=0.40\textwidth]{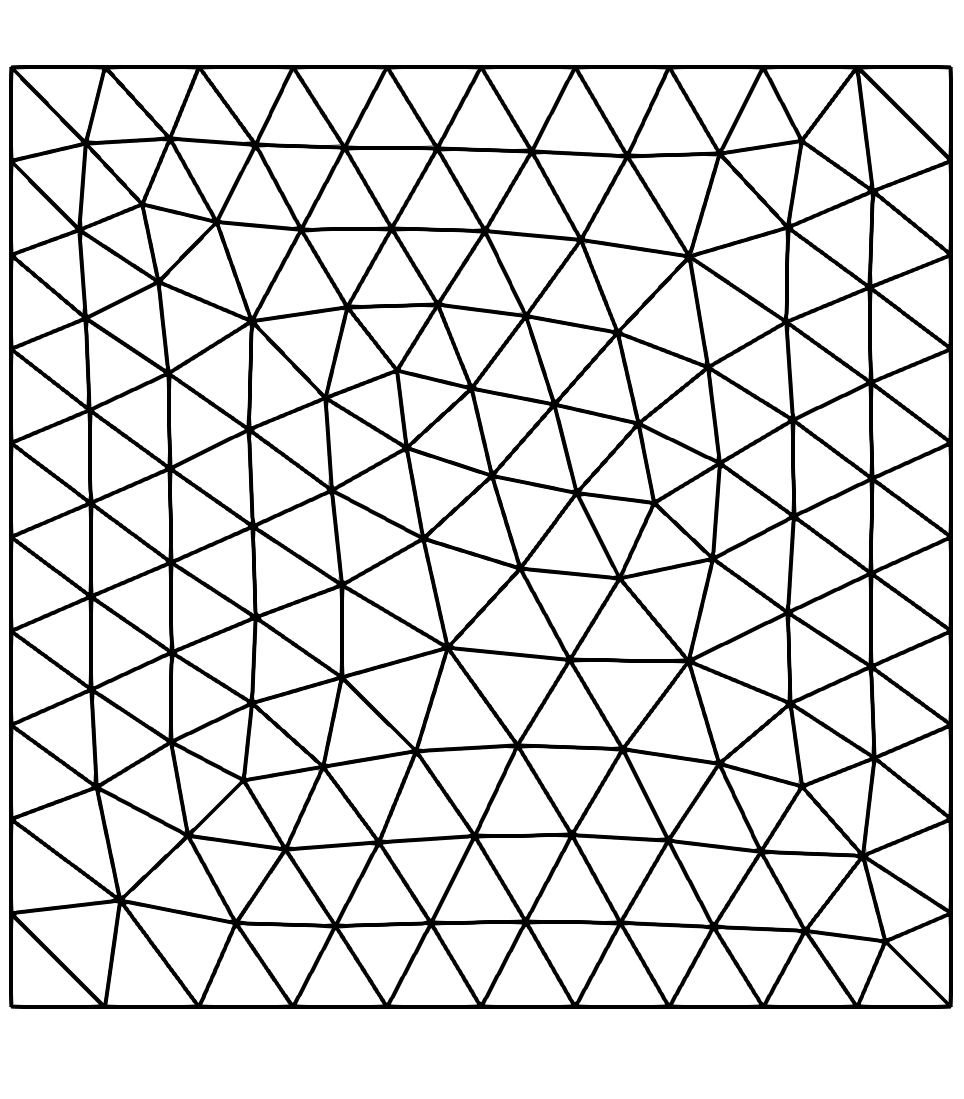}
\includegraphics[width=0.40\textwidth]{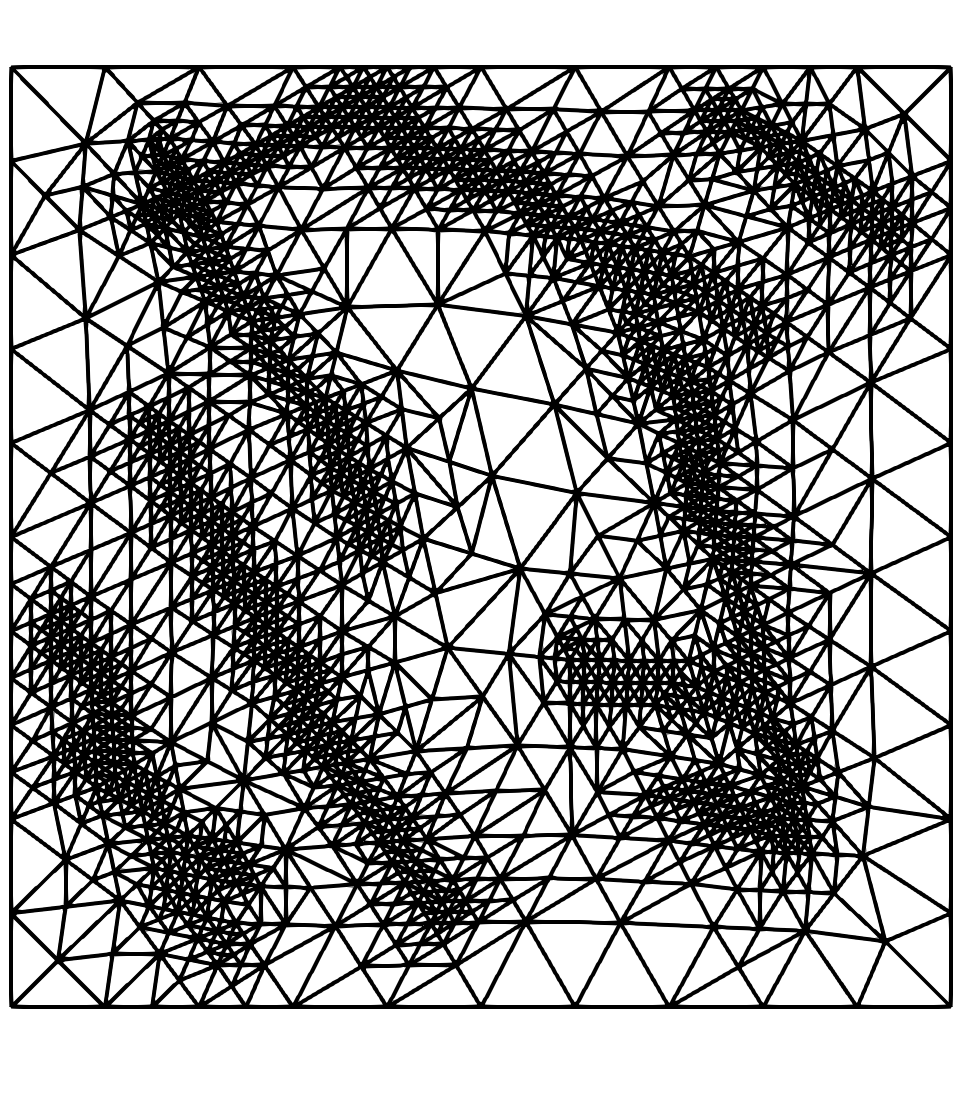}
\vspace{1pt}
\caption{\textbf{Example 2.}
Left: a coarse mesh  with size $h=0.1$.
Right: a locally refined mesh  with $h\approx 0.1/8$
near the fractures.}
\label{fig:ex2-msh}
\end{figure}

We take polynomial degree $k=0,1,2$. 
As suggested from the previous example, we take 
penalty parameters $C_b=C_c=1$ and $s_b=0$, and use $s_c=2$ if $k=0$  and $s_c=3$ if $k=1,2$.
The quantity of interest is the pressure approximation along the line segment $(0, 0.5)--(1, 0.9)$ for both cases, which are recorded in Figure~\ref{fig:ex2-cut}.
We observe a good convergence towards the reference solution as mesh refines for all cases. 
\begin{figure}[ht]
\centering
\includegraphics[width=0.45\textwidth]{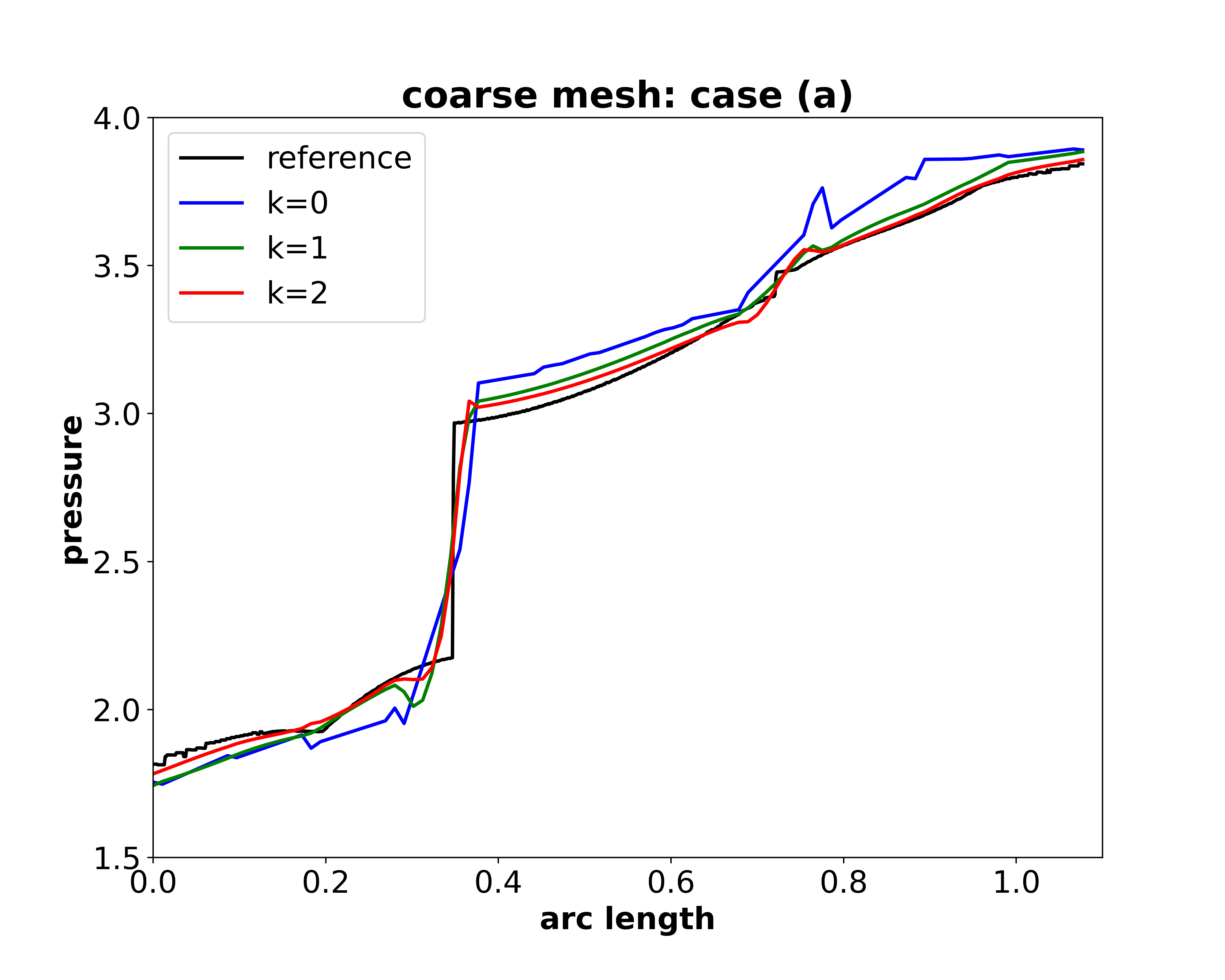}
\includegraphics[width=0.45\textwidth]{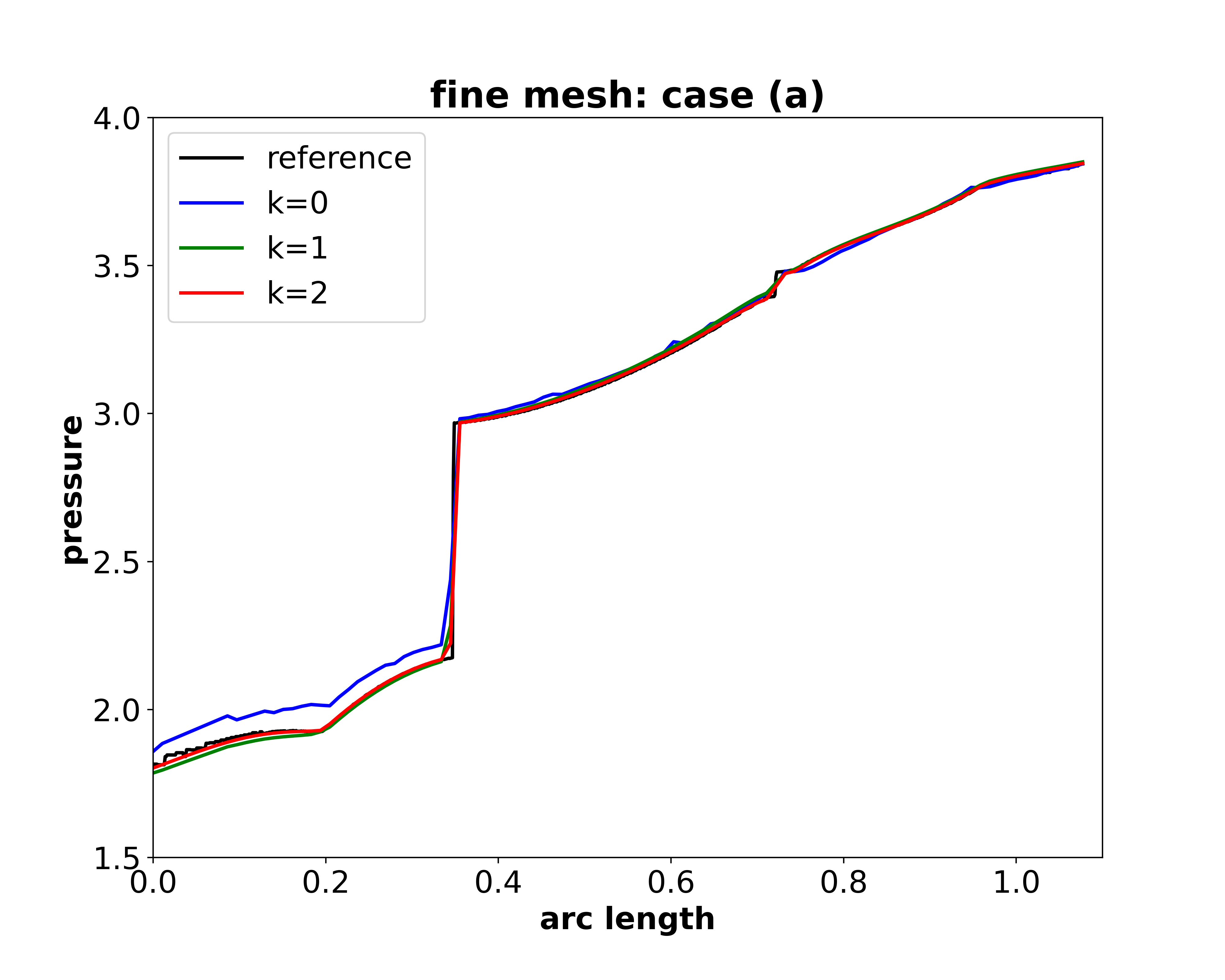}\\
\includegraphics[width=0.45\textwidth]{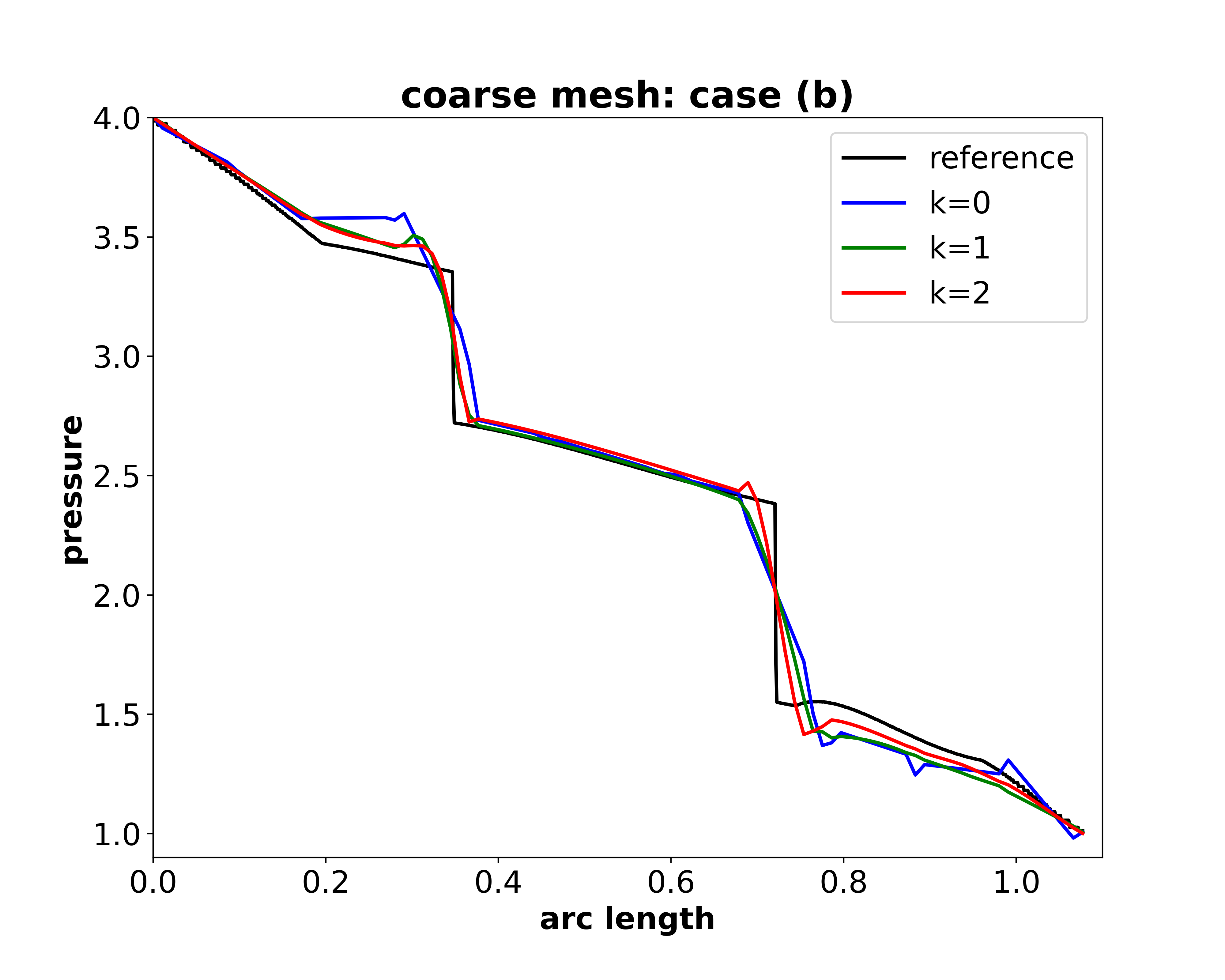}
\includegraphics[width=0.45\textwidth]{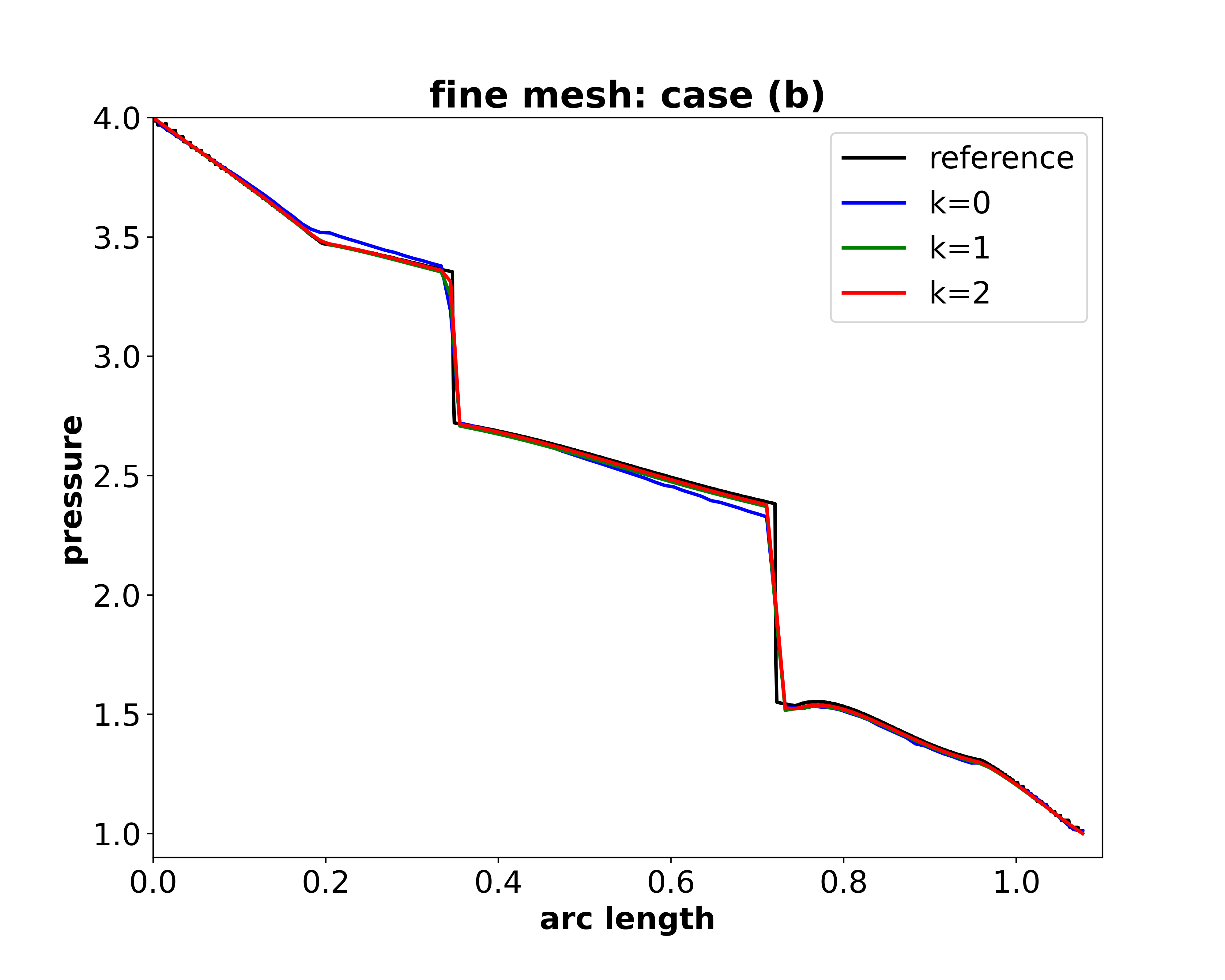}\\
\caption{\textbf{Example 2.} Pressure along line segment $(0,0.5)--(1,0.9)$ for 
the scheme \eqref{hdg} on two meshes.
Top row: case (a); Bottom row: case (b).
}
\label{fig:ex2-cut}
\end{figure}
Contour plots of the pressure on the fine mesh are shown in 
Figure~\ref{fig:ex2-cont}, which are consist with the results in the literature \cite{FLEMISCH2018239}.
\begin{figure}[ht]
\centering
\includegraphics[width=0.31\textwidth]{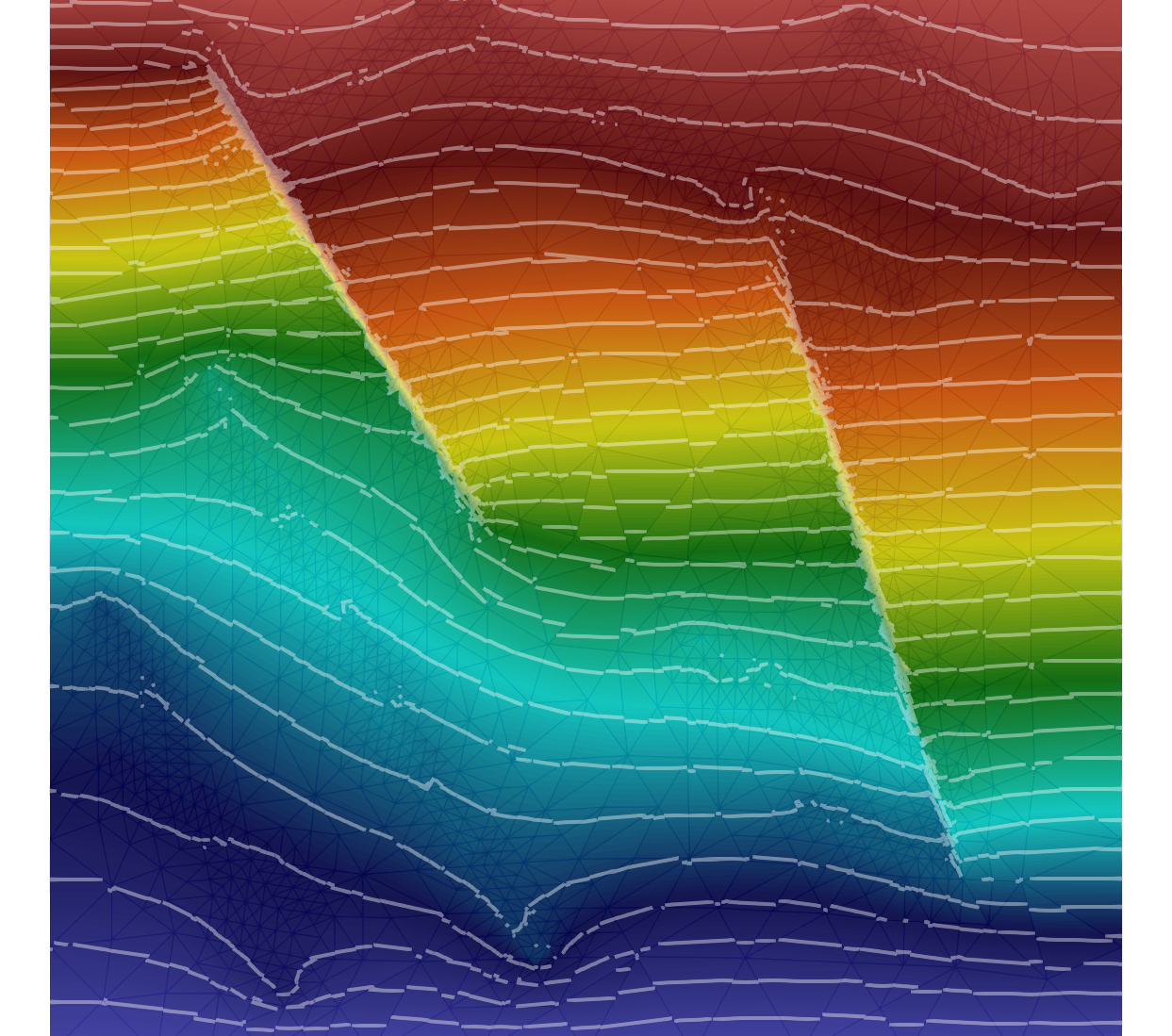}
\includegraphics[width=0.31\textwidth]{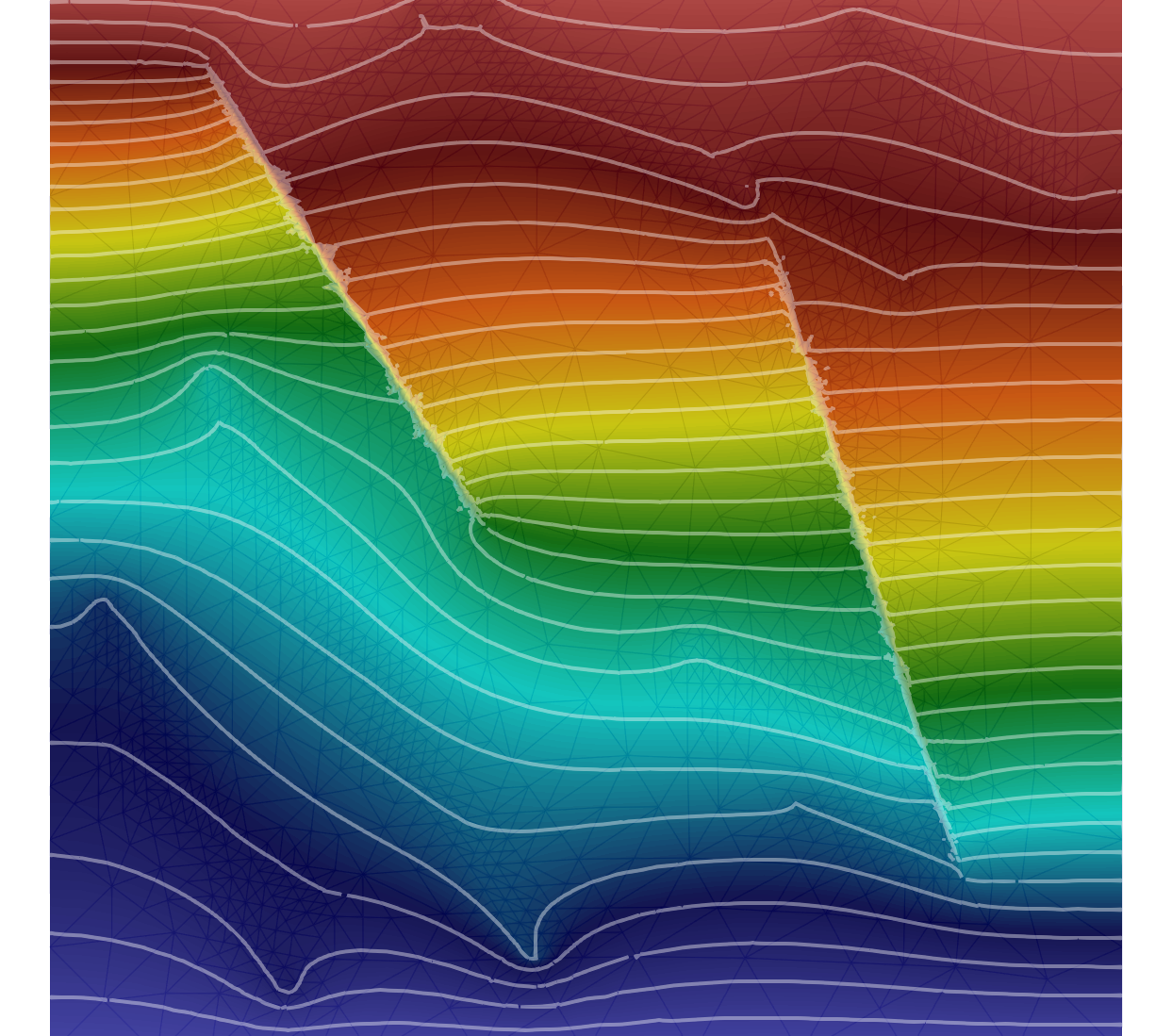}
\includegraphics[width=0.31\textwidth]{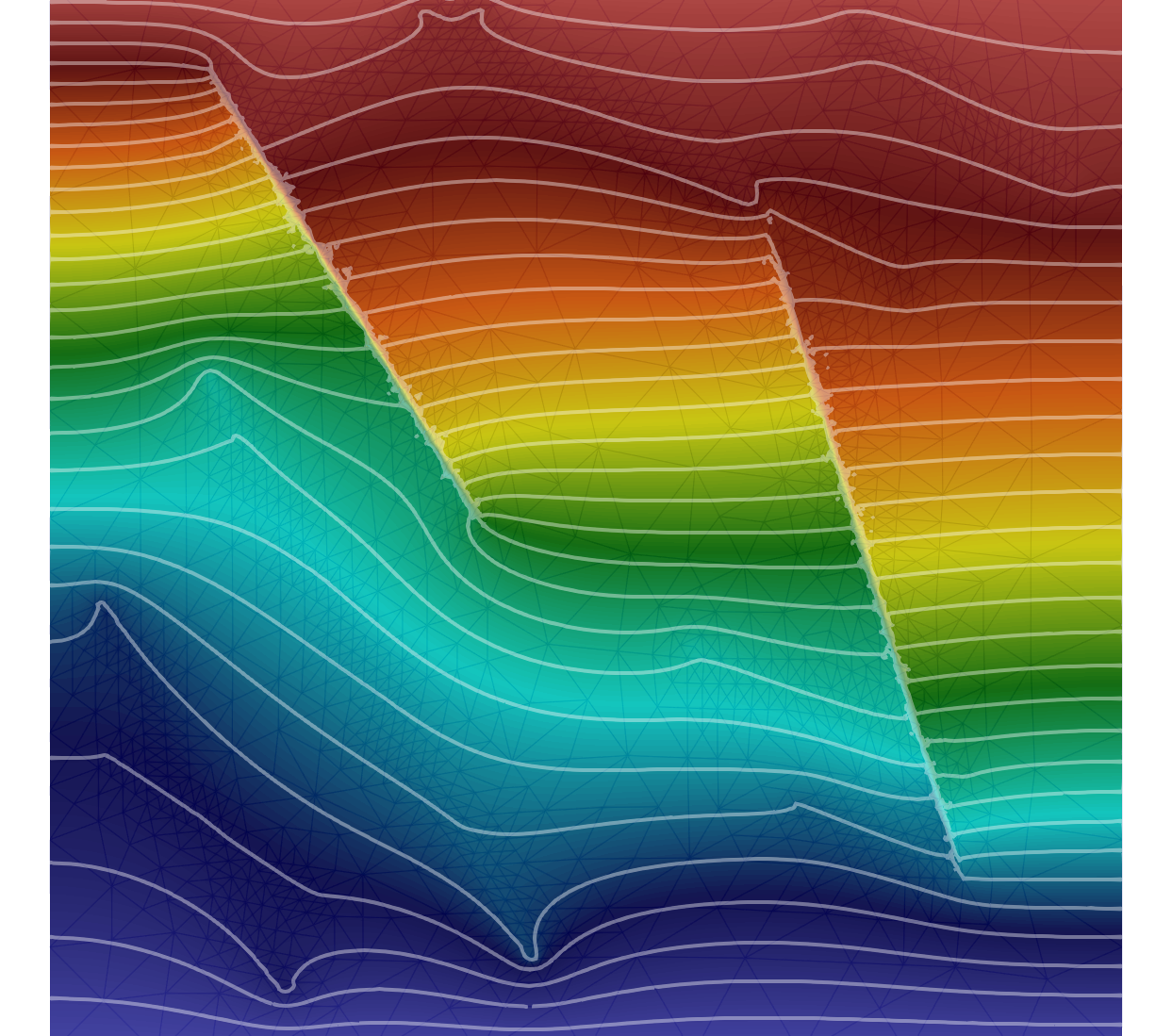}
\includegraphics[width=0.31\textwidth]{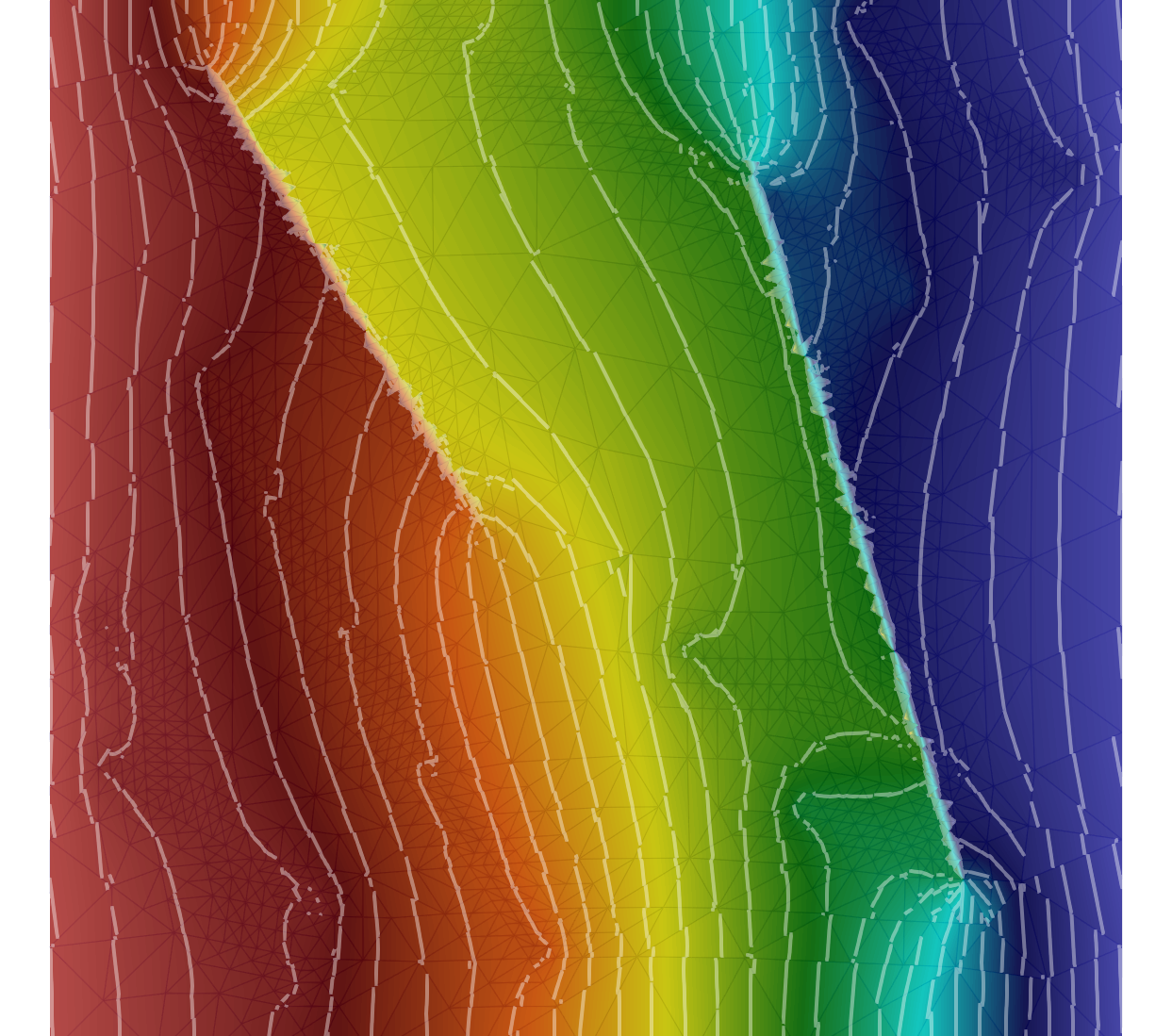}
\includegraphics[width=0.31\textwidth]{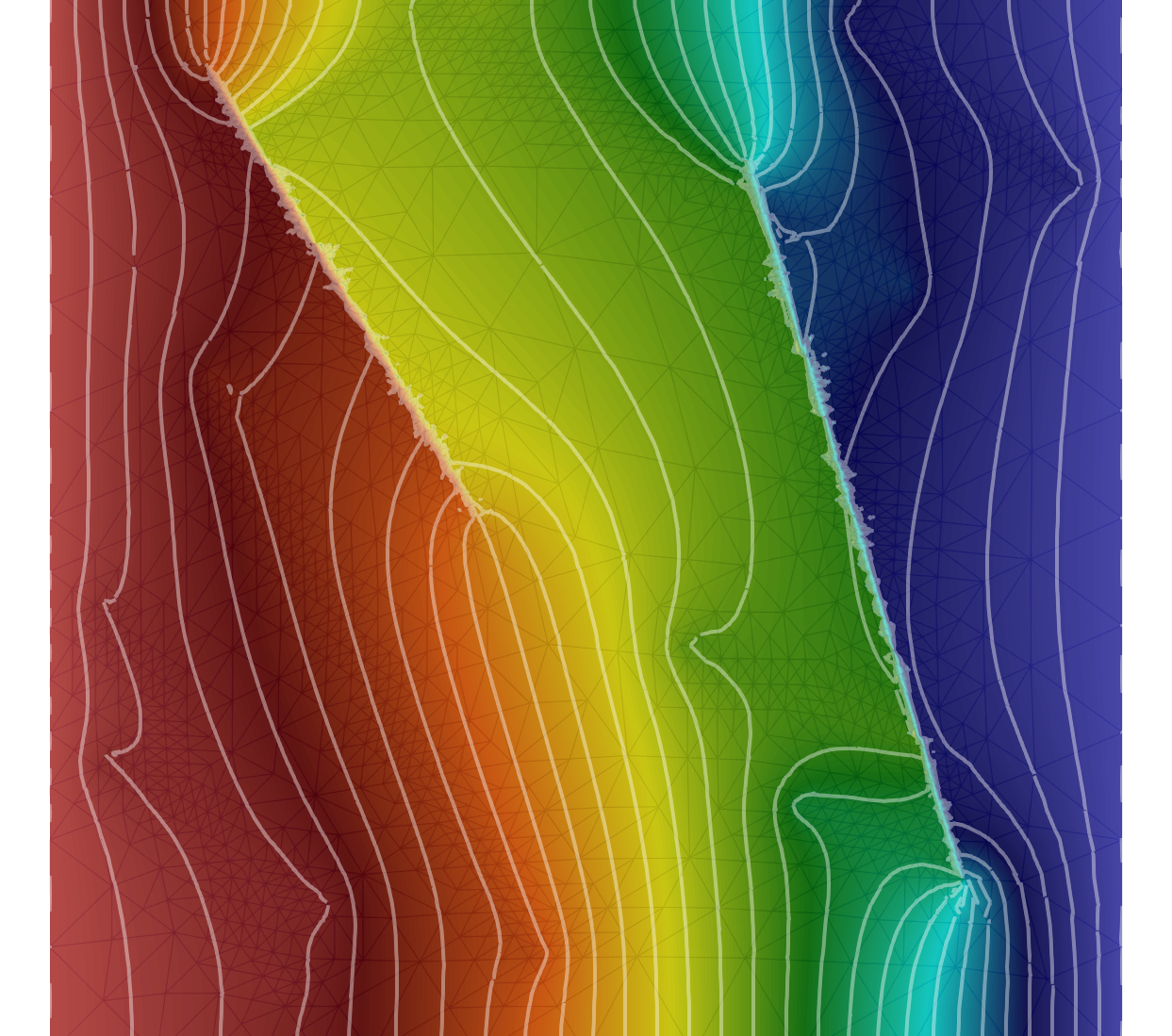}
\includegraphics[width=0.31\textwidth]{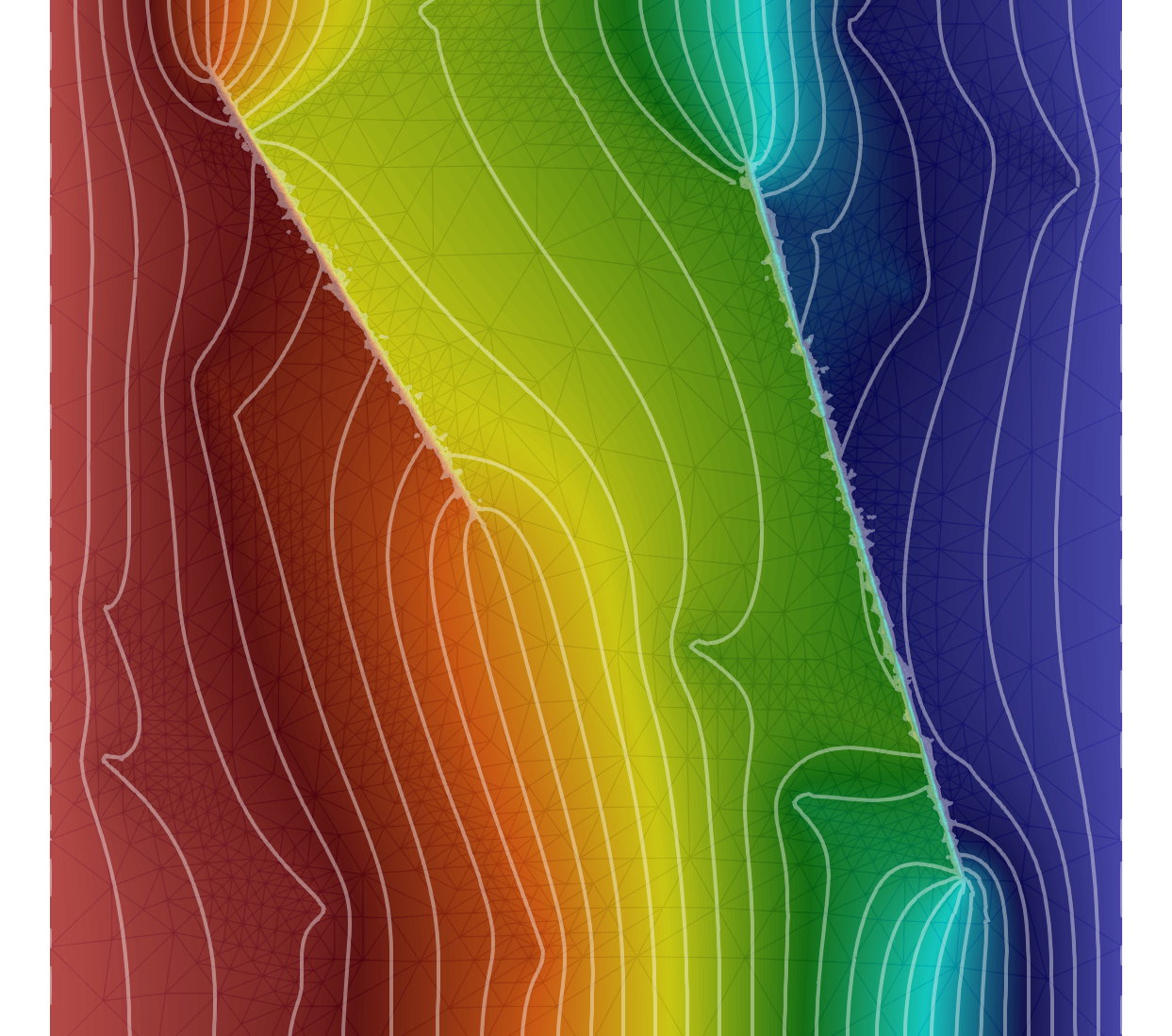}
\caption{\textbf{Example 2.} Pressure contour for $k=0$ (left), 
$k=1$ (middle), and $k=2$ (right).
Color range: (0,1). Thirty uniform contour lines 
from 0 to 1. 
}
\label{fig:ex2-cont}
\end{figure}

\subsection*{Example 3: a Realistic Case in 2D}
We consider a real set of fractures from an interpreted outcrop in
the Sotra island, near Bergen in Norway.
The setup is adapted from \cite[Benchmark 4]{FLEMISCH2018239}.
The domain along with boundary conditions is given in  Figure~\ref{fig:ex3-domain}. 
The size of the domain is 700 $m$ $\times$ 600 $m$ with uniform scalar permeability $\mathbb K_m=10^{-14}m^2$. 
The set of fractures is composed of 63 line segments 
with thickness $\epsilon=10^{-2}m^2$. The exact coordinates for the
fracture positions are provided in the git repository \url{https://git.iws.uni-stuttgart.de/benchmarks/fracture-flow}.
Three subcases will be considered: (a) all conductive fractures with 
permeability $k_c=10^{-8}m^2$, (b) all blocking fractures with permeability $k_b=10^{-20}m^2$, and (c) 9 blocking fractures with $k_b=10^{-20}m^2$ and 54 conductive fractures with
$k_c=10^{-8}m^2$. 
Location of the blocking/conductive fractures for case (c) are marked in red/blue in the right panel of Figure~\ref{fig:ex3-msh}.
We note that while case (a) has been extensively studied, see e.g. \cite{FLEMISCH2018239}.
The other two cases are new. 
\begin{figure}[ht]
  \centering
    \includegraphics[width=.6\textwidth]{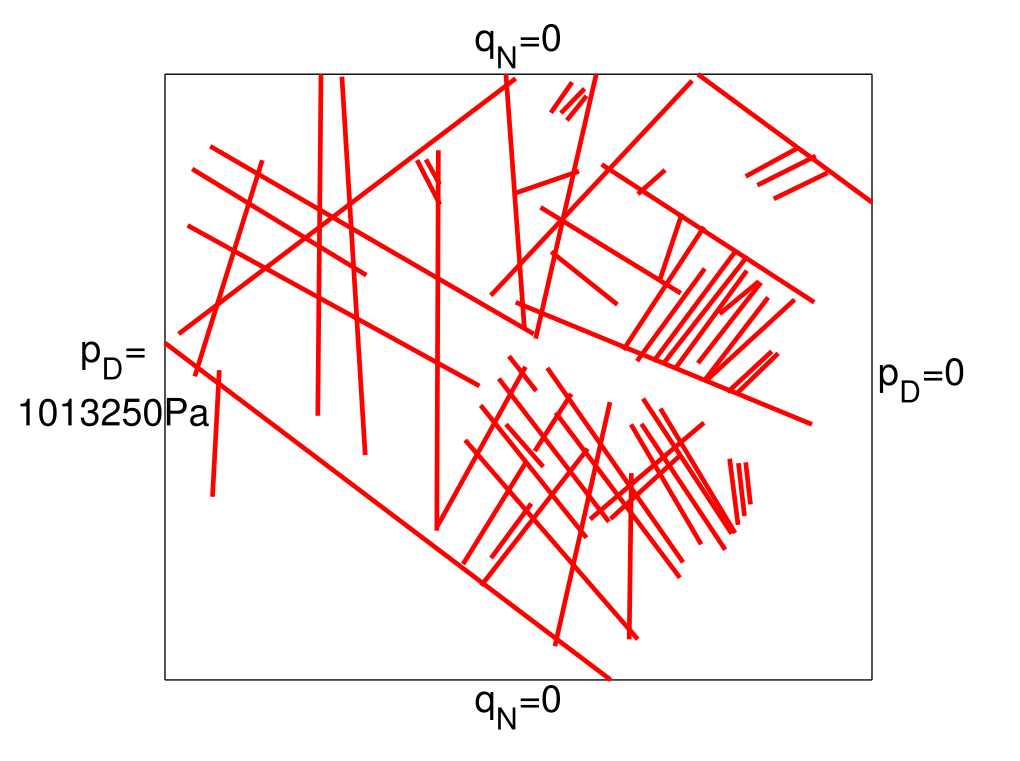}
  \caption{\textbf{Example 3:} Computational domain and boundary conditions.
  }
  \label{fig:ex3-domain}
\end{figure}

Here we run simulation on a non-dimensional setting to avoid extreme values where 
the domain is scaled back to be $\Omega = (0,1)\times (0, 6/7)$, 
matrix permeability $\mathbb K_m=1$, and inflow pressure boundary condition  $p_D=1$ on the left boundary.
We consider our scheme \eqref{hdg} with polynomial degree $k=0,1,2$ 
on an unfitted mesh obtained from a uniform triangular mesh with $h=0.01$ by performing two steps of local mesh refinements around the fractured cells; see left of Figure~\ref{fig:ex3-msh}.
The mesh has about $154k$ total triangular cells, and 
about $27.5k$ fractured cells, which further split to $5.5k$ blocking fractured cells and $22k$ conductive fractured cells for case (c).
\begin{figure}[ht]
\centering
\vspace{1pt}
\includegraphics[width=0.45\textwidth]{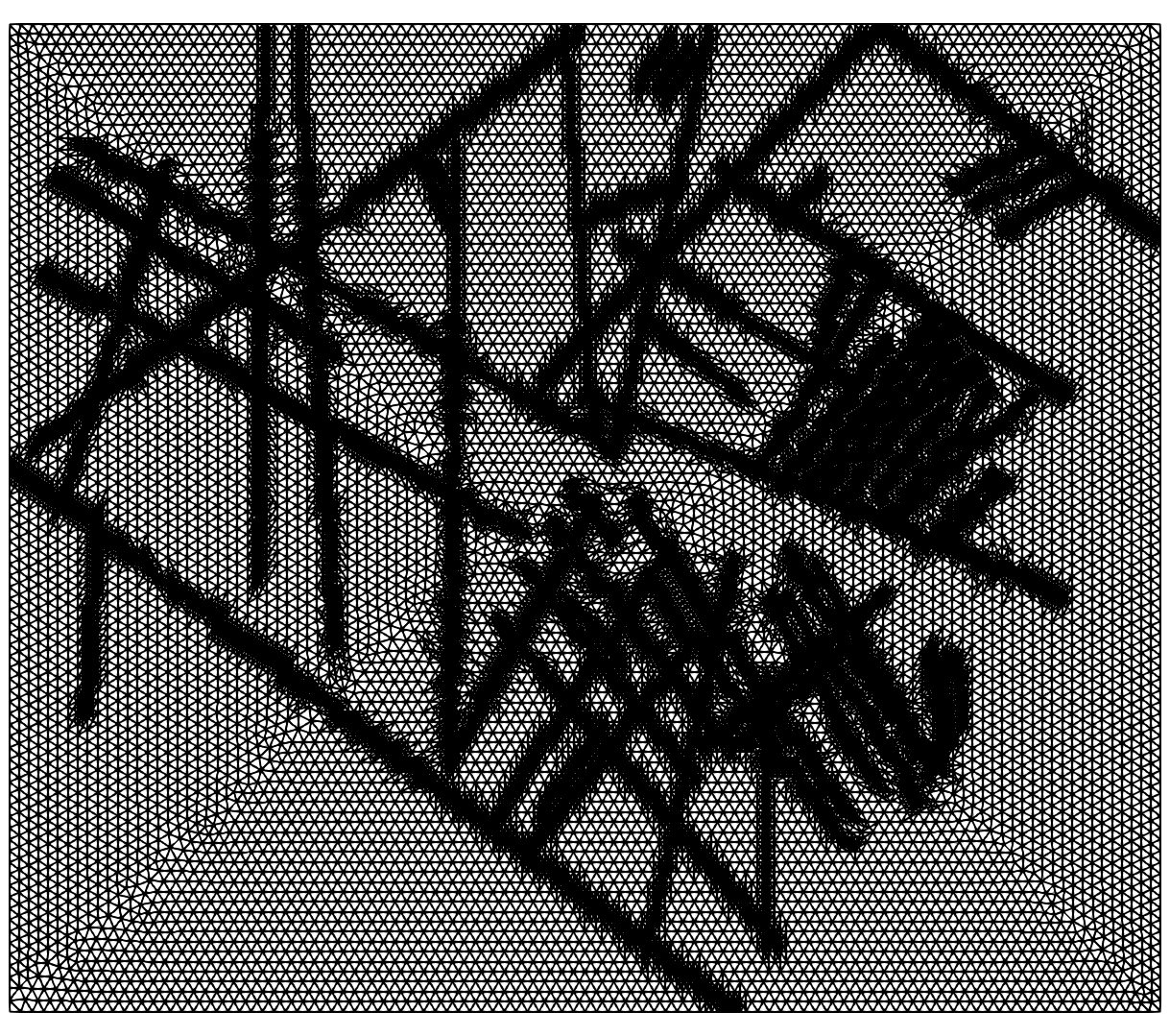}
\includegraphics[width=0.45\textwidth]{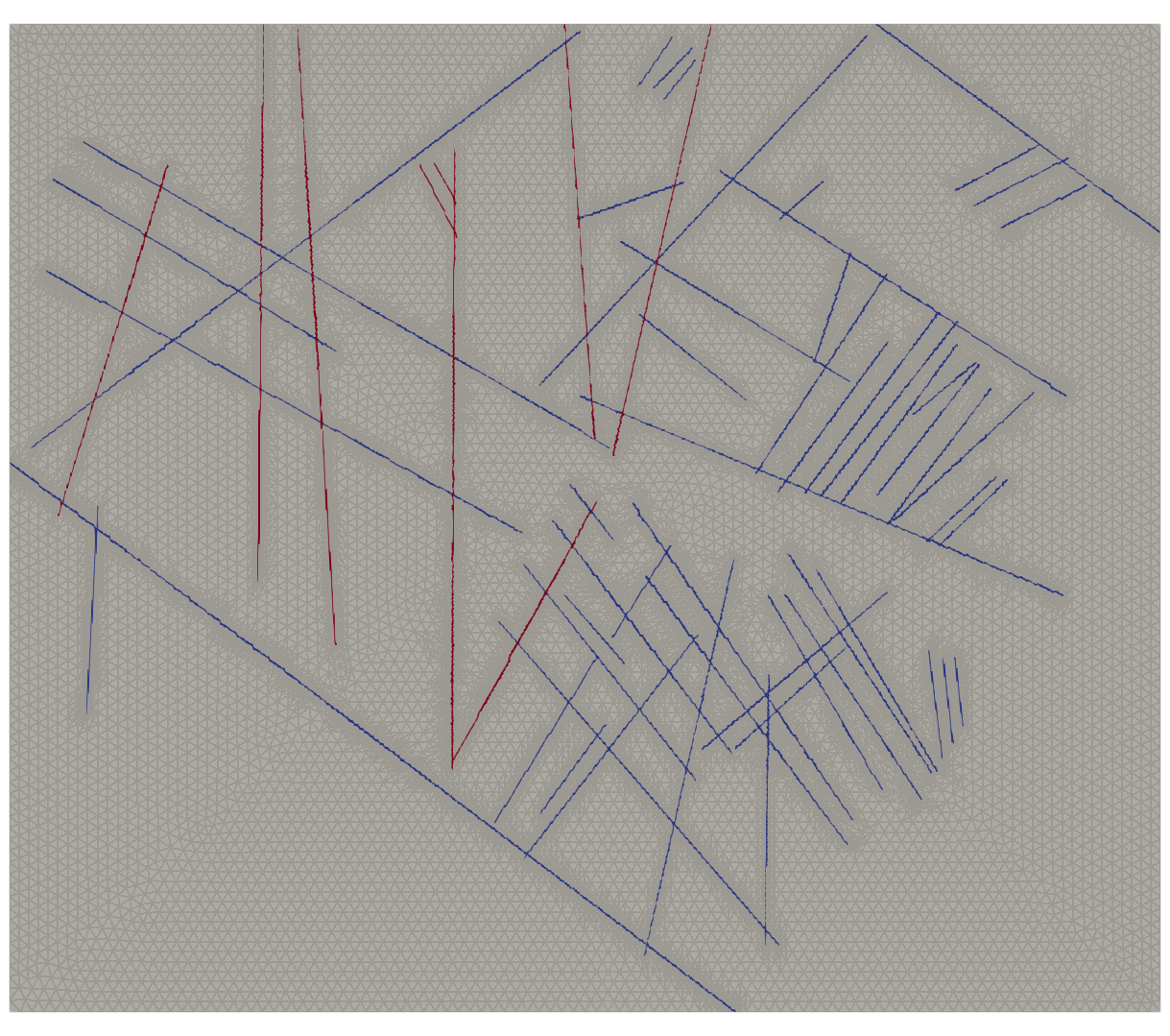}
\vspace{1pt}
\caption{\textbf{Example 3.}
Left: unfitted mesh with 154,174 cells.
Right: location of blocking (in red) and conductive (in blue) fractured cells for case (c).}
\label{fig:ex3-msh}
\end{figure}
The penalty parameters in \eqref{stab} are
given in Table~\ref{tab:1}. 
Here due to stronger conductive/blocking effects (with  permeability differs by six orders of magnitude), we need to choose the penalty parameters $s_b$ and $C_c$ differently than the previous two examples. In particular, we note that the $C_c$ values are tuned to make the case (a) results matching with existing work. Moreover, the stabilization on blocking fractured cells are reduced by taking $s_b=2$, as taking larger stabilization with $s_b=0$ leads to pressure leakage across blocking fractured cells.  
% Here we reduce pressure stabilization on blocking fractured cells to allow pressure discontinuity across element boundaries. We particular note that for this example,  we have $\epsilon_i\gamma_i^{-1}=10^4$ on blocking fractured cells which indicates a stronger blocking effects than the previous cases. 
\begin{table}[ht!]
\centering
\begin{tabular}{c|c c c c}
 $k$    & $C_b$ & $s_b$& $C_c$ & $s_c$\\
     \hline
0& 1 &  2     & 6 &  2\\[.1ex]
1& 1 &  2     & 0.08 &  3\\[.1ex]
2& 1 &  2     & 0.16 &  3\\
\end{tabular}
\caption{\textbf{Example 3.} Choice of the penalty parameters for different polynomial degree $k$.}
\label{tab:1}
\end{table}
% These problems are extremely challenging due to the complex fracture networks. 

The pressure approximations along the two lines $y = 5/7$ and $x = 625/700$ are recorded in Figure~\ref{fig:ex3-cut}, where reference data from the {Mortar-DFM} scheme on a fitted mesh with about $10k$ cells reported in \cite{FLEMISCH2018239} for case (a) is also presented.
We observe a good agreement with the reference data for case (a) for our schemes. 
Moreover, we observe very close results for $k=1$ and $k=2$ for case (b) and case (c), where the results for $k=0$ is slightly off due to coarse mesh resolution and low-order approximations.
Further refining the mesh for $k=0$ leads to results closer to the $k=1,2$ cases in Figure~\ref{fig:ex3-cut}.
\begin{figure}[ht]
\centering
\includegraphics[width=0.45\textwidth]{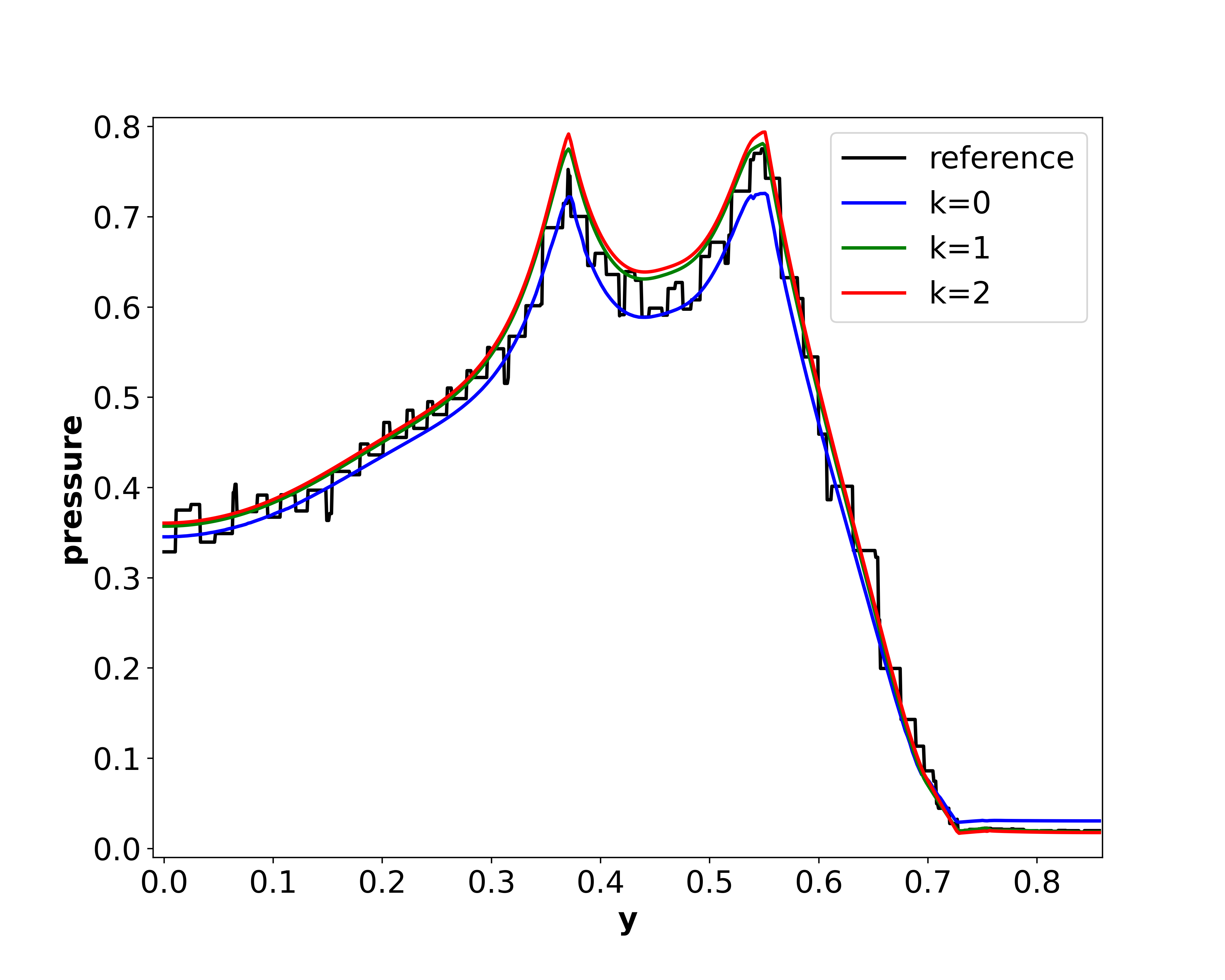}
\includegraphics[width=0.45\textwidth]{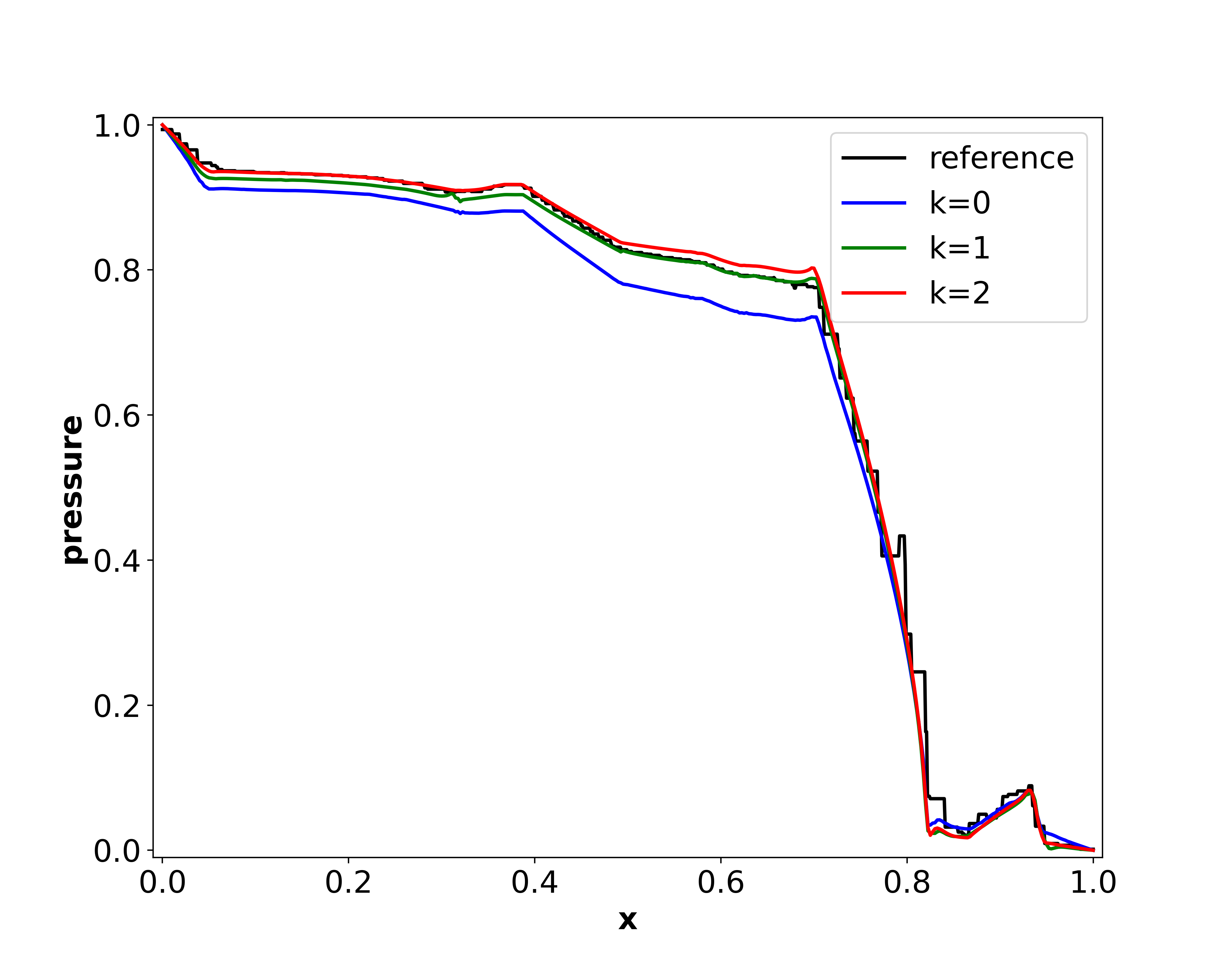}\\
\includegraphics[width=0.45\textwidth]{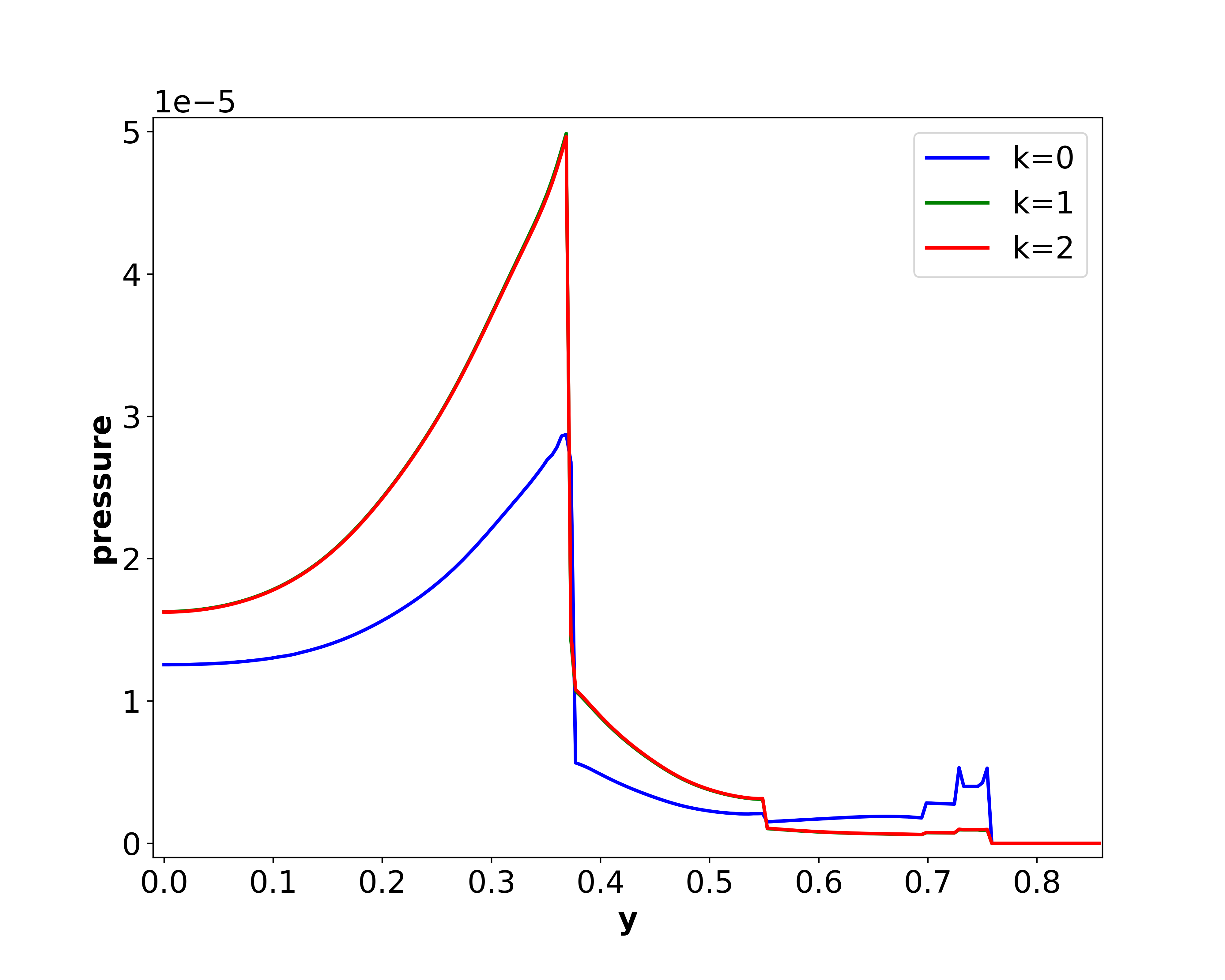}
\includegraphics[width=0.45\textwidth]{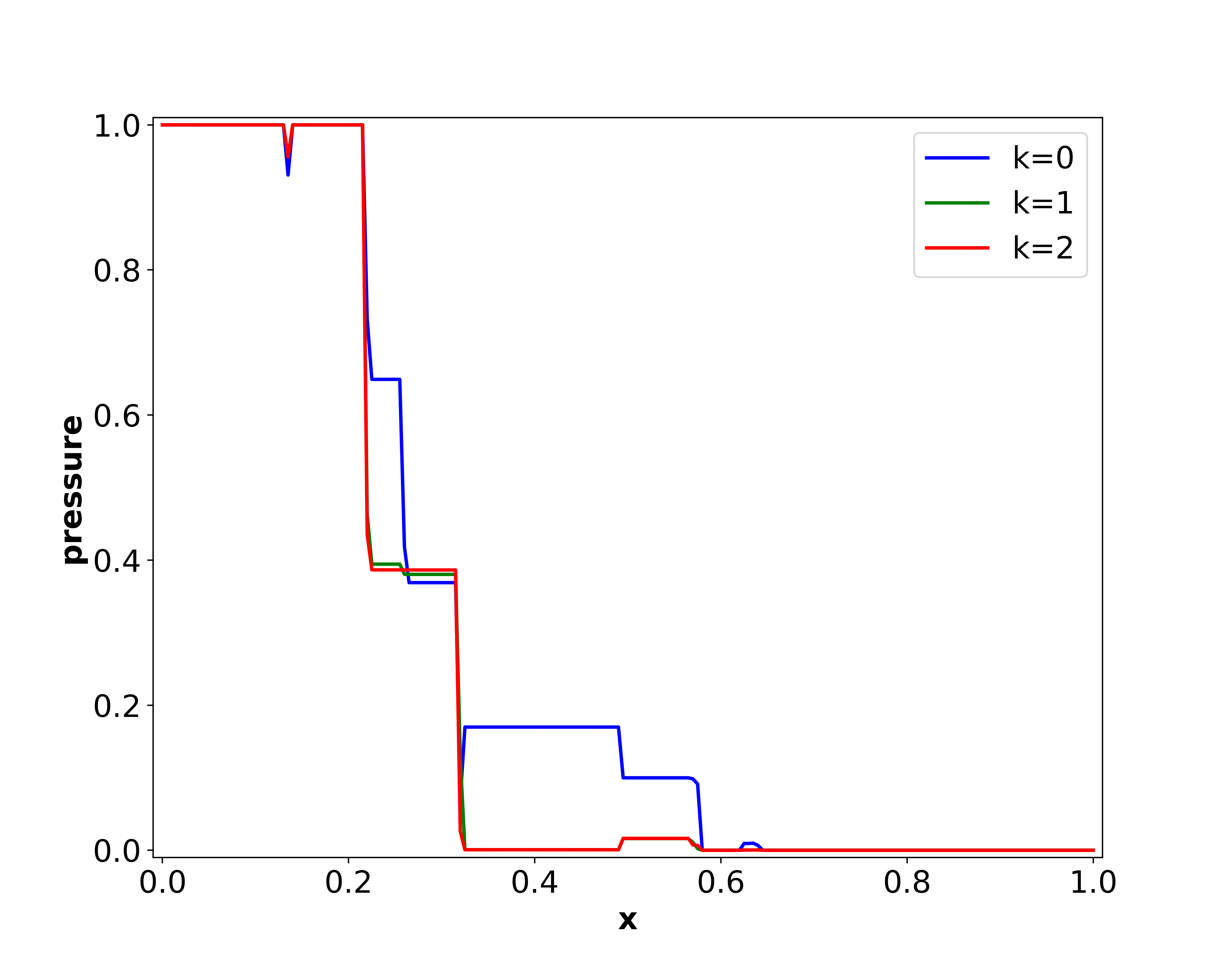}\\
\includegraphics[width=0.45\textwidth]{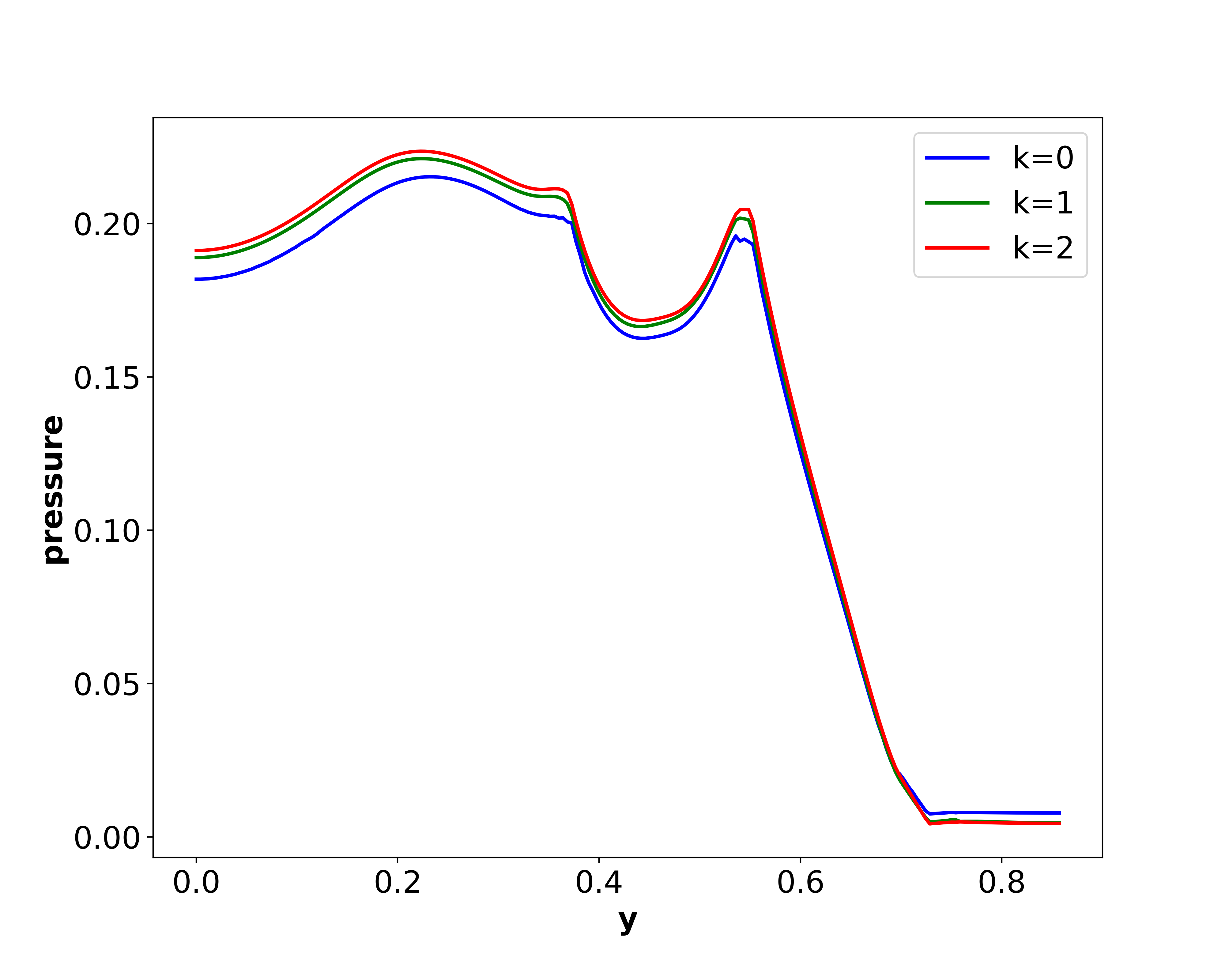}
\includegraphics[width=0.45\textwidth]{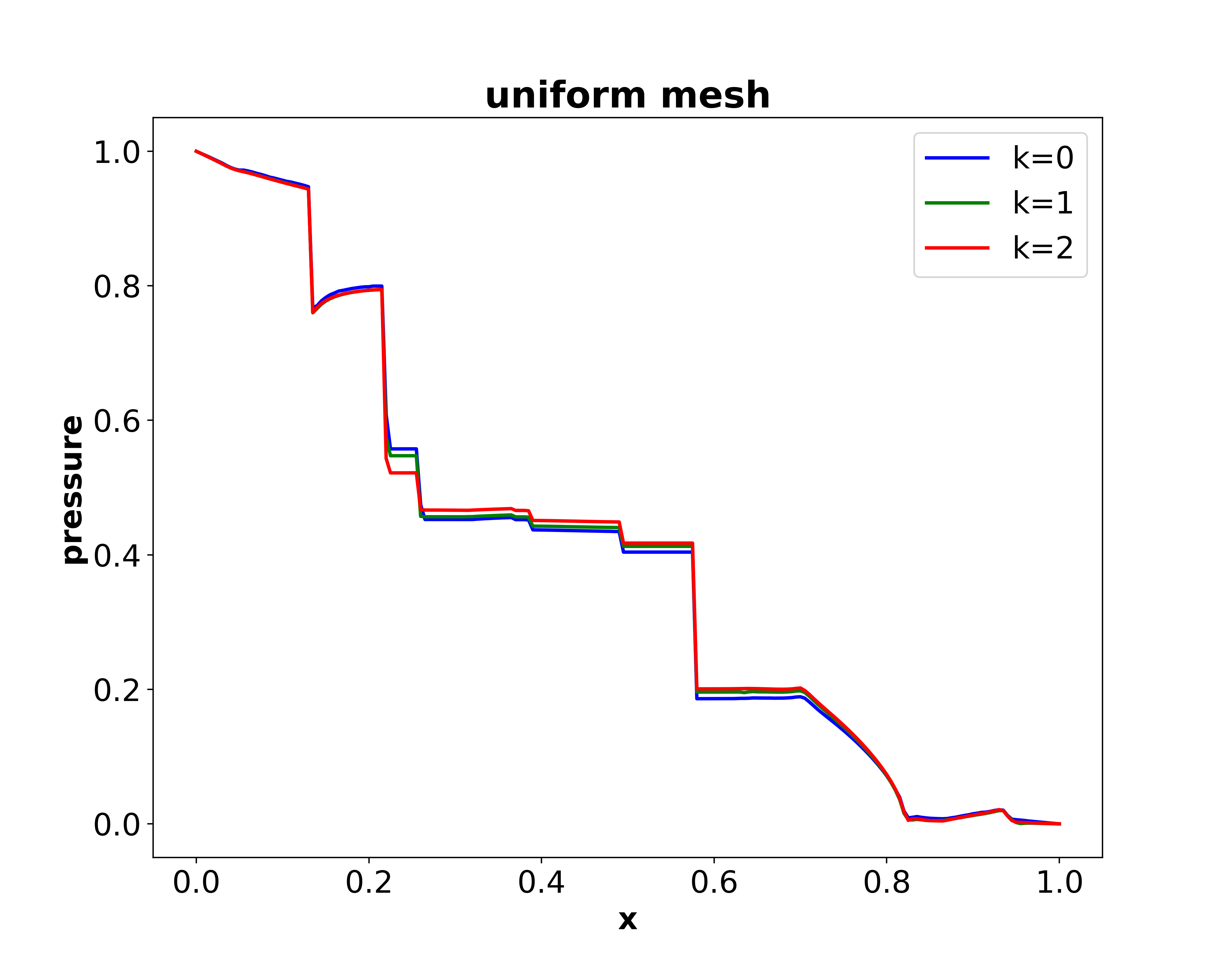}\\
\caption{\textbf{Example 3.} Pressure along line $y=5/7$ (left column) and 
along line $x=625/700$ (right column). 
Here reference data for case (a) is the result for the {\sf Mortar-DFM} scheme in \cite{FLEMISCH2018239}. 
}
\label{fig:ex3-cut}
\end{figure}

Contour plots of the pressure are shown in 
Figure~\ref{fig:ex3-cont}, 
where the case (a) results are again consist with those in the literature \cite{FLEMISCH2018239}. We also clearly observe the blocking effects (with discontinuous pressures) of the fractures for case (b), and the combined conductive/blocking effects of the fractures for case (c). 
% It is interesting to see that for case (c), the blocking fractures blocked the flow on the top right region, and it mainly flow through the bottom section of connected conductive fractures. 
This example confirms the ability of the proposed HDG scheme \ref{hdg} in simulating realistic complex fracture networks on unfitted meshes with both conductive and blocking fractures.

\begin{figure}[ht]
\centering
\includegraphics[width=0.31\textwidth]{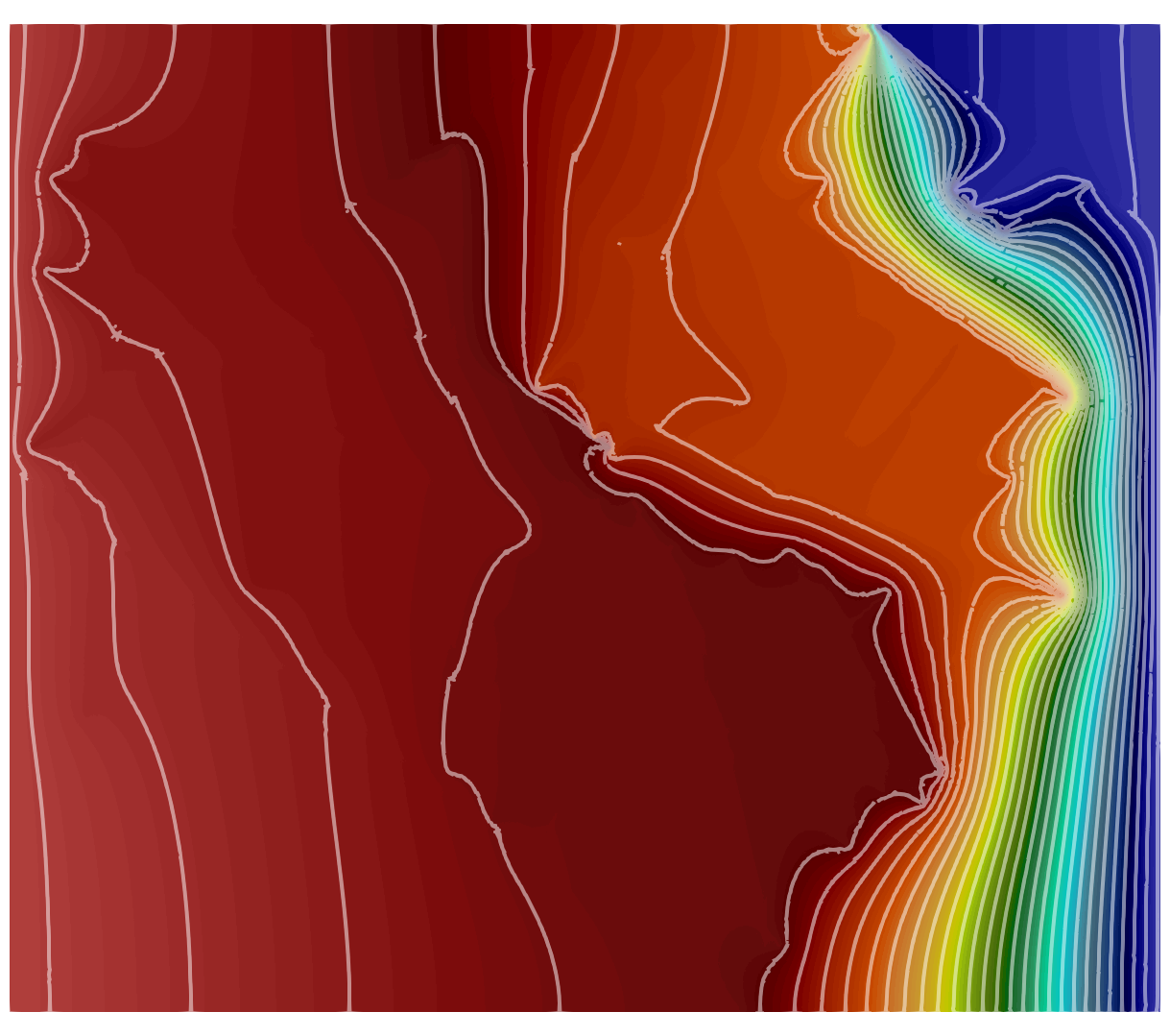}
\includegraphics[width=0.31\textwidth]{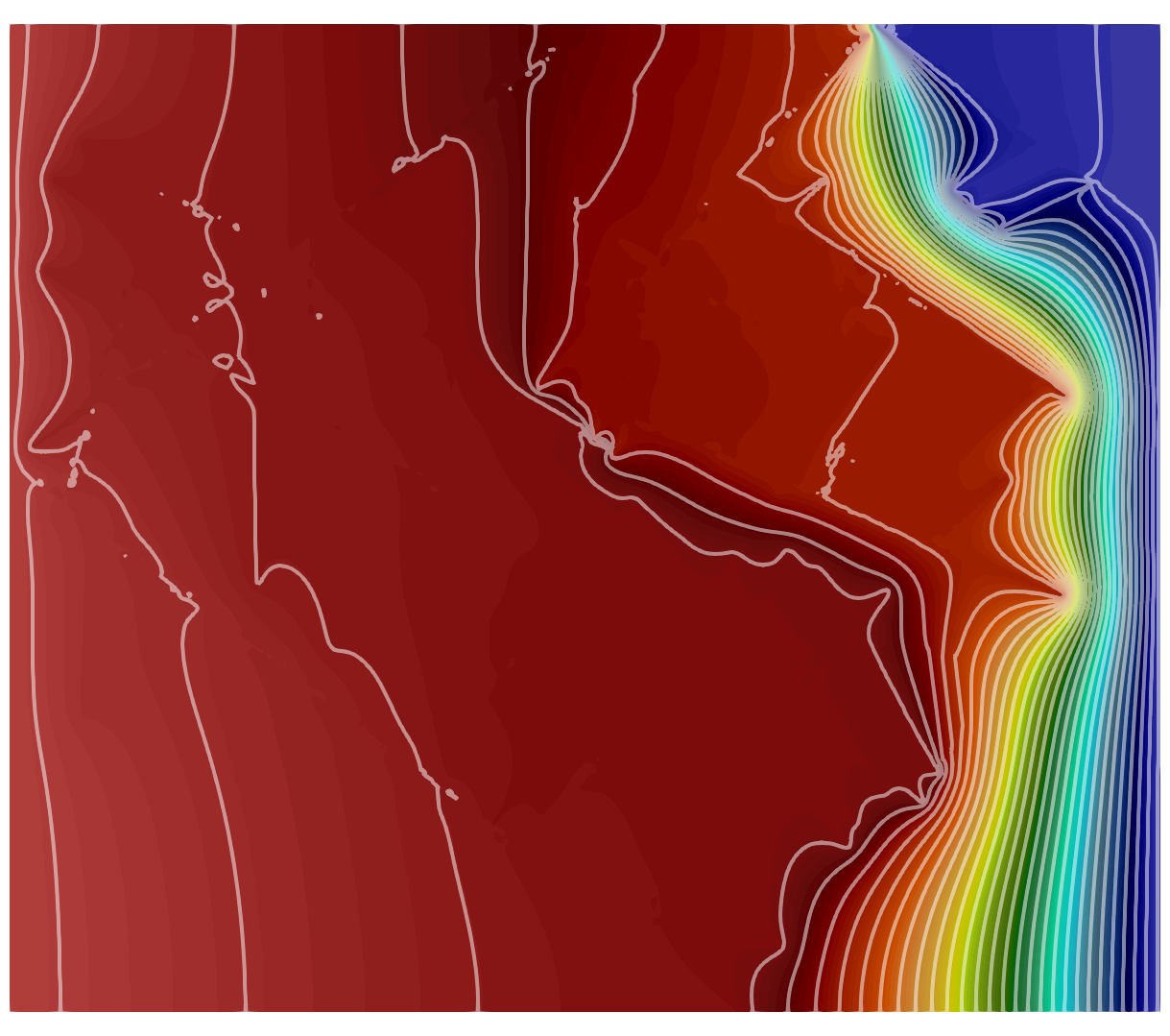}
\includegraphics[width=0.31\textwidth]{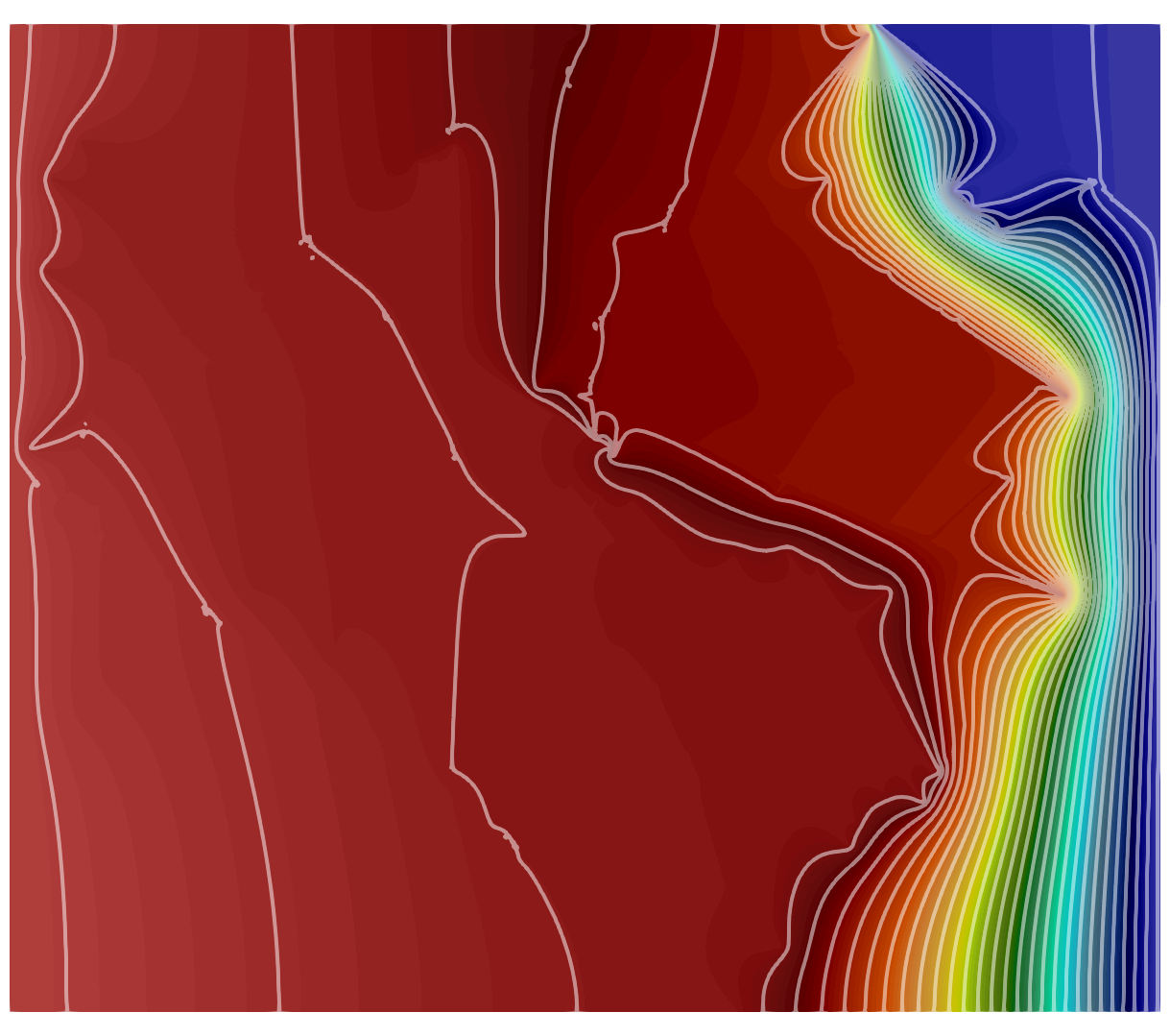}
\includegraphics[width=0.31\textwidth]{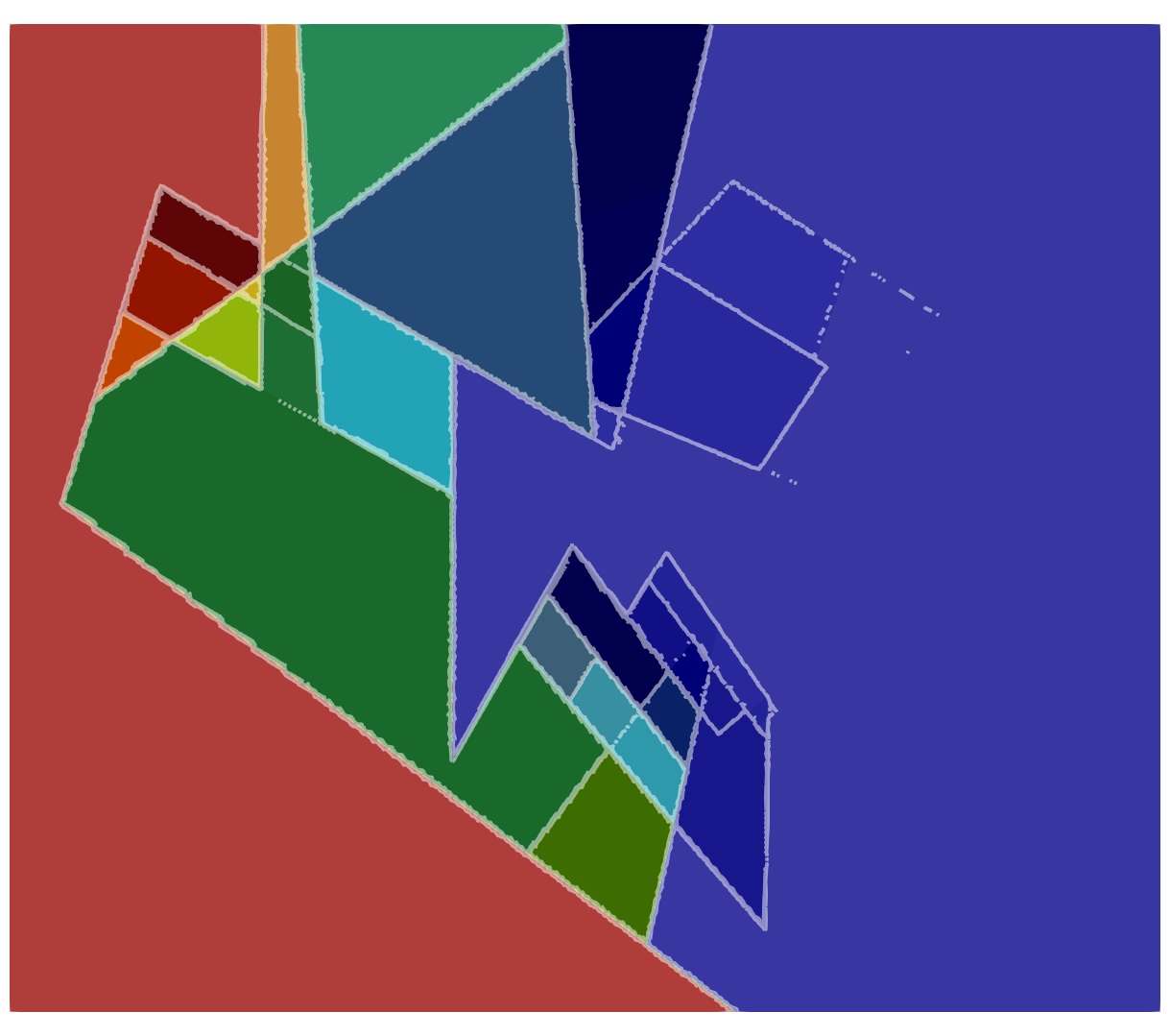}
\includegraphics[width=0.31\textwidth]{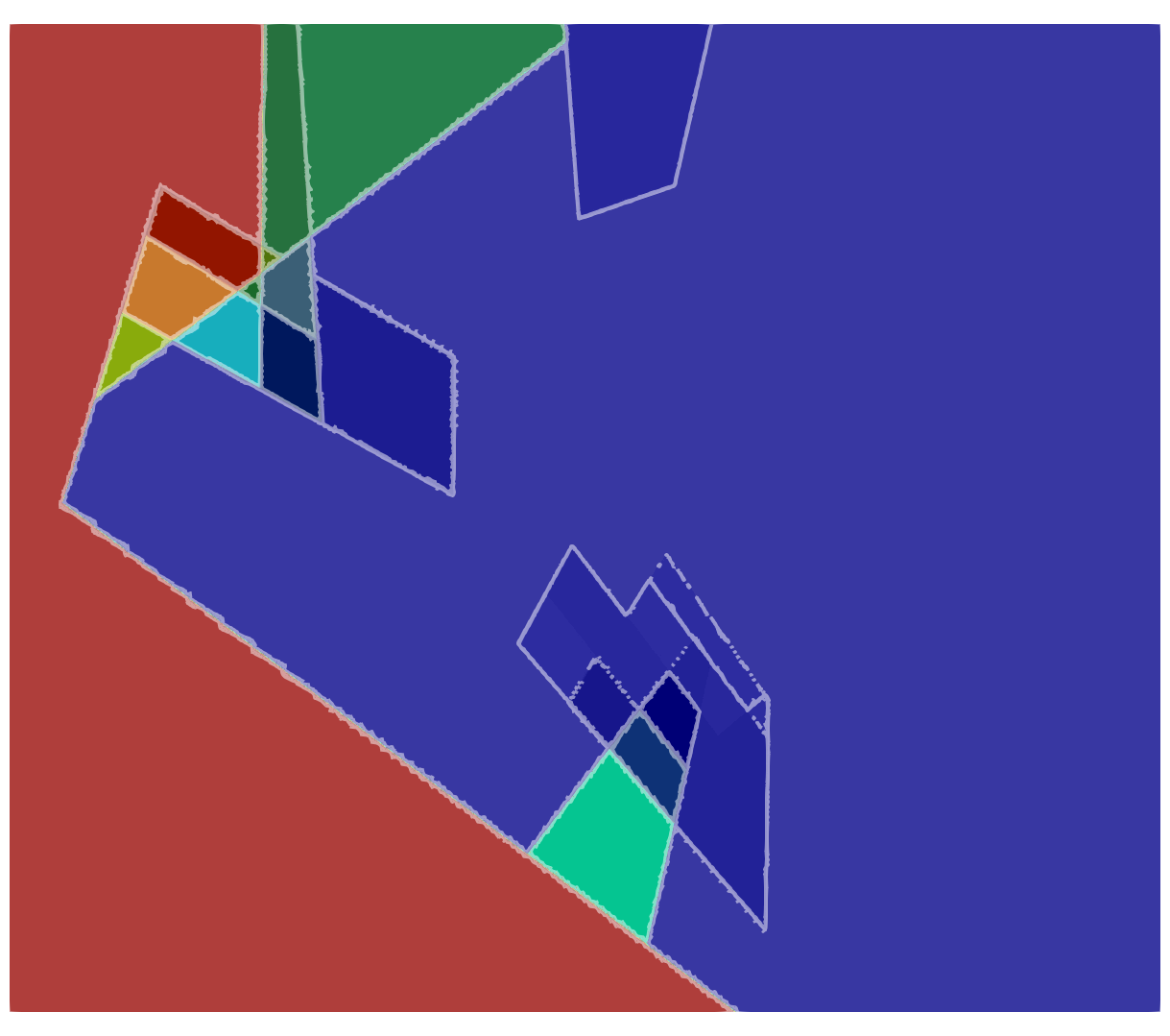}
\includegraphics[width=0.31\textwidth]{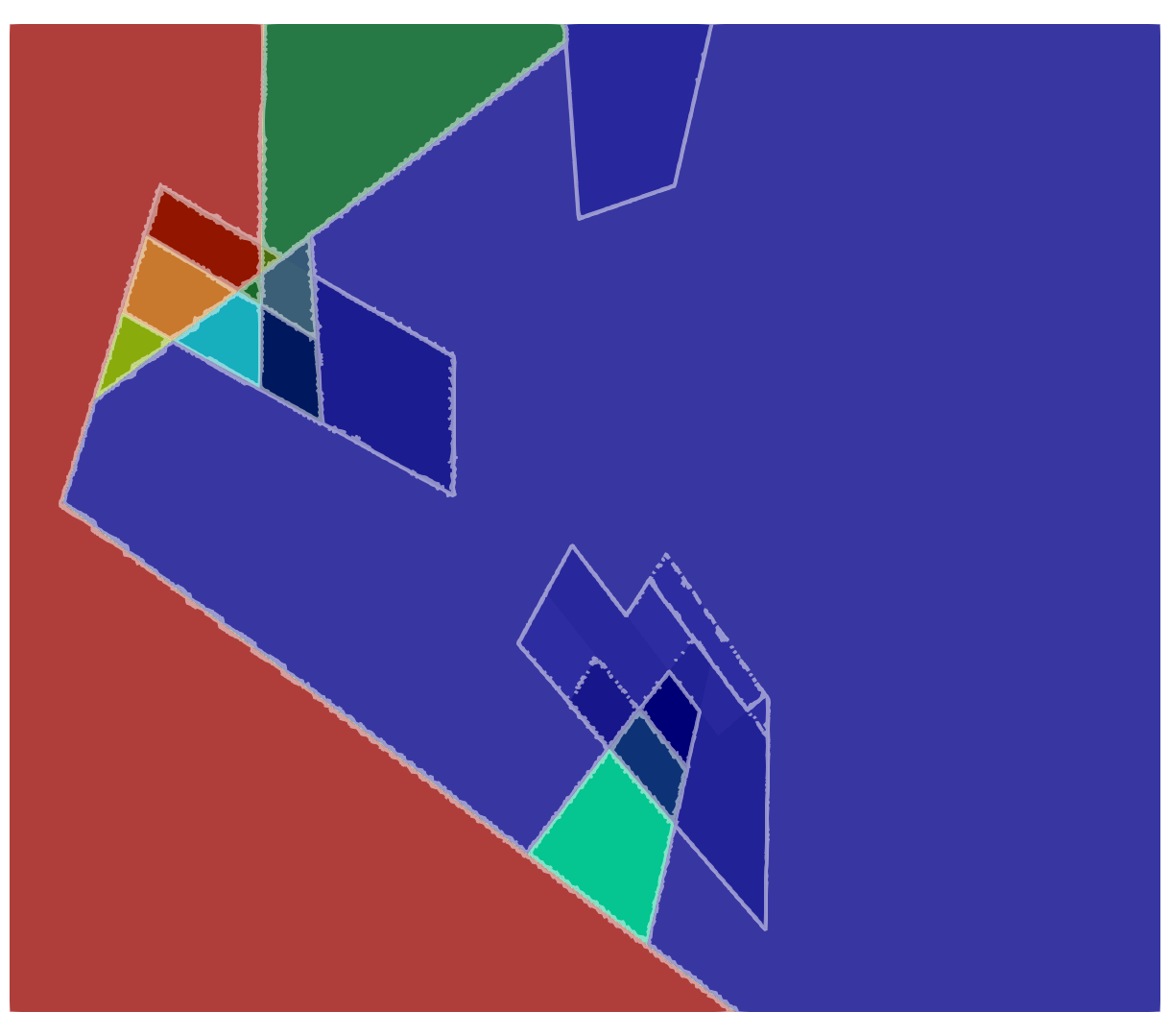}
\includegraphics[width=0.31\textwidth]{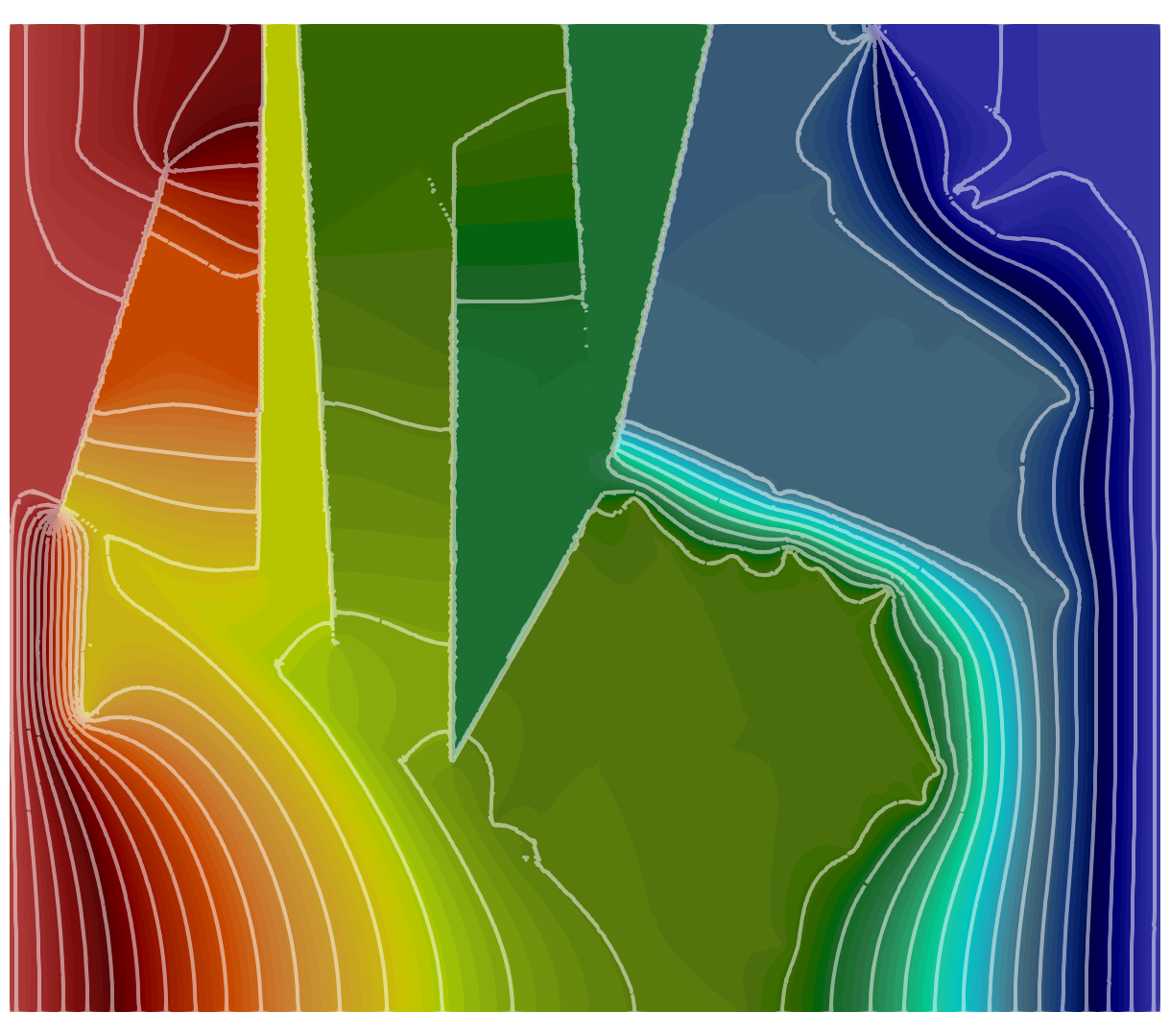}
\includegraphics[width=0.31\textwidth]{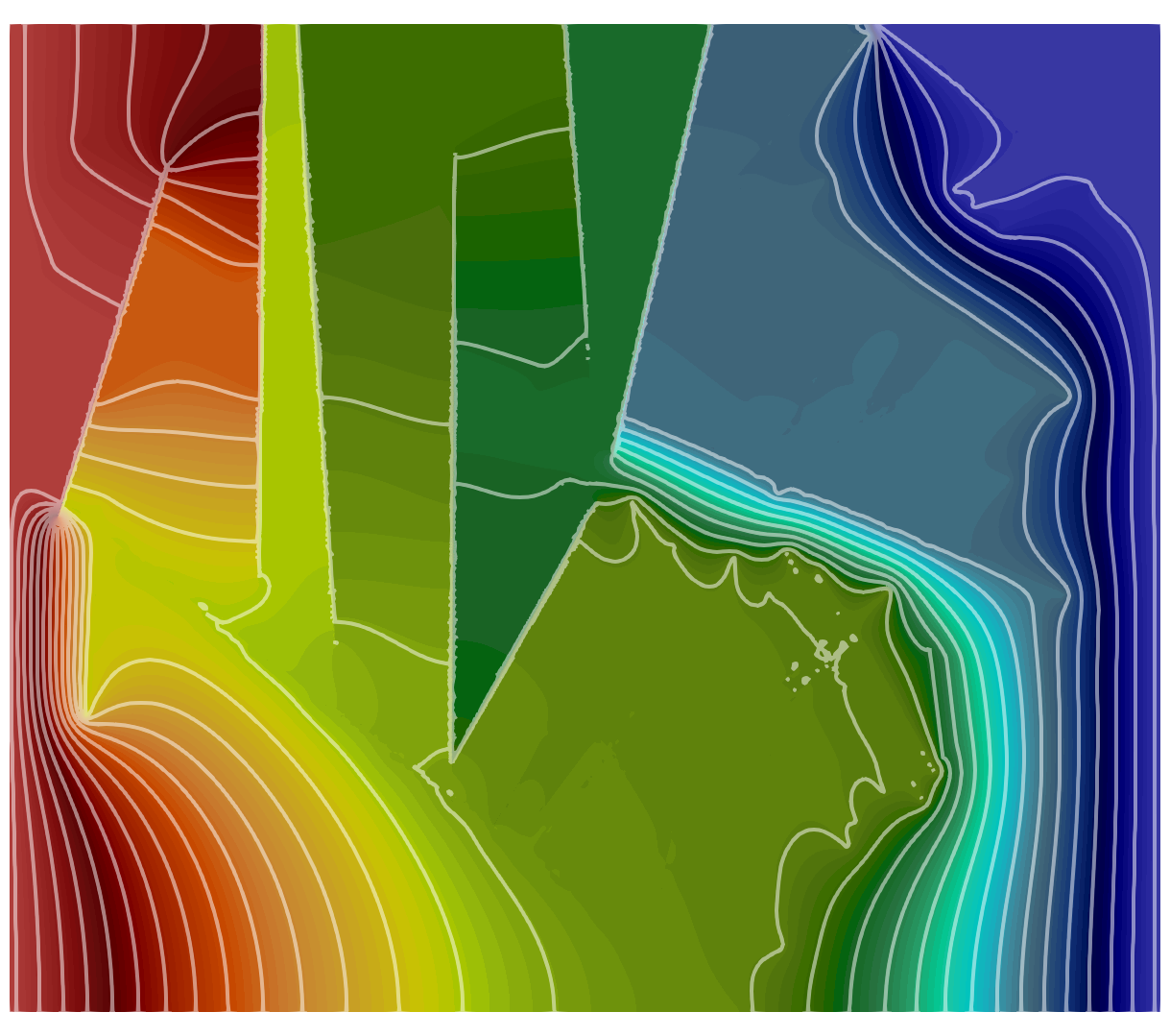}
\includegraphics[width=0.31\textwidth]{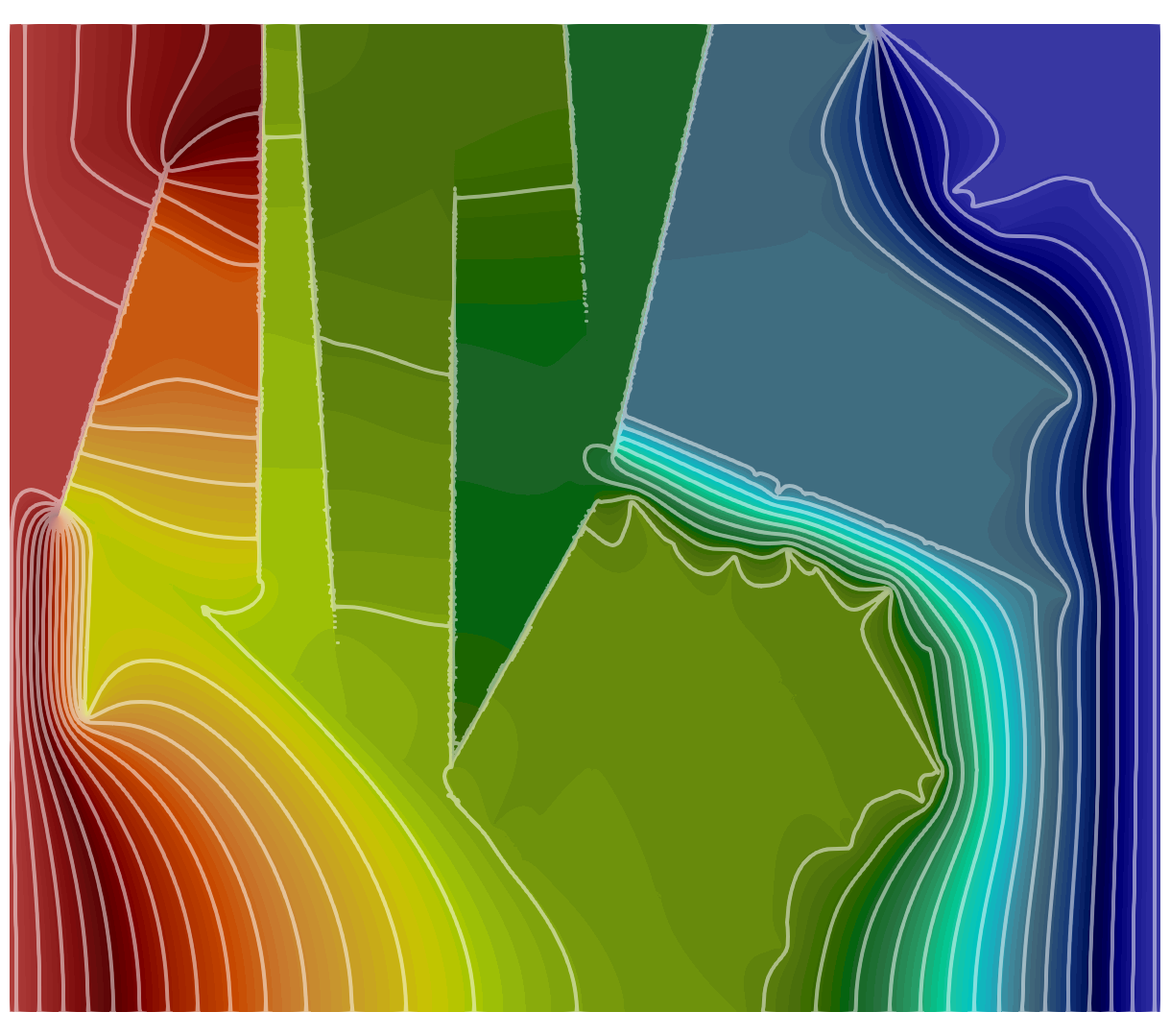}
\caption{\textbf{Example 3.} Pressure contours for $k=0$ (left), $k=1$ (middle), and $k=2$ (right). 
Top row: case (a). Middle row: case (b). Bottom row: case (c).  
Color range: (0,1).
Thirty uniform contour lines 
from 0.01 to 0.99. 
}
\label{fig:ex3-cont}
\end{figure}

\subsection*{Example 4: single fracture in 3D}
This is the first 3D benchmark case proposed in \cite{Berre_2021}.
The matrix domain $\Omega=(0, 100)\times (0,
100)\times (0, 100)$ which is crossed by a conductive planar fracture $\Gamma_1$
connected by the points $(0,0,80)$, $(100,0,20)$, $(100,100,20)$, $(0,100,80)$
with a thickness of  $\epsilon=10^{-2}$.
The matrix permeability is heterogeneous and is taken to be $\bld K_m = 10^{-6}$ when $z\ge 10$ and 
$\bld K_m=10^{-5}$ when $z<10$.
The fracture conductivity is $k_c = 0.1$ so that $\epsilon k_c= 10^{-3}$.
We apply the Dirichlet boundary conditions on the two boundaries
\[
\Gamma_{in}:=\{0\}\times(0,
100)\times(90, 100), 
\Gamma_{out}:=(0, 100)\times\{0\}
\times(0, 10), 
\]
where $p=4$ on $\Gamma_{in}$ and $p=1$ on $\Gamma_{out}$.
No flow boundary conditions is used on the rest of the domain boundary.

% We perform the simulation on a nondimensionalized domain 
% $\bar \Omega=(0,1)\times (0,1)\times (0,1)$ with 
% permeability $\bar{\bld K}_m = 0.1$ when $\bar z\ge 0.1$ and 
% $\bar{\bld K}_m=1$ when $\bar z <0.1$. 
% The effective permeability on the fracture is $\bar{\epsilon k}_c = 100$.

In this example, we apply the HDG scheme \eqref{hdg} with degree $k=0, 1,2$ on two uniform hexahedral meshes with mesh size $h=10$ (1000 cubic cells) and $h=5$ (8000 cubic cells)
where we take $s_c=2$ and $C_c=1$ for $k=0$, and 
$s_c=3$ and $C_c=1$ for $k=1,2$
in the stabilization parameters. Here the characteristic length in \eqref{stab} is $L=100$.
% and use the relative 
% mesh size $h_K/L$ instead of $h_K$ in \eqref{stab}.
% \begin{figure}[ht]
% \centering
% \includegraphics[width=0.31\textwidth]{figs/singleC0.png}
% \includegraphics[width=0.31\textwidth]{figs/singleC1.png}
% \includegraphics[width=0.31\textwidth]{figs/singleC2.png}
% \includegraphics[width=0.31\textwidth]{figs/singleF0.png}
% \includegraphics[width=0.31\textwidth]{figs/singleF1.png}
% \includegraphics[width=0.31\textwidth]{figs/singleF2.png}
% \caption{\textbf{Example 4.} Pressure contour for $k=0$ (left), 
% $k=1$ (middle), and $k=2$ (right).
% Top row: coarse mesh (1000 cubic cells).
% Bottom row: fine mesh (8000 cubic cells).
% Color range: (1, 4).
%  Thirty uniform contour lines 
% from 1 to 4. 
% }
% \label{fig:ex4-cont}
% \end{figure}
We plot in Figure~\ref{fig:ex4-cut} pressure along the diagonal line $(0,0,0)-(100,100,100)$ together with reference data provided in \cite{Berre_2021} which is obtained from
the {\sf USTUTT-MPFA} method therein on a mesh with approximately 1 million matrix elements. Good agreement with reference solution is observed for $k=1$ and $k=2$. 
The results for $k=0$ is slightly off, but it improves as mesh refines.
\begin{figure}[ht]
\centering
\includegraphics[width=0.45\textwidth]{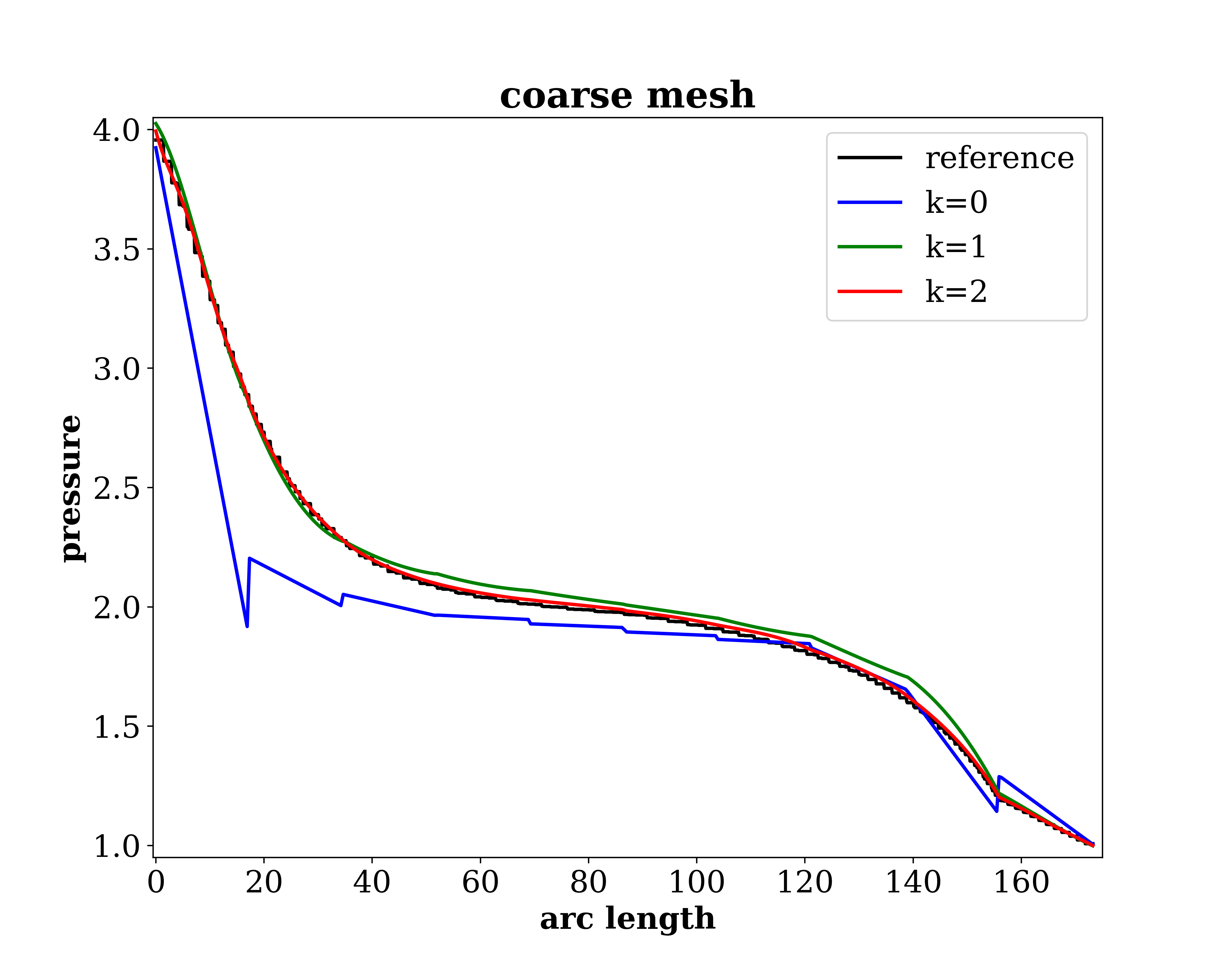}
\includegraphics[width=0.45\textwidth]{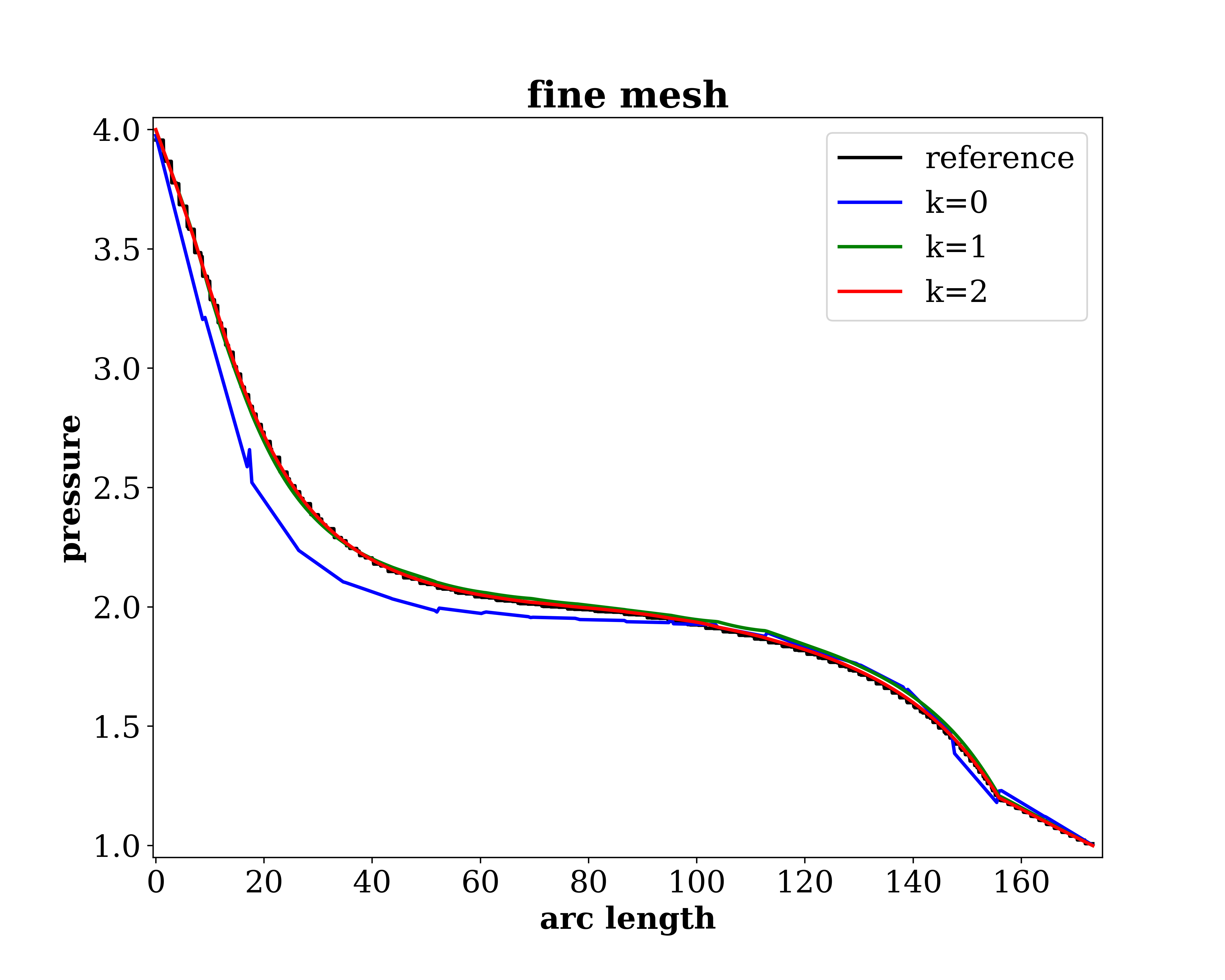}
\caption{\textbf{Example 4.} Pressure along line $(0,0,0)--(100,100,100)$. Here reference data is the result from the {\sf USTUTT-MPFA} scheme in \cite{Berre_2021} on a mesh with roughly 1 million matrix elements. 
}
\label{fig:ex4-cut}
\end{figure}

\subsection*{Example 5: Network with Small Features in 3D}
This is the third benchmark case proposed in \cite{Berre_2021}, in which small geometric features exist.
The domain is the box
$\Omega = (0, 1)\times (0, 2.25)\times
(0, 1)$, containing 8 planer conductive fractures;
% ; see more details of the setup in \cite{Berre_2021}.; 
see Figure \ref{fig:small}.
  \begin{figure}[ht]
  \centering
    \includegraphics[width=.8\textwidth]{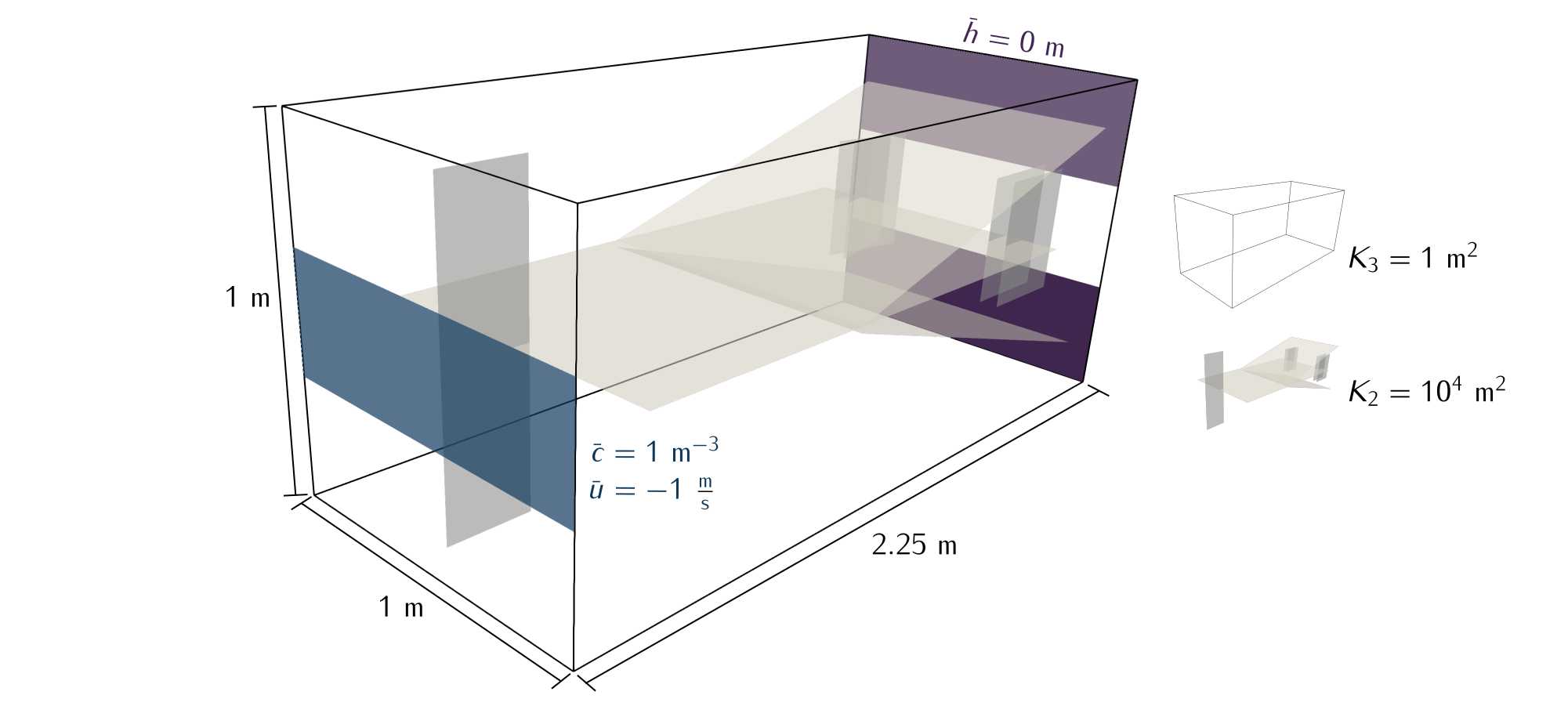}
    \caption{Example 5: Conceptual model and geometrical description
    of the domain.}
    \label{fig:small}
  \end{figure}
Homogeneous Dirichlet boundary condition is imposed on
the outlet boundary
\[
  \partial \Omega_{out}:=\{
    (x,y,z): 0<x<1, y = 2.25, z < 1/3 or z > 2/3\},
\]
inflow boundary condition $\bld u\cdot\bld n = -1$ is imposed
on the inlet boundary
\[
  \partial \Omega_{in}:=\{
    (x,y,z): 0<x<1, y = 0, 1/3<z < 2/3\},
\]
and no-flow boundary condition is imposed on the remaining boundaries.
The permeability in the matrix is $\mathbb K_m=1$, and that in the fracture is $k_c=10^4$. Fracture thickness is
$\epsilon=0.01$. 
The locations of these 8 fractures can be found in the 
git repository
\url{https://git.iws.uni-stuttgart.de/benchmarks/fracture-flow-3d} where a sample gmsh geometric file was also  provided.
This problem is very challenging due to the small intersections among the fractures exist.
The reported works in \cite{Berre_2021}
showed large discrepancies among the 16 participating methods. 

Here we run simulations for the scheme \eqref{hdg} with 
$k=0,1,2$ on two meshes: a fitted mesh with about 148k tetrahedral cells obtained from the above mentioned gmsh file with maximal mesh size $h\approx 0.074$ and minimal mesh size $h\approx 0.01$, 
and an unfitted mesh with about 132k tetrahedral cells obtained from local mesh refinements of a coarse uniform mesh near the fractures with maximal mesh size $h\approx 0.26$ and minimal mesh size $h\approx 0.026$.
The penalty parameters for polynomial degree $k=0,1,2$ on these two meshes are shown in Table~\ref{tab:2} below, where the characteristic length $L=2.25$.
\begin{table}[ht!]
\centering
\begin{tabular}{c|c c |c c}
& \multicolumn{2}{c|}{fitted mesh} &  
\multicolumn{2}{c}{unfitted mesh} \\\hline
 $k$    &  $C_c$ & $s_c$&  $C_c$ & $s_c$\\\hline
0 & 5 &  2  & 2.5 &  2\\[.1ex]
1 & 0.7  &  3  & 0.7 &  3\\[.1ex]
2 & 3.5 &  3  & 0.7 &  3\\[.2ex]
\end{tabular}
\caption{\textbf{Example 5.} Choice of the penalty parameters for different polynomial degree and meshes.}
\label{tab:2}
\end{table}
We first show the pressure contours along the plane $z=0.7$ which intersects with 6 fractures in Figure~\ref{fig:ex5-cont}, from which we observe 
large variations among the results on the two meshes, especially for $y\ge 1$ where the 
 fractures near the outlet starts to interact with the flow. 
This observation is in line with the findings in \cite{Berre_2021} where 
 significant difference among participating methods were reported for the pressure distribution along the center line 
$(0.5,1.1,0)-(0.5,1.1, 1.0)$, suggesting that the small features of the fracture network geometry may be not adequately resolved on these meshes.
We also plot the pressure distribution along this center line in Figure~\ref{fig:ex5-cut}, where the shaded region depicts the area between the 10th and the 90th percentile of the published results in \cite{Berre_2021} on similar meshes about $150k$ cells.
It is observed from this figure that our results on both meshes are within the range of the published results in \cite{Berre_2021}, 
and that the results on the fitted mesh may be more accurate as it is closer to the reference solution from {\sf USTUTT-MPFA} on a mesh with roughly 1 million cells.

% that 
% the case with $k=0$ on the unfitted mesh (top right) does not properly capture the vertical fracture near the inlet. 
% While the results are qualitatively similar on both meshes, we do observe large difference in the middle region where $y$

\begin{figure}[ht]
\centering
\includegraphics[width=0.47\textwidth]{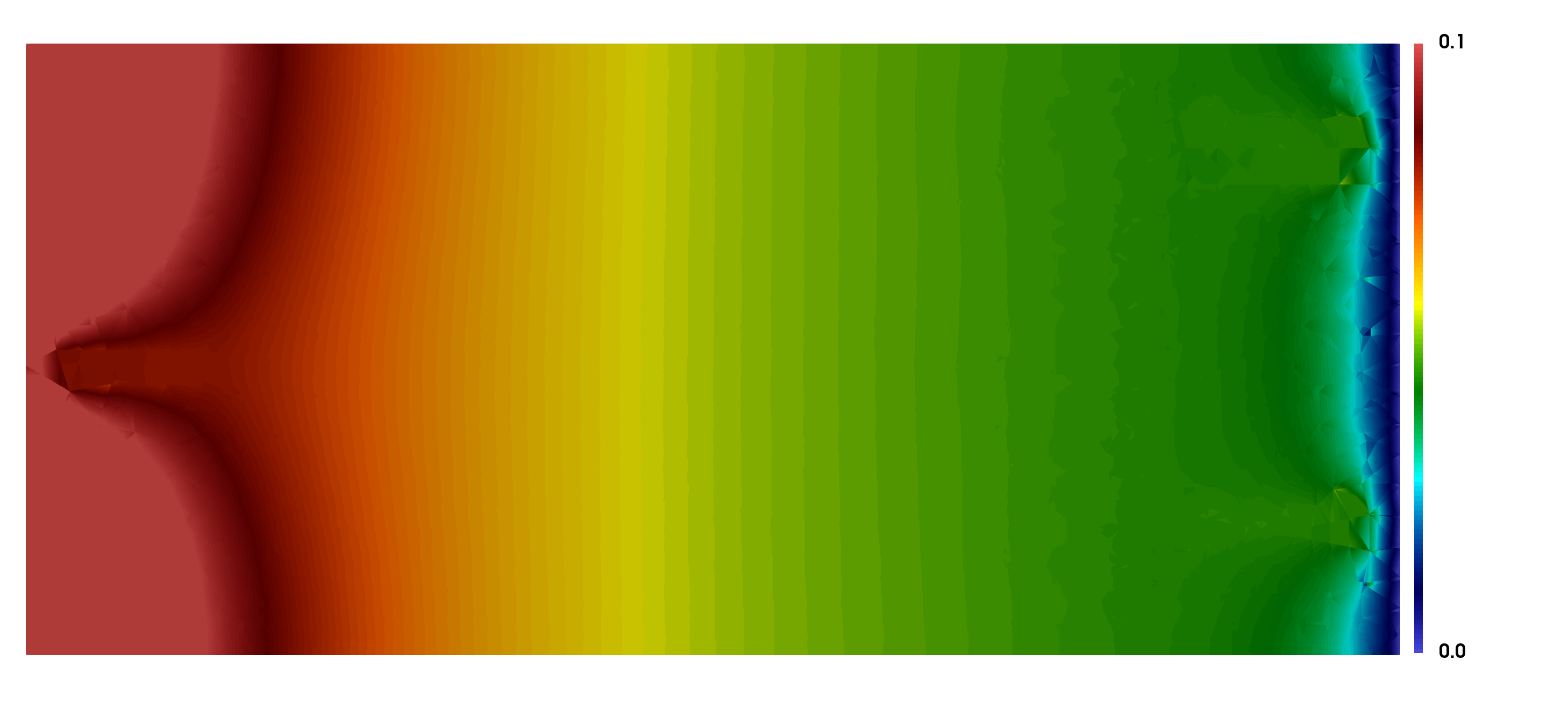}
\includegraphics[width=0.47\textwidth]{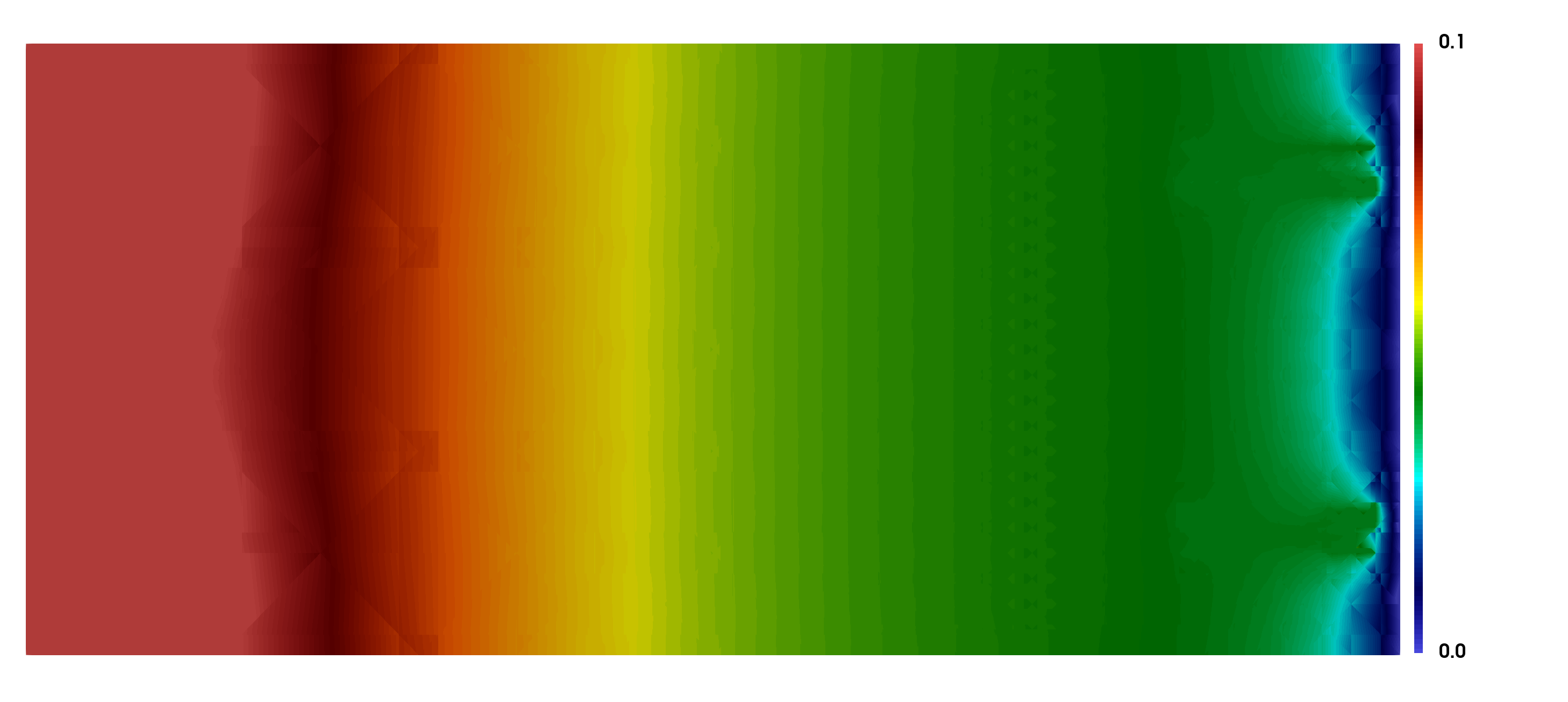}
\includegraphics[width=0.47\textwidth]{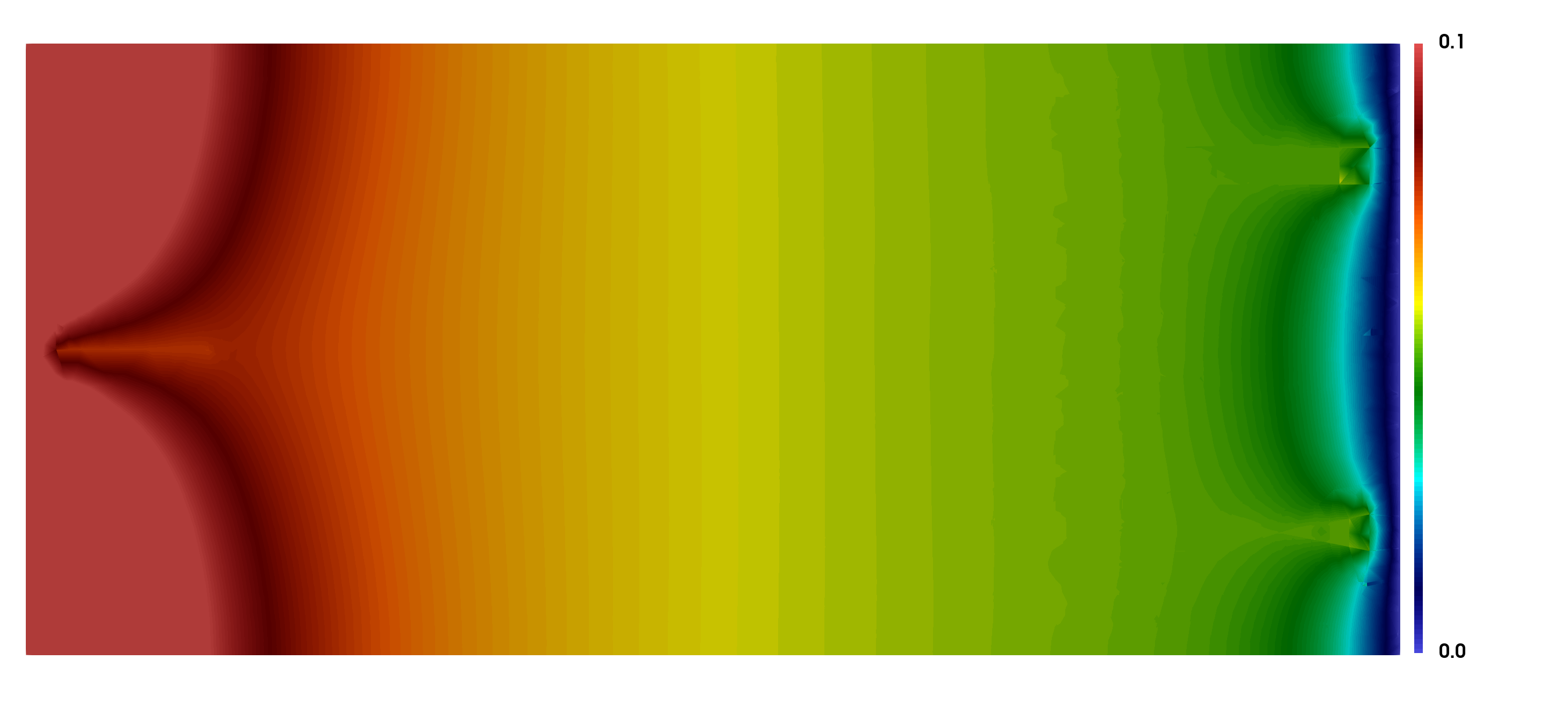}
\includegraphics[width=0.47\textwidth]{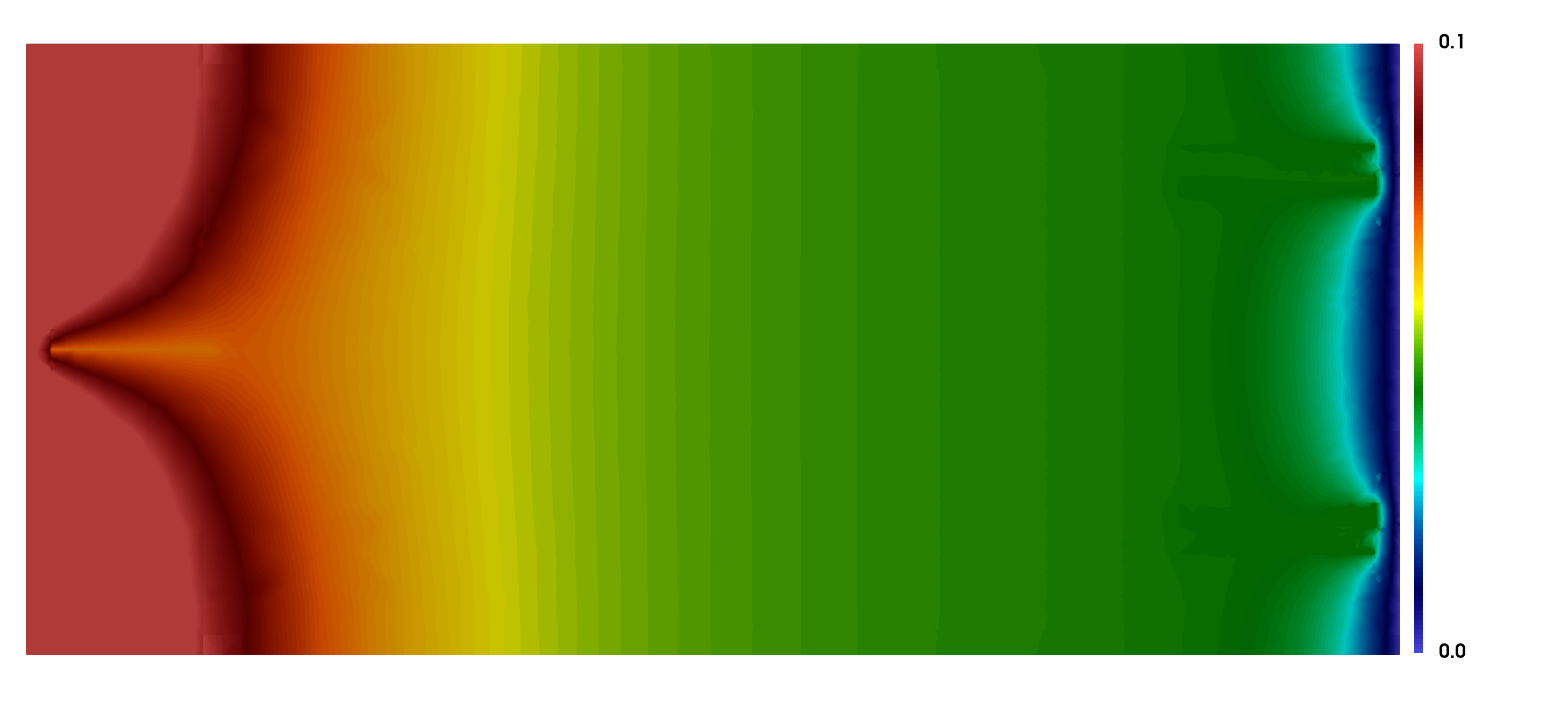}
\includegraphics[width=0.47\textwidth]{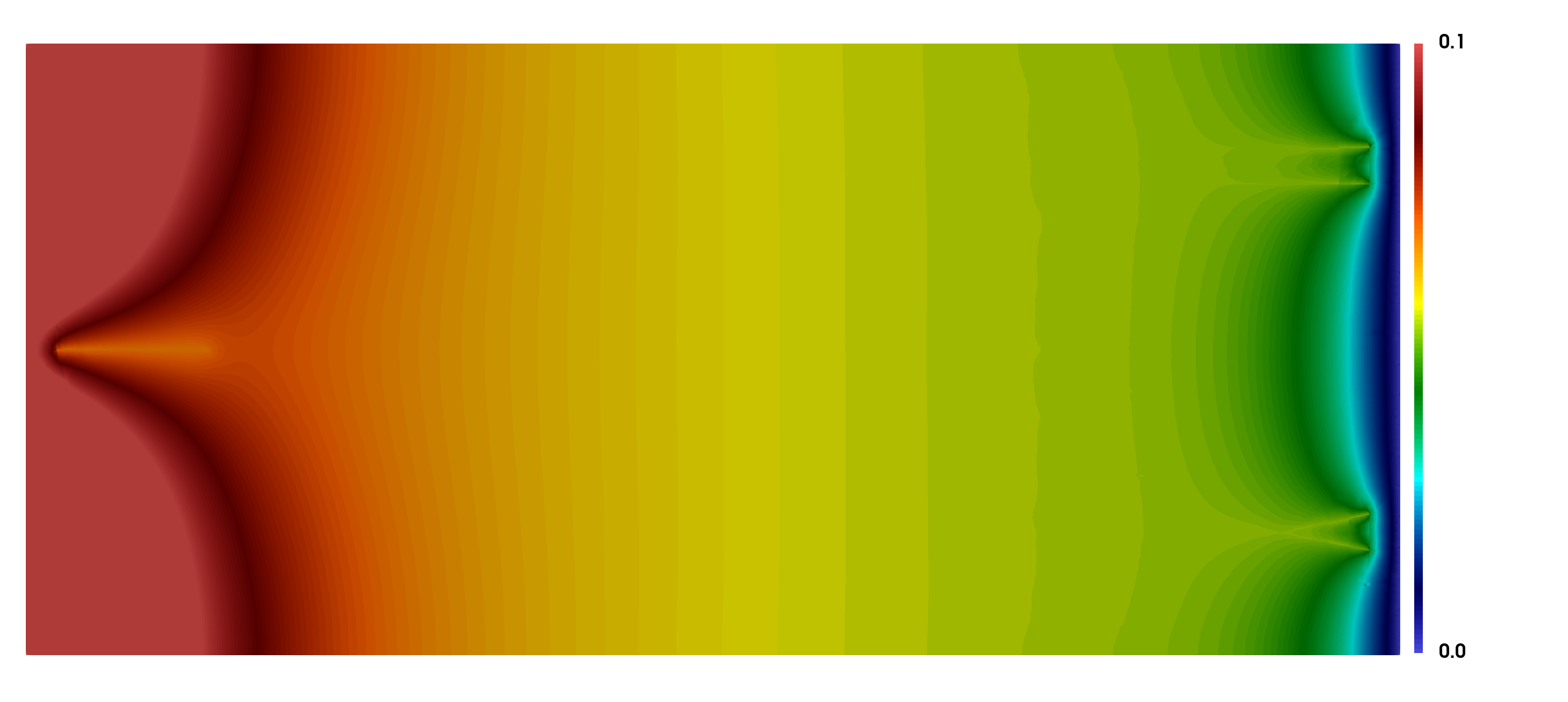}
\includegraphics[width=0.47\textwidth]{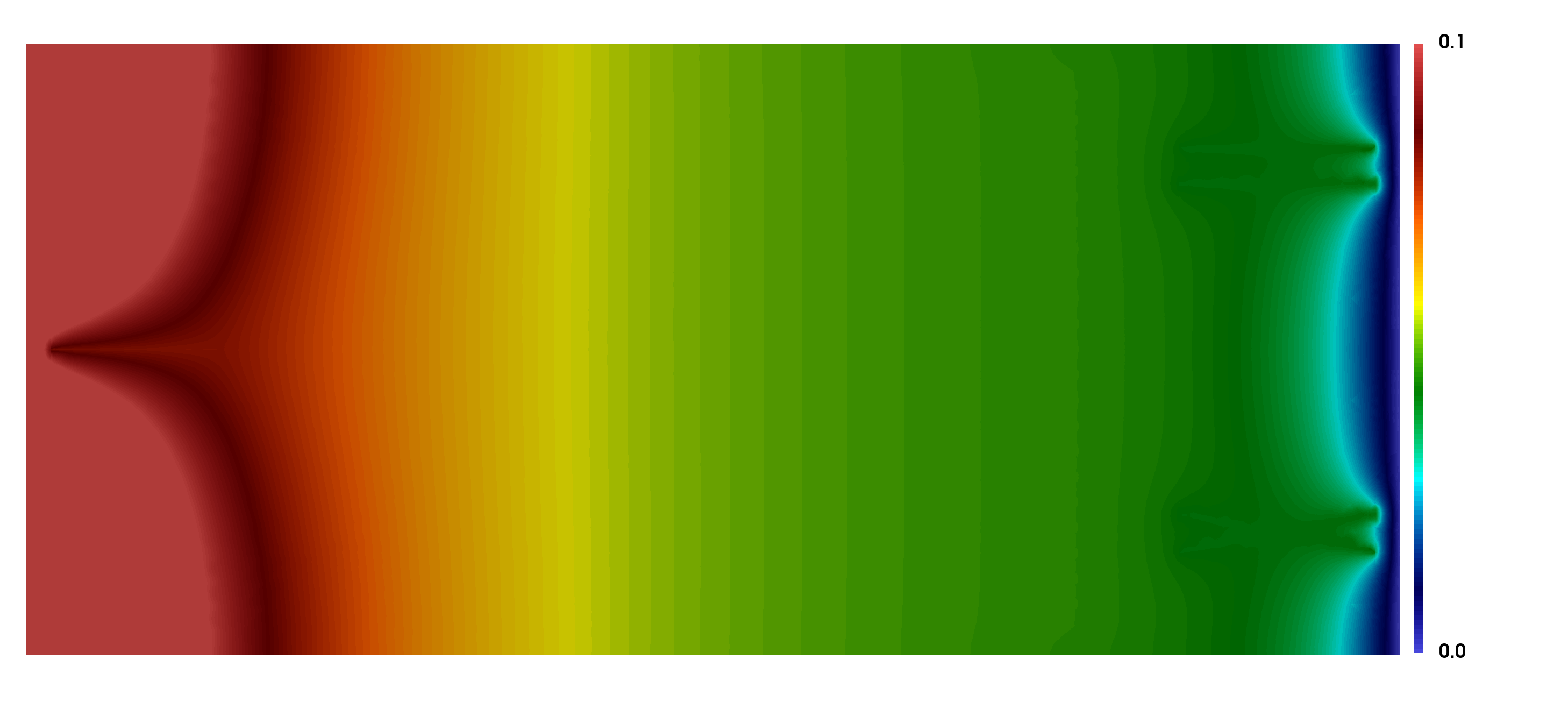}
\caption{\textbf{Example 5.} Pressure contour 
on the plane $z=0.7$ for $k=0$ (top), 
$k=1$ (middle), and $k=2$ (bottom).
Left: fitted mesh (\~148k cells). Right: unfitted mesh (\~132k cells).
}
\label{fig:ex5-cont}
\end{figure}

\begin{figure}[ht]
\centering
\includegraphics[width=0.45\textwidth]{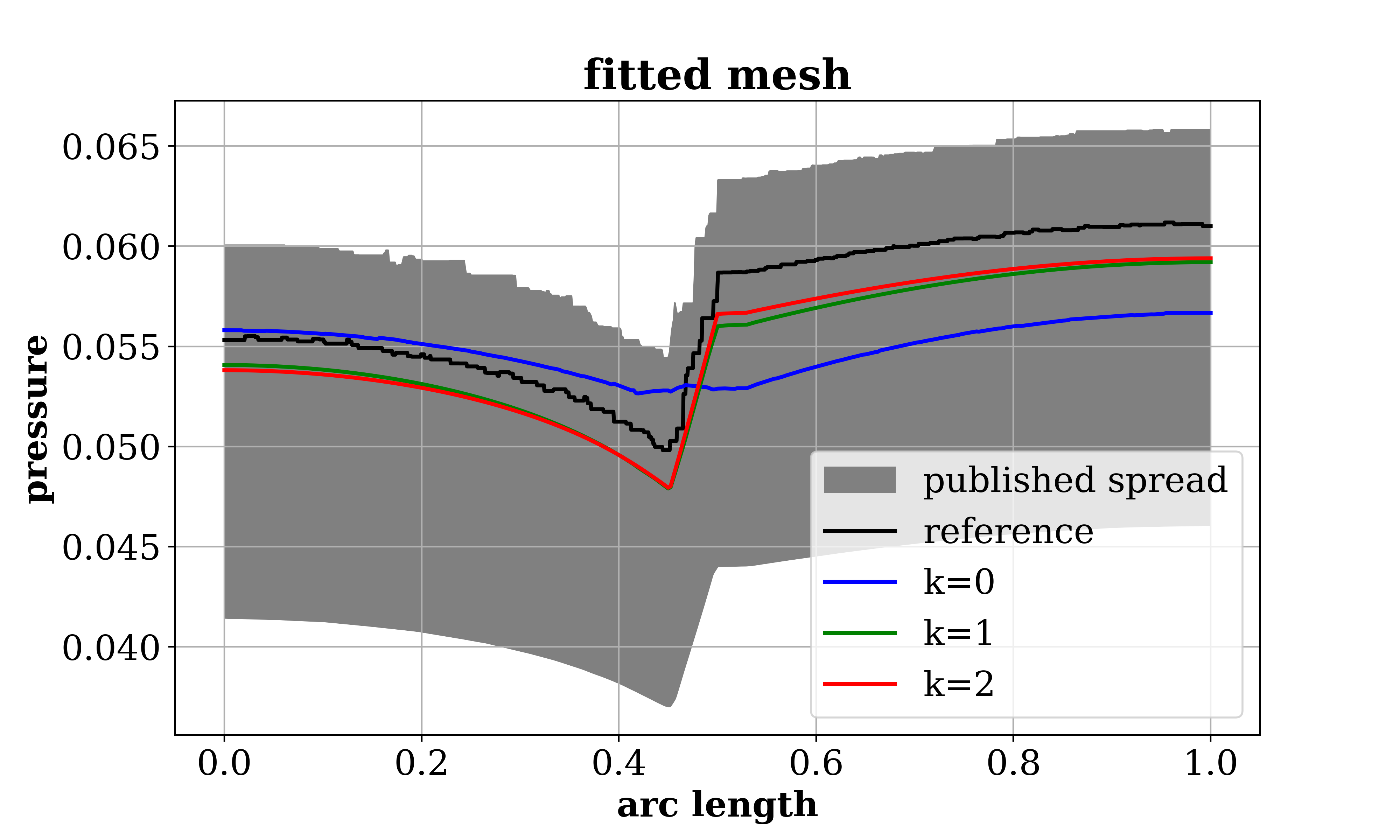}
\includegraphics[width=0.45\textwidth]{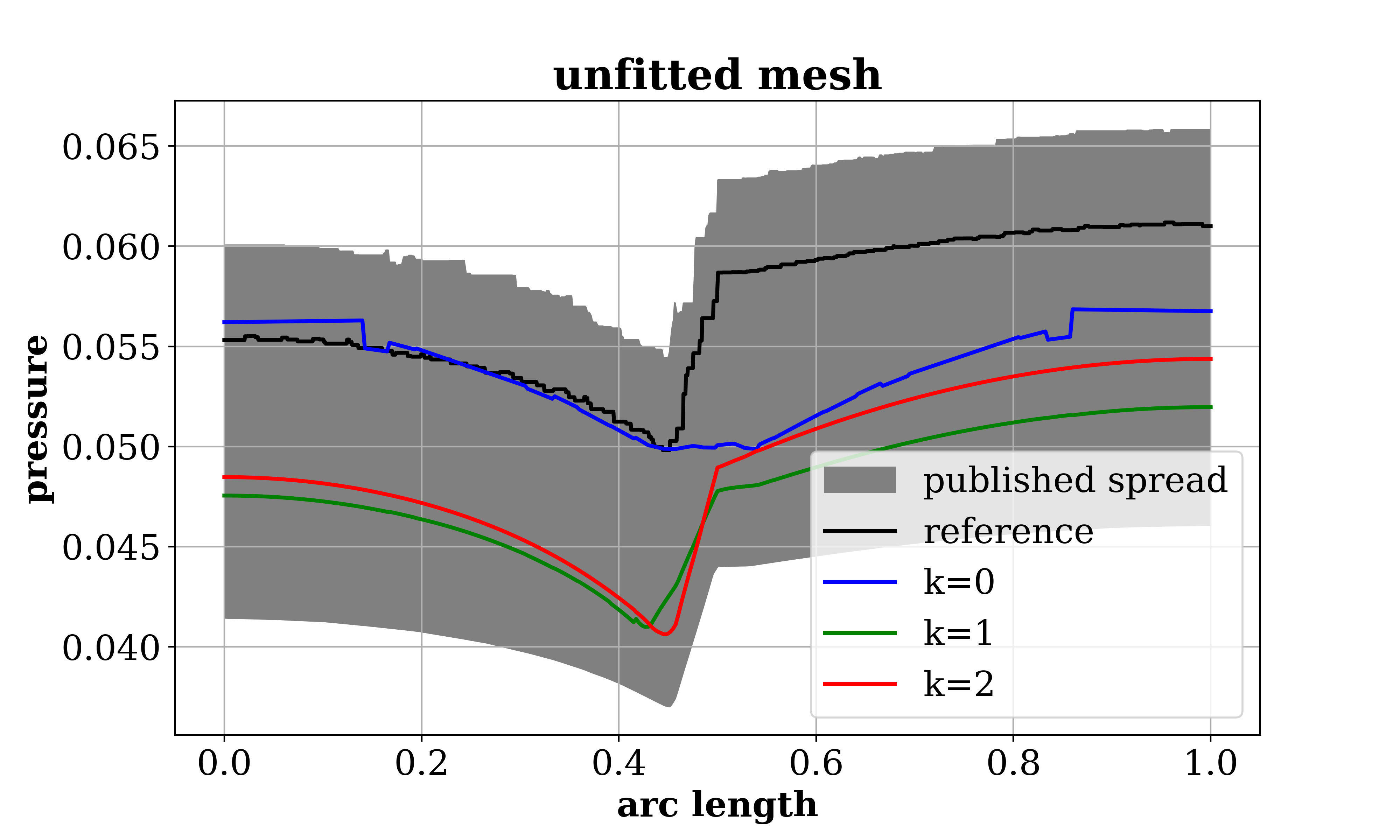}
\caption{\textbf{Example 5.} Pressure along line $(0.5,1.1,0)--(0.5,1.1,1.0)$. Here reference data is the result from the {\sf USTUTT-MPFA} scheme in \cite{Berre_2021} on a mesh with roughly 1 million matrix elements.
The shaded region depicts the area between the 10th and the 90th percentile of the published results in \cite{Berre_2021} on similar meshes with about $150k$ cells.}
\label{fig:ex5-cut}
\end{figure}

\subsection*{Example 6: Field Case in 3D}
In our final numerical example, we consider a similar setting as the last benchmark case proposed in \cite{Berre_2021}.
The geometry is based on  a postprocessed outcrop from the island of Algerøyna, outside Bergen, Norway, which contains 52 fracture. The simulation domain is the box $\Omega = (-500, 350)\times
(100, 1500)\times (-100, 500)$.
The fracture geometry is depicted in Figure \ref{fig:field}.
Homogeneous Dirichlet boundary condition is imposed on
the outlet boundary
\[
  \partial \Omega_{out}:=
  \underbrace{
  \{-500\}\times(100, 400)\times (-100, 100)}_{\partial\Omega_{out,0}}
\;\cup\;
\underbrace{
  \{350\}\times(100, 400)\times (-100, 100)}_{\partial\Omega_{out,1}}
\]
uniform unit inflow $\bld u\cdot \bld n = 1$
is imposed on the inlet boundary
\[
  \partial \Omega_{in}:=
  \underbrace{\{-500\}\times(1200, 1500)\times (300, 500)}_{\partial\Omega_{in,0}}
\;\cup\;
\underbrace{
(-500, -200)\times \{1500\}\times (300, 500)}_{\partial\Omega_{in, 1}}.
\]
  \begin{figure}[ht]
  \centering
    \includegraphics[width=.8\textwidth]{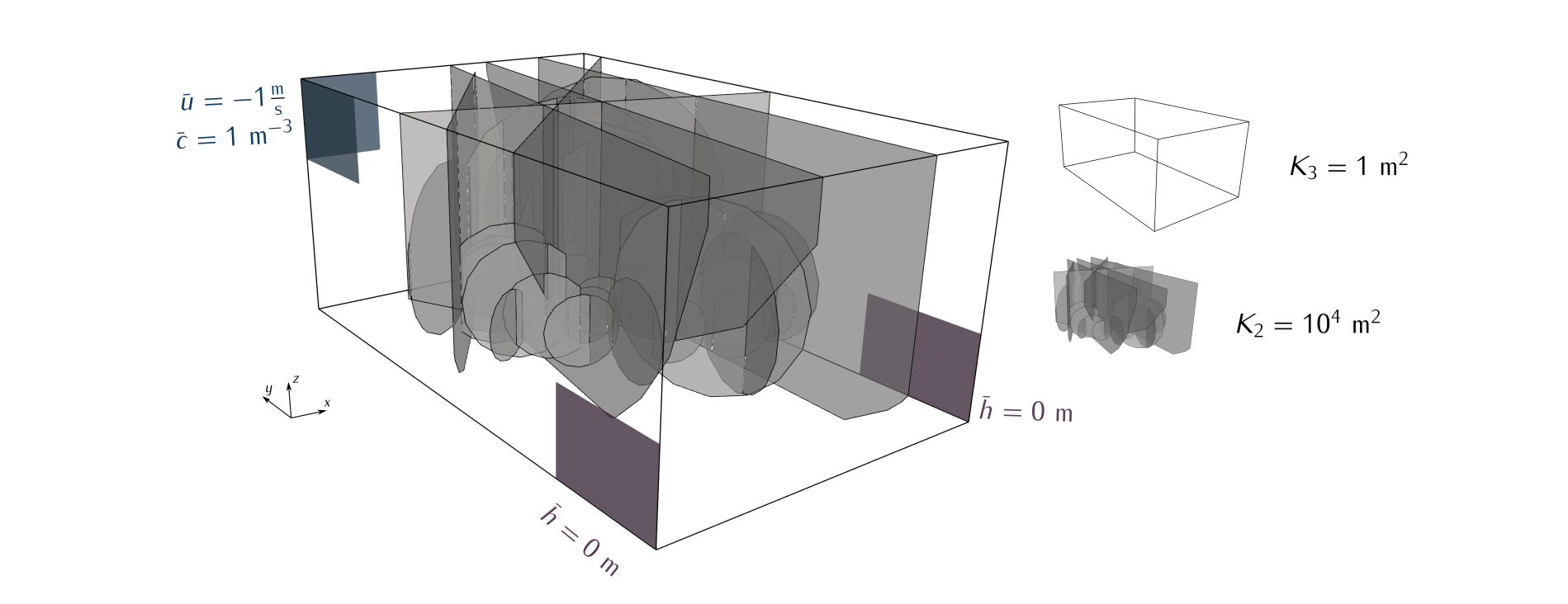}
    \caption{\textbf{Example 6.}  Conceptual model and geometrical description
    of the domain.}
    \label{fig:field}
  \end{figure}
Matrix permeability is $\mathbb K_m = 1$, and 
fracture thickness is $\epsilon = 10^{-2}$.
Similar to Example 3 in 2D, we consider three subcases: (a) all conductive fractures with 
permeability $k_c=10^{4}$, (b) all blocking fractures with permeability $k_b=10^{-4}m^2$, and (c) 2 blocking fractures with $k_b=10^{-4}$ and 50 conductive fractures with
$k_c=10^{4}$. 
Location of the blocking/conductive fractures for case (c) are marked in red/blue in the right panel of Figure~\ref{fig:ex6-msh}.

We perform the method \eqref{hdg} on 
two unfitted meshes; see Figure~\ref{fig:ex6-msh} for the fine mesh. The coarse mesh contains $~27k$ tetrahedral cells which is obtained by performing two steps of local mesh refinement around the fractured cells of a uniform background mesh with $h\approx 200$.
And the fine mesh contains $~210k$ cells with $~57.6k$ fractured cells which spits
to $47.8$ conductive fractured cells and $9.8k$ blocking fractured cells for case (c), which is a further uniform refinement of that coarse mesh.  As in the previous examples, we take polynomial degree 
$k=0,1,2$.  
For the penalty parameters, we take $C_b=1$, $s_b=2$, 
$C_c=1$ and the scaling power $s_c=3$ for $k=1$ and $k=2$, and $s_c=1$ for $k=0$. 
Moreover, for the case $k=0$ the default choice of stabilization \eqref{stab} is too strong which leads to almost constant (zero) pressure approximations for case (a). Here we further reduce the stabilization for $k=0$
by a factor of $L=1400$.
\begin{figure}[ht]
\centering
\vspace{1pt}
\includegraphics[width=0.80\textwidth]{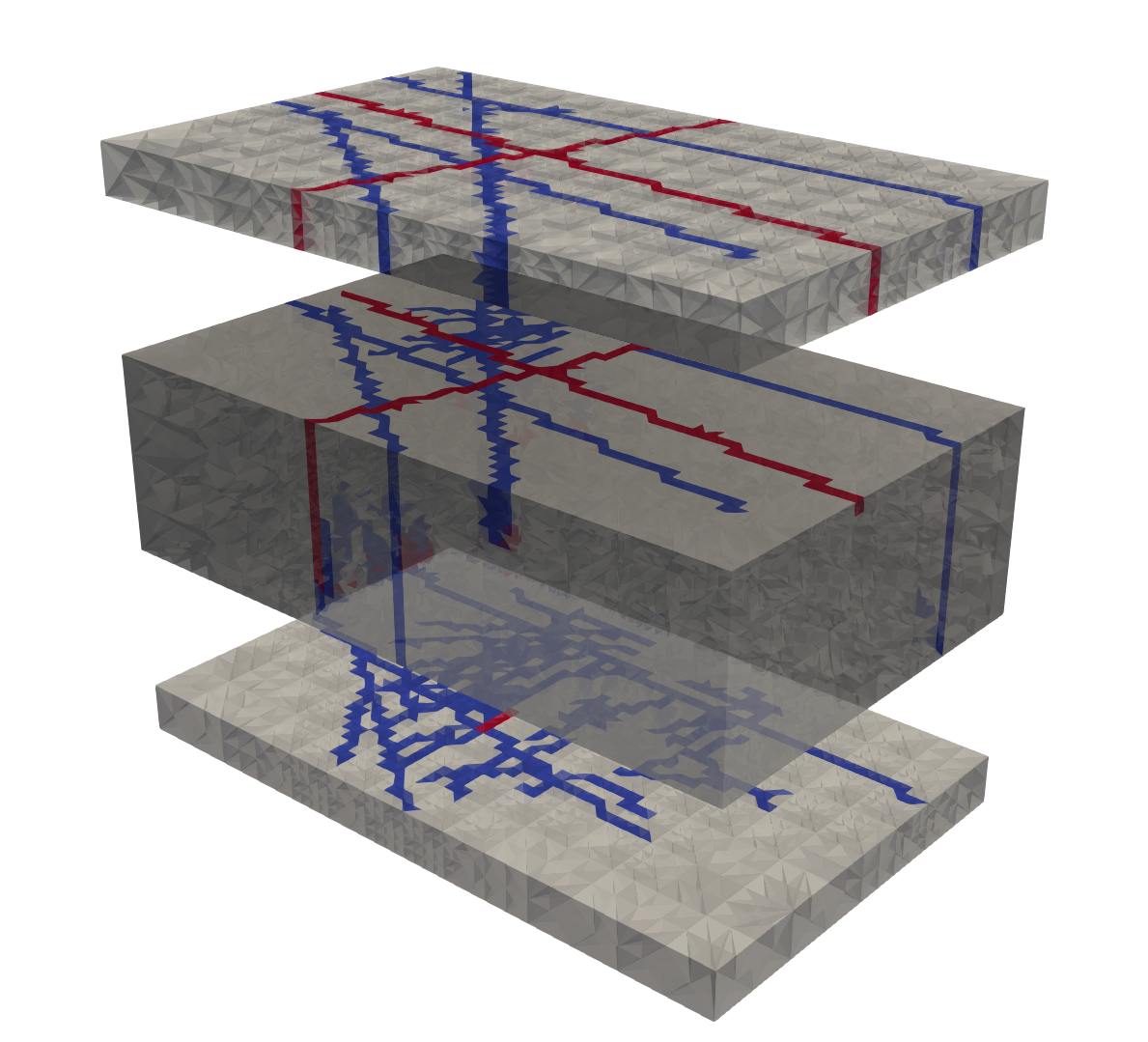}
\vspace{1pt}
\caption{\textbf{Example 6.}
Blocking (red) and conductive (blue) fractured cells on the fine mesh with 216816 cells for case (c).
Here the mesh is translated in the $z$-direction for $z<100$ and $z>400$ for better data visualization.}
\label{fig:ex6-msh}
\end{figure}

The pressure along the two diagonal lines
  $(-500, 100,-100)$--$(350, 1500, 500)$
  and
  $(350, 100,-100)$--$(-500, 1500, 500)$ are shown in Figures \ref{fig:ex6-cut1}-\ref{fig:ex6-cut3} for all three cases,
  where for case (a) the shaded region depicts the area between the 10th and the 90th percentile of the published results in \cite{Berre_2021} on fitted meshes with about $260k$ cells.
  It is observed that all methods produce qualitatively similar results for each case, even on the coarse mesh. 
  And our results for case (a) is consistent with the published results in \cite{Berre_2021}.
\begin{figure}[ht]
\centering
\includegraphics[width=0.45\textwidth]{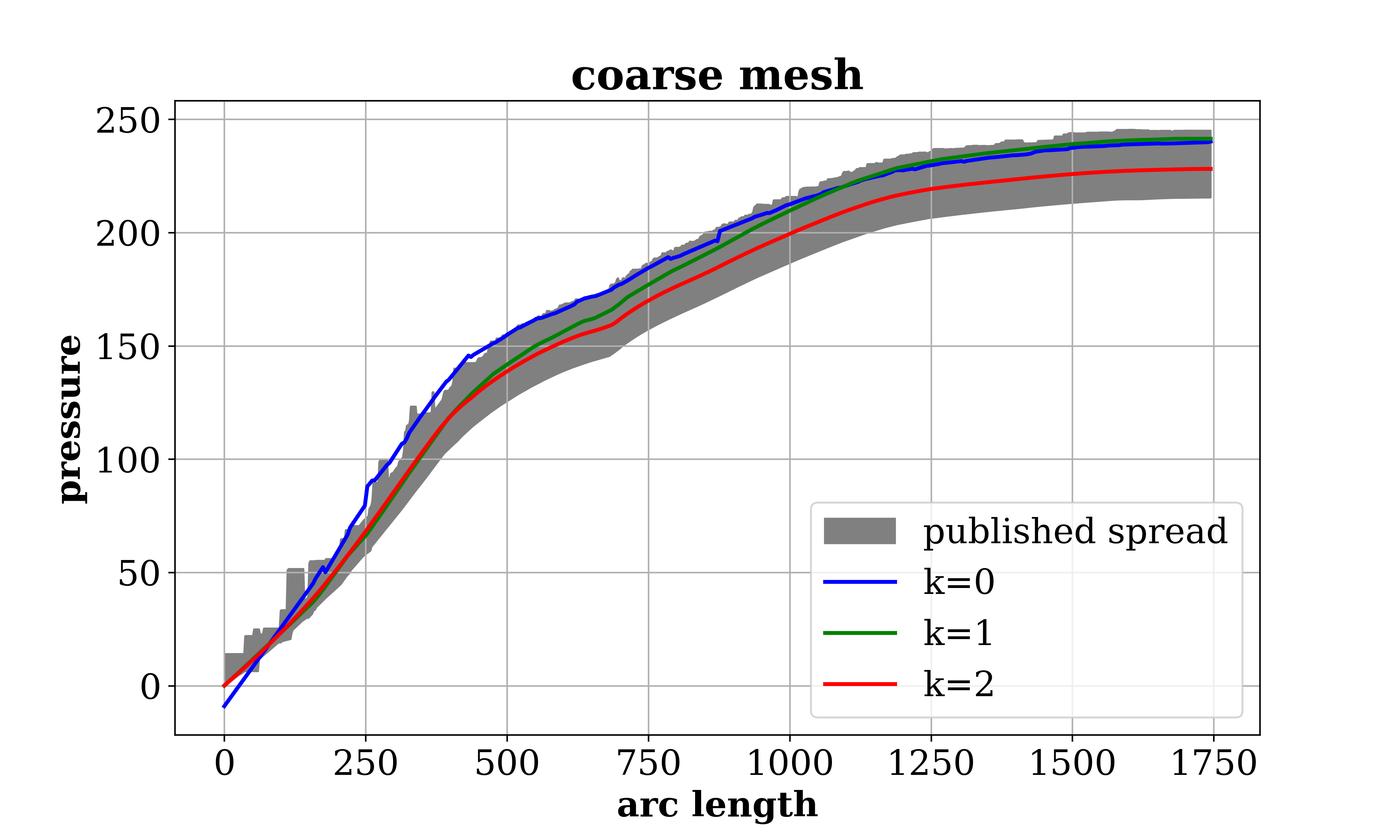}
\includegraphics[width=0.45\textwidth]{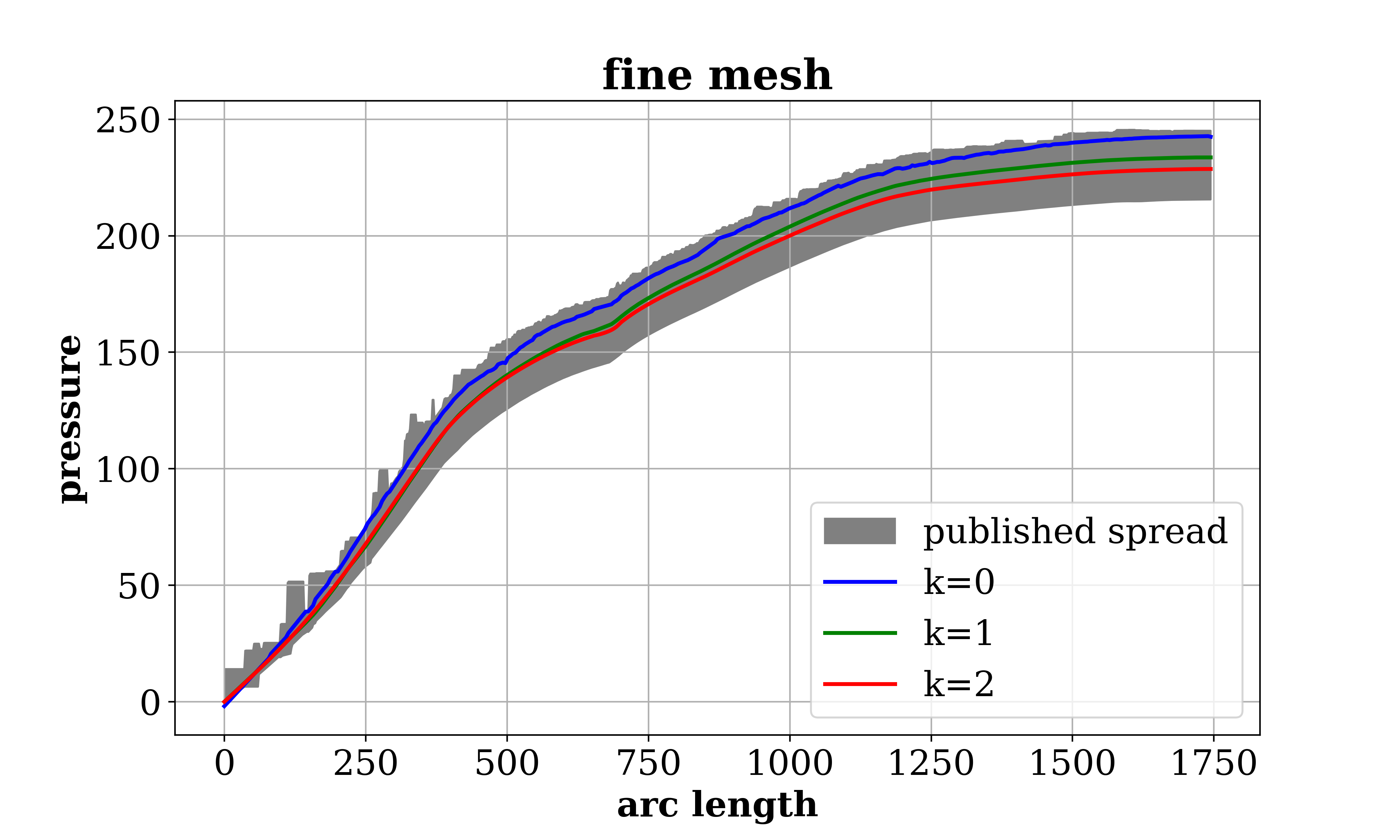}
\includegraphics[width=0.45\textwidth]{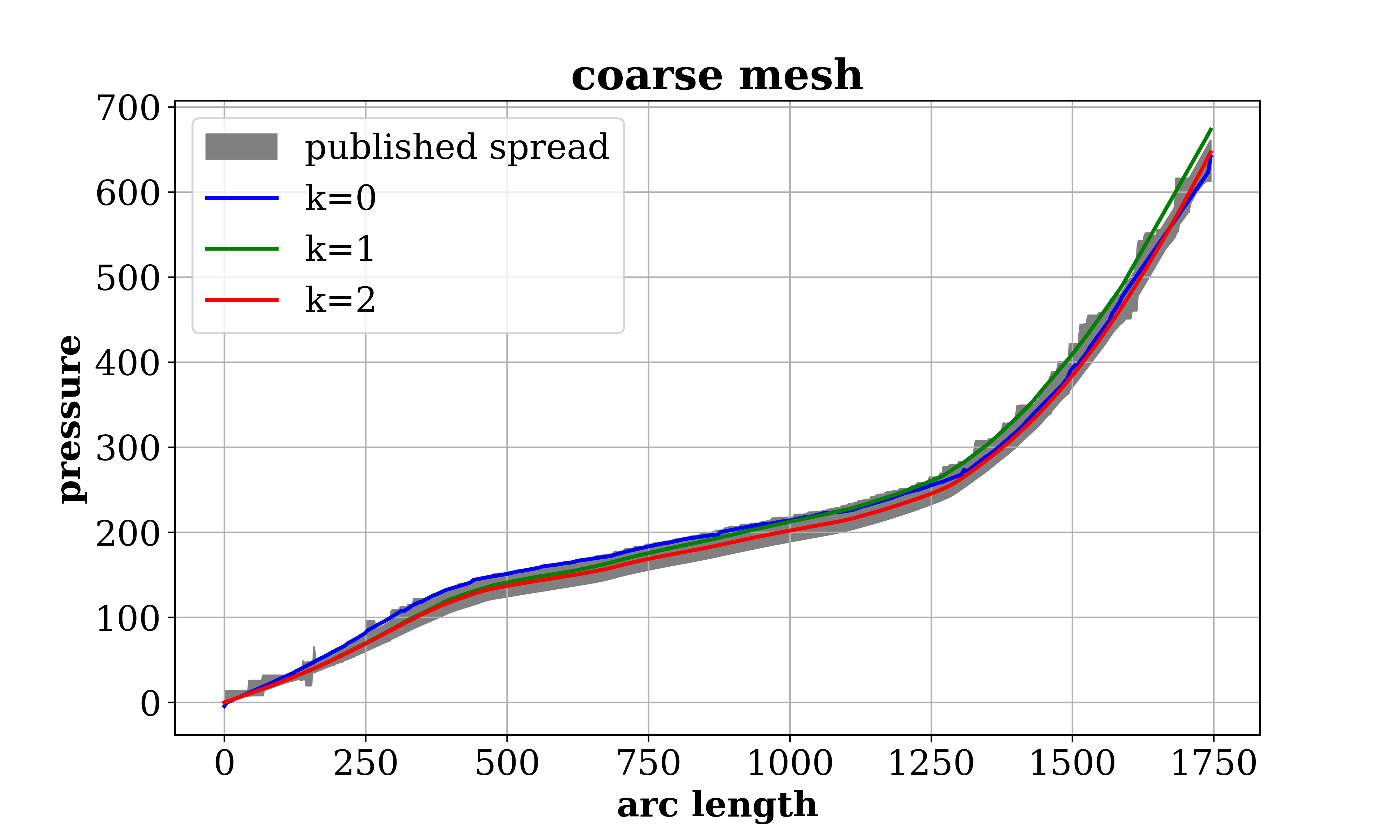}
\includegraphics[width=0.45\textwidth]{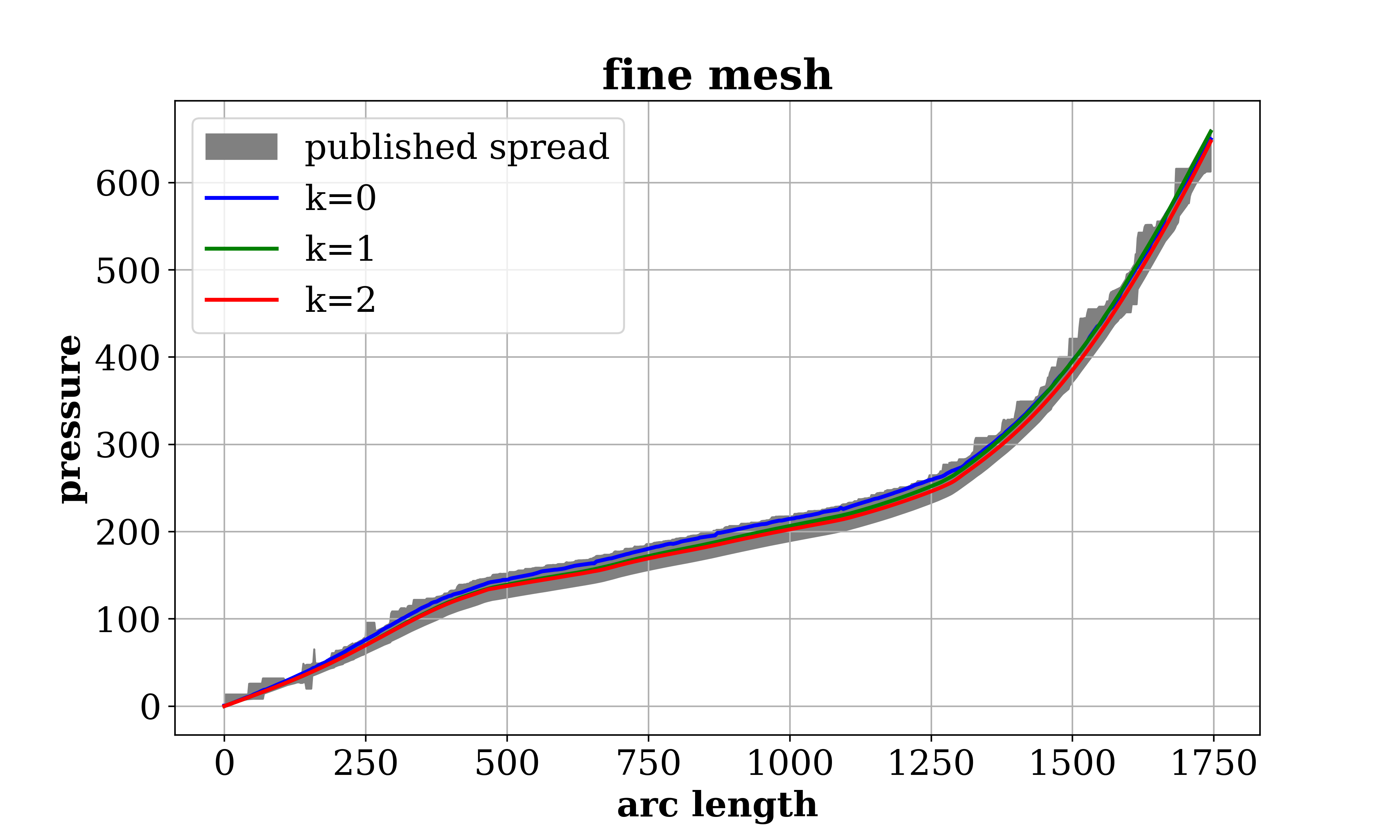}
\caption{\textbf{Example 6: Case (a).} Pressure along line $(-500, 100,-100)$--$(350, 1500, 500)$ (top) and line 
  $(350, 100,-100)$--$(-500, 1500, 500)$ (bottom). The shaded region depicts the area between the 10th and the 90th percentile of the published results in \cite{Berre_2021} on fitted meshes with about $260k$ cells.}
\label{fig:ex6-cut1}
\end{figure}
\begin{figure}[ht]
\centering
\includegraphics[width=0.45\textwidth]{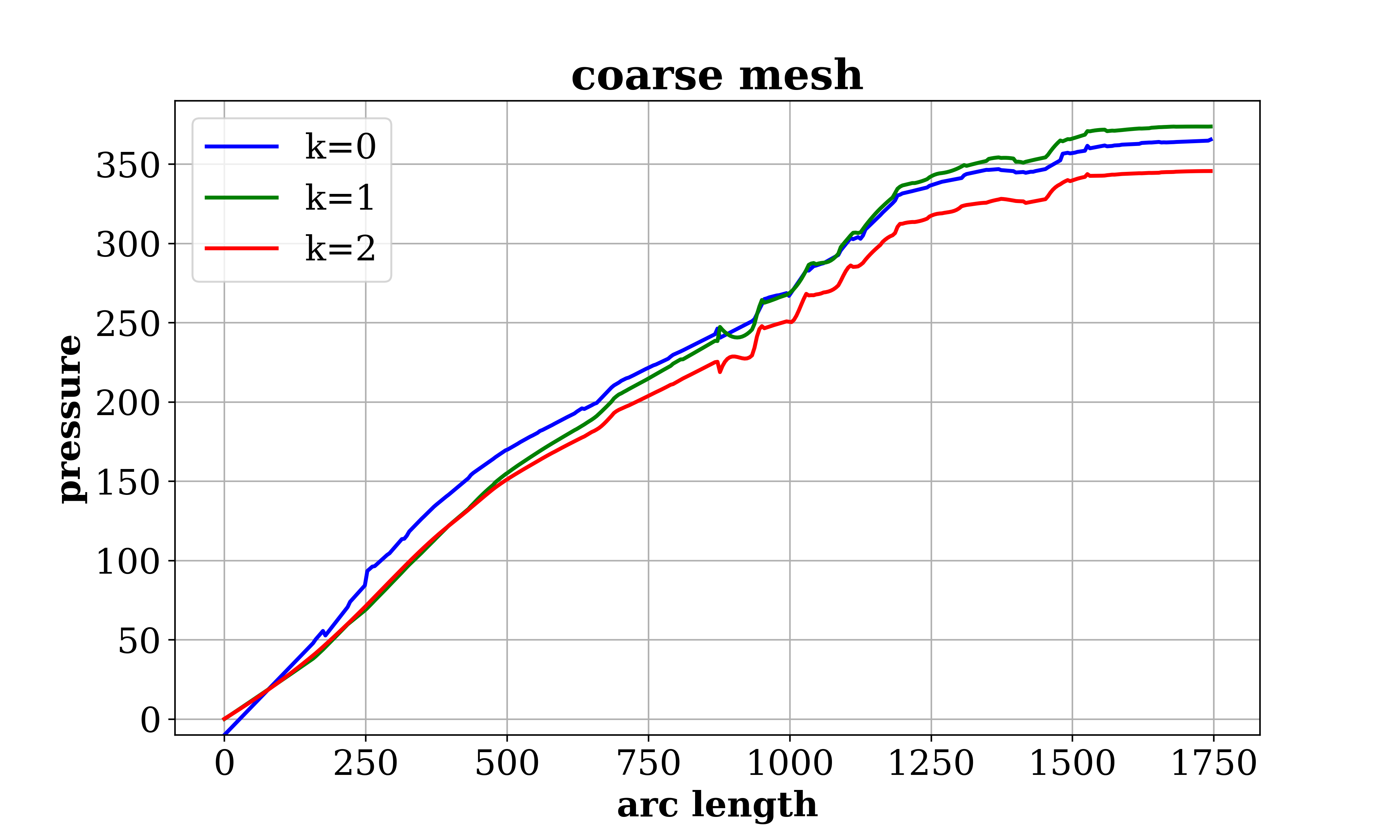}
\includegraphics[width=0.45\textwidth]{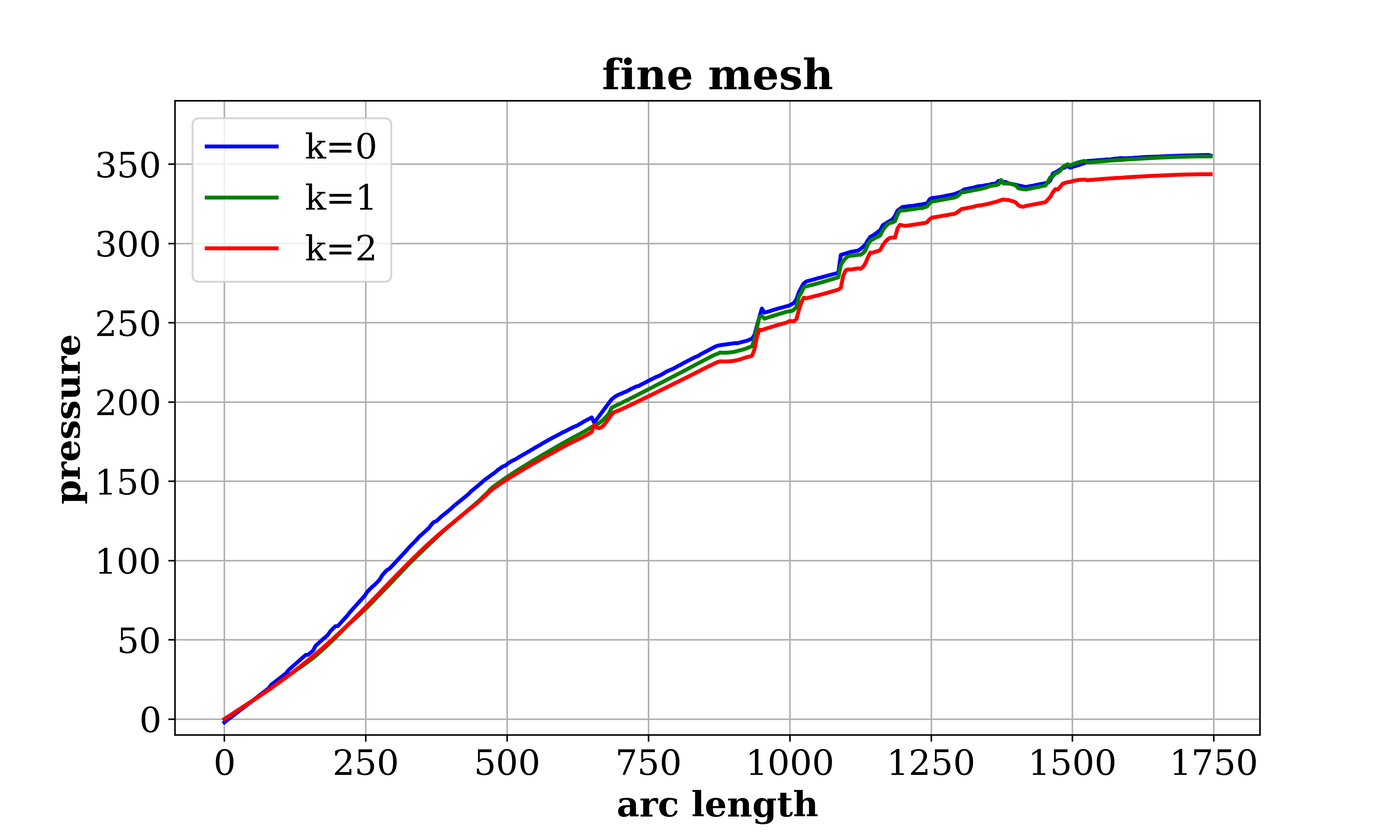}
\includegraphics[width=0.45\textwidth]{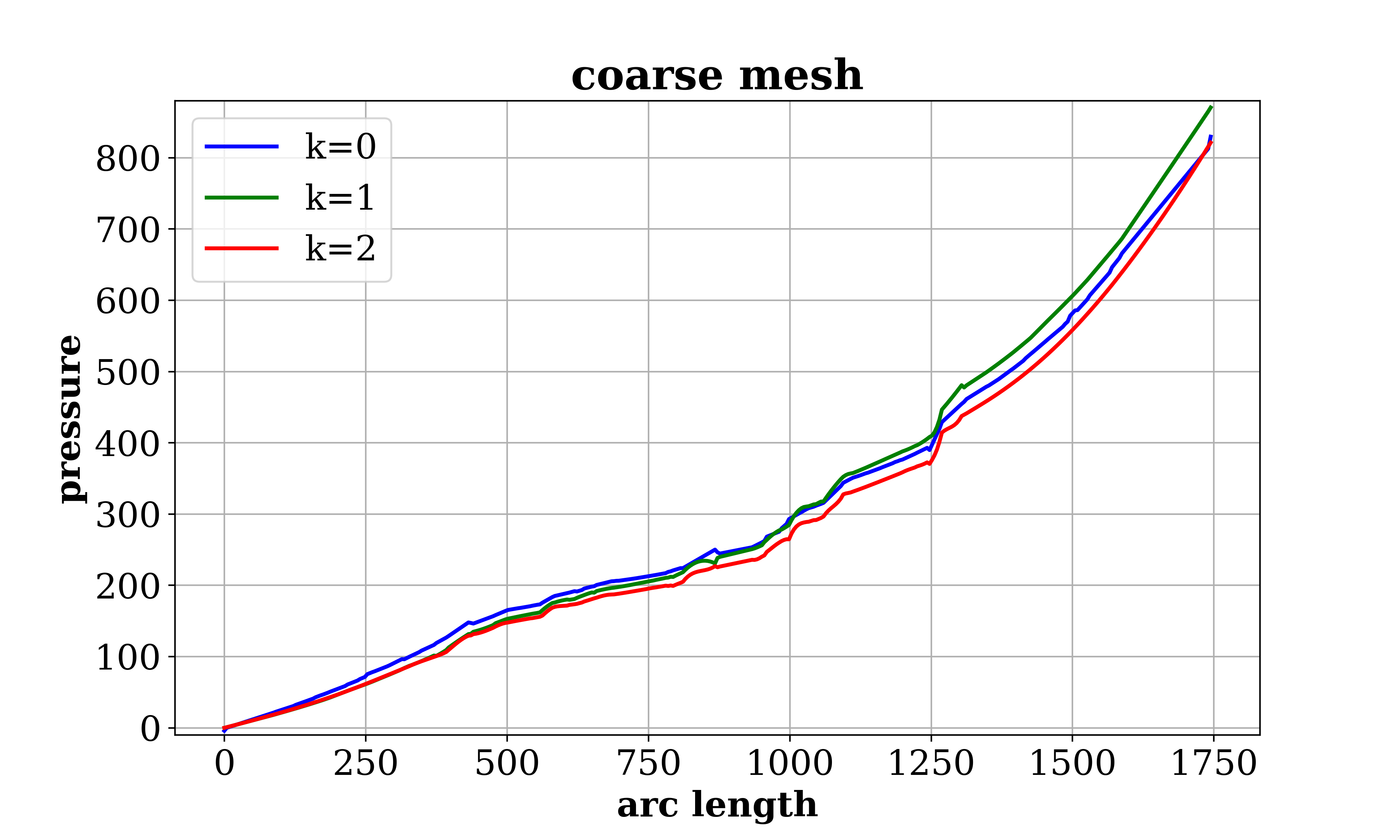}
\includegraphics[width=0.45\textwidth]{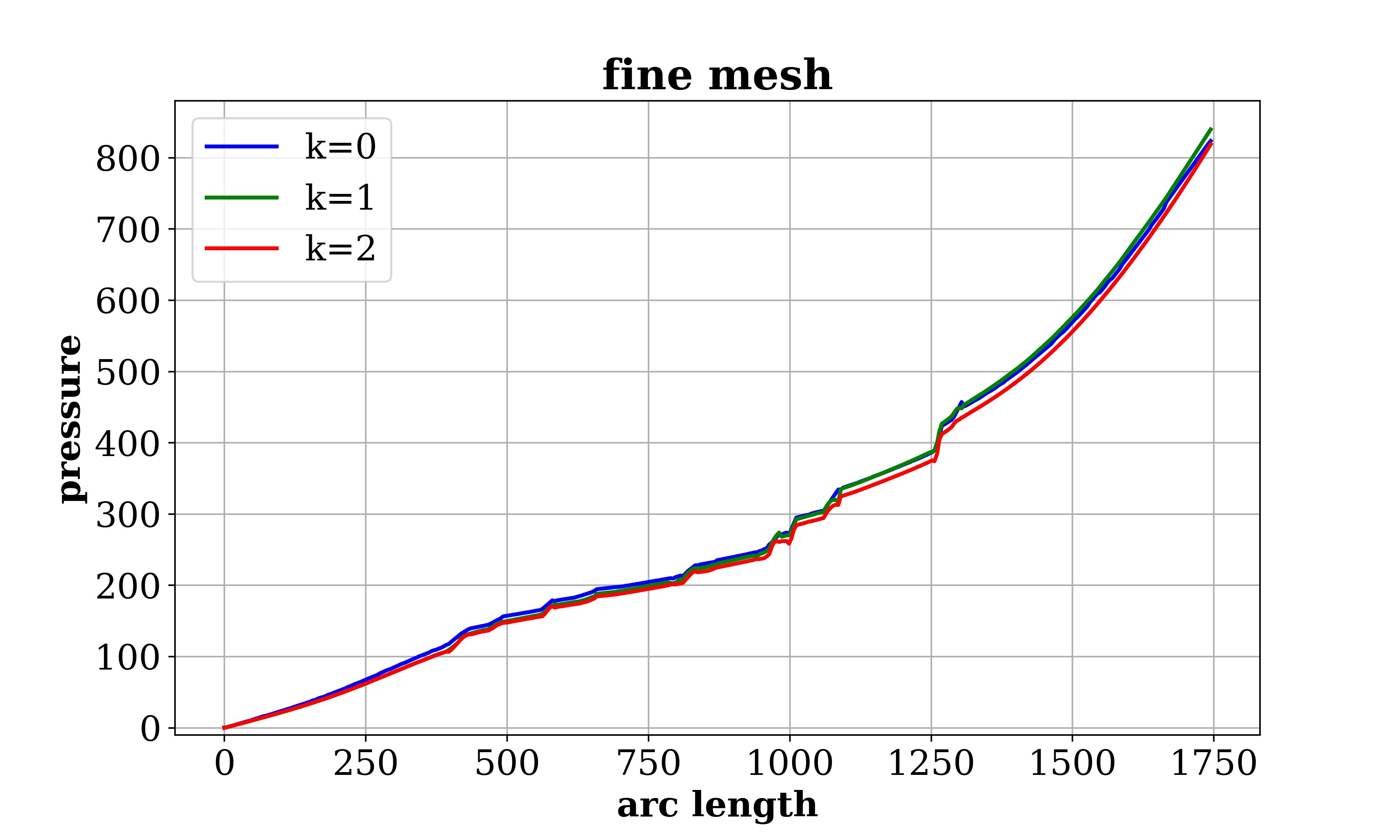}
\caption{\textbf{Example 6: Case (b).} Pressure along line $(-500, 100,-100)$--$(350, 1500, 500)$ (top) and line 
  $(350, 100,-100)$--$(-500, 1500, 500)$ (bottom).}
\label{fig:ex6-cut2}
\end{figure}

\begin{figure}[ht]
\centering
\includegraphics[width=0.45\textwidth]{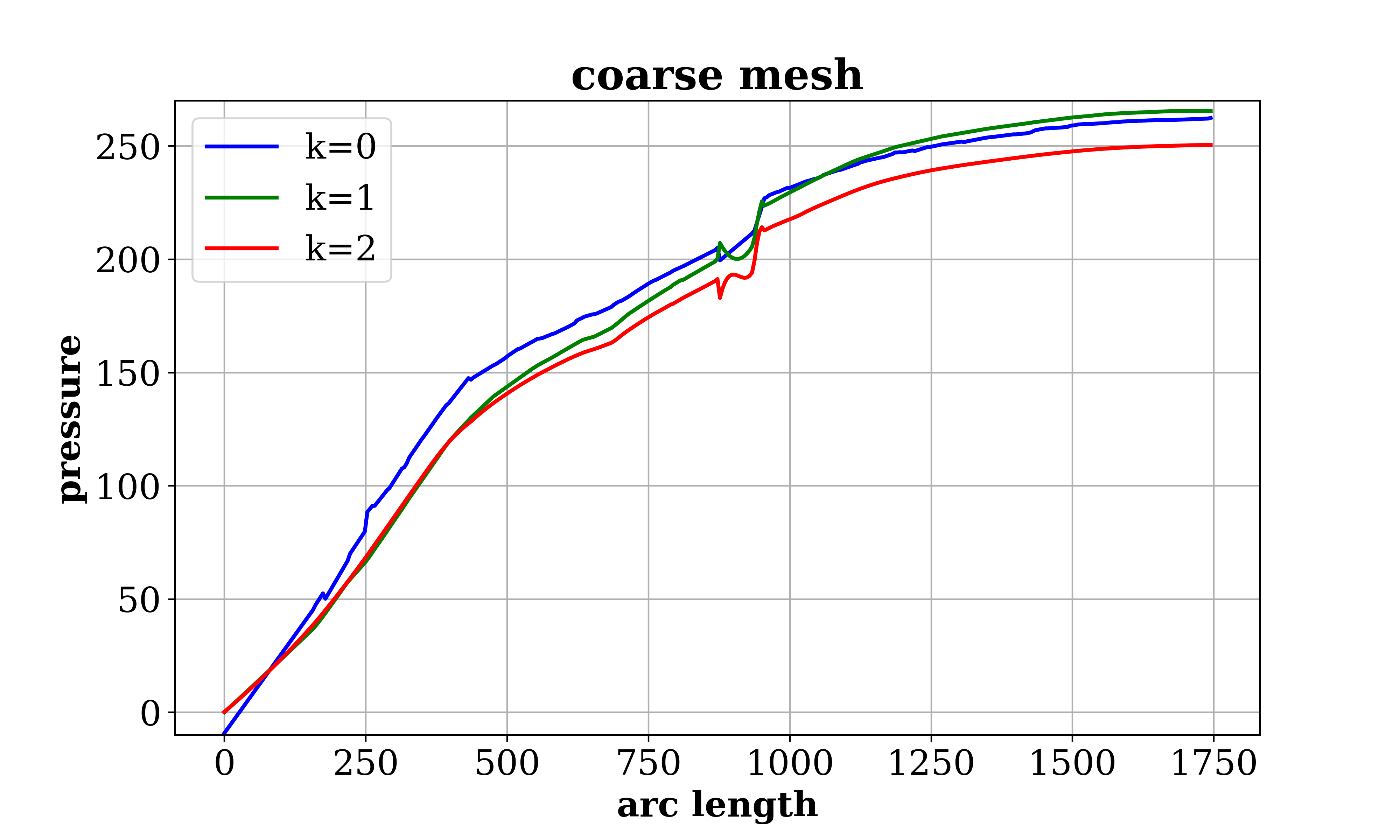}
\includegraphics[width=0.45\textwidth]{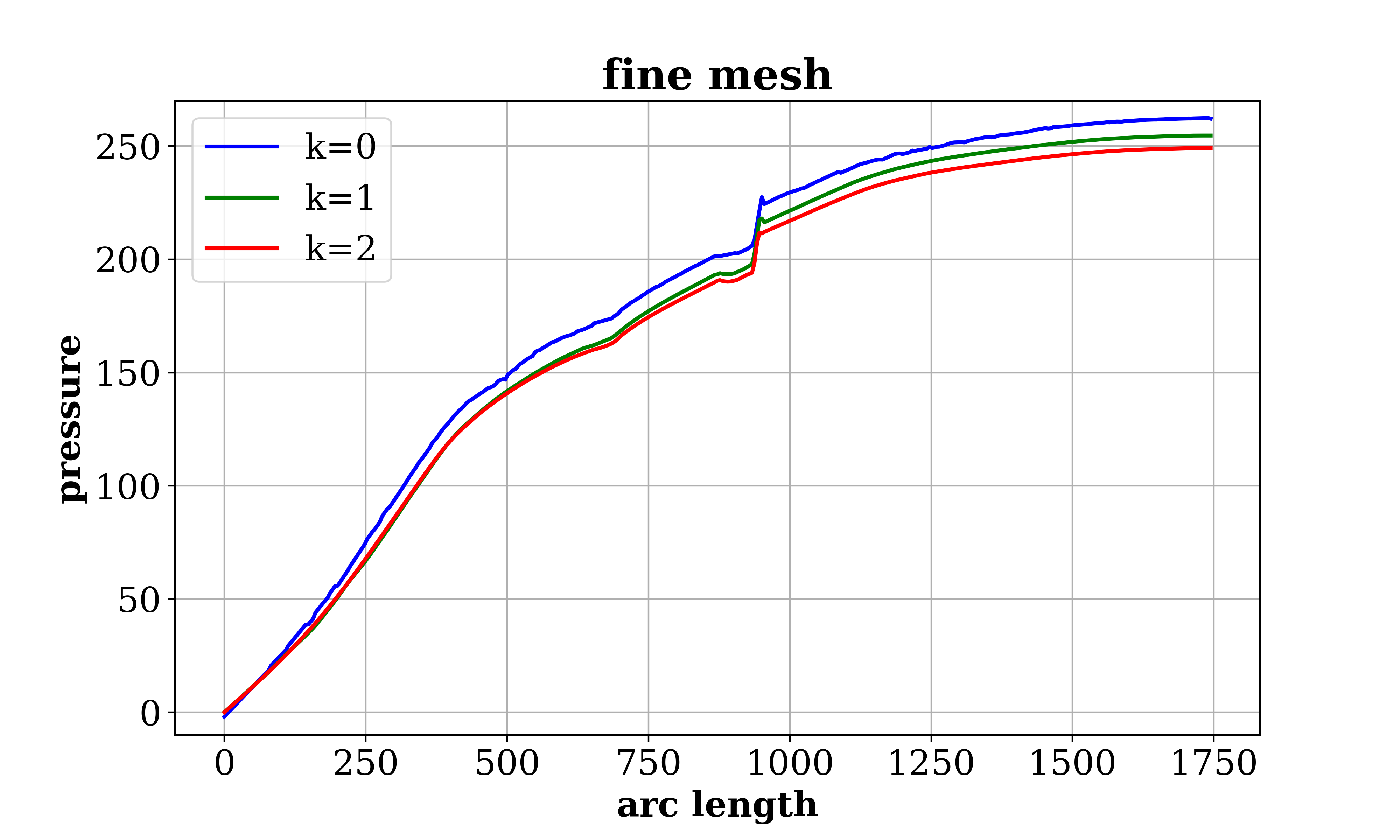}
\includegraphics[width=0.45\textwidth]{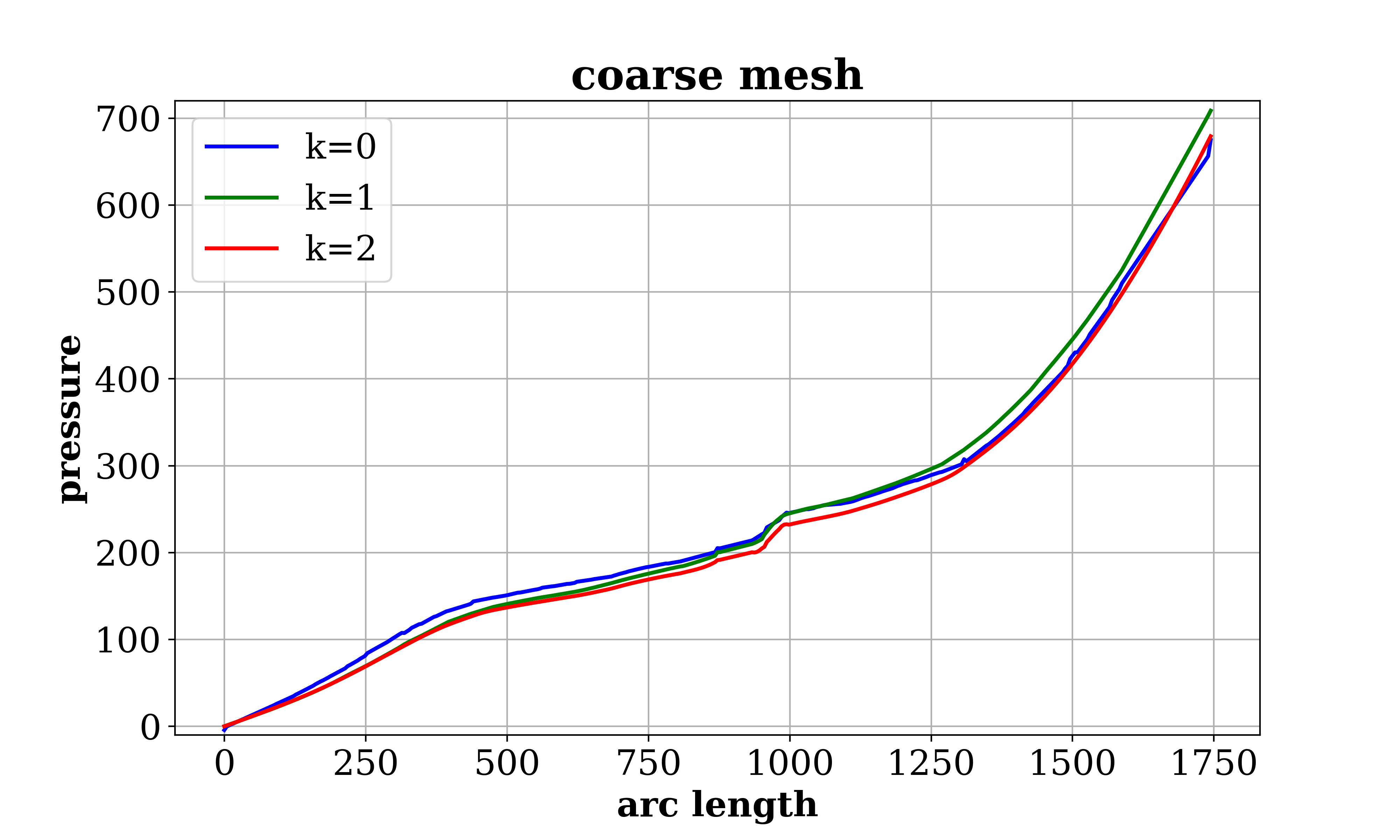}
\includegraphics[width=0.45\textwidth]{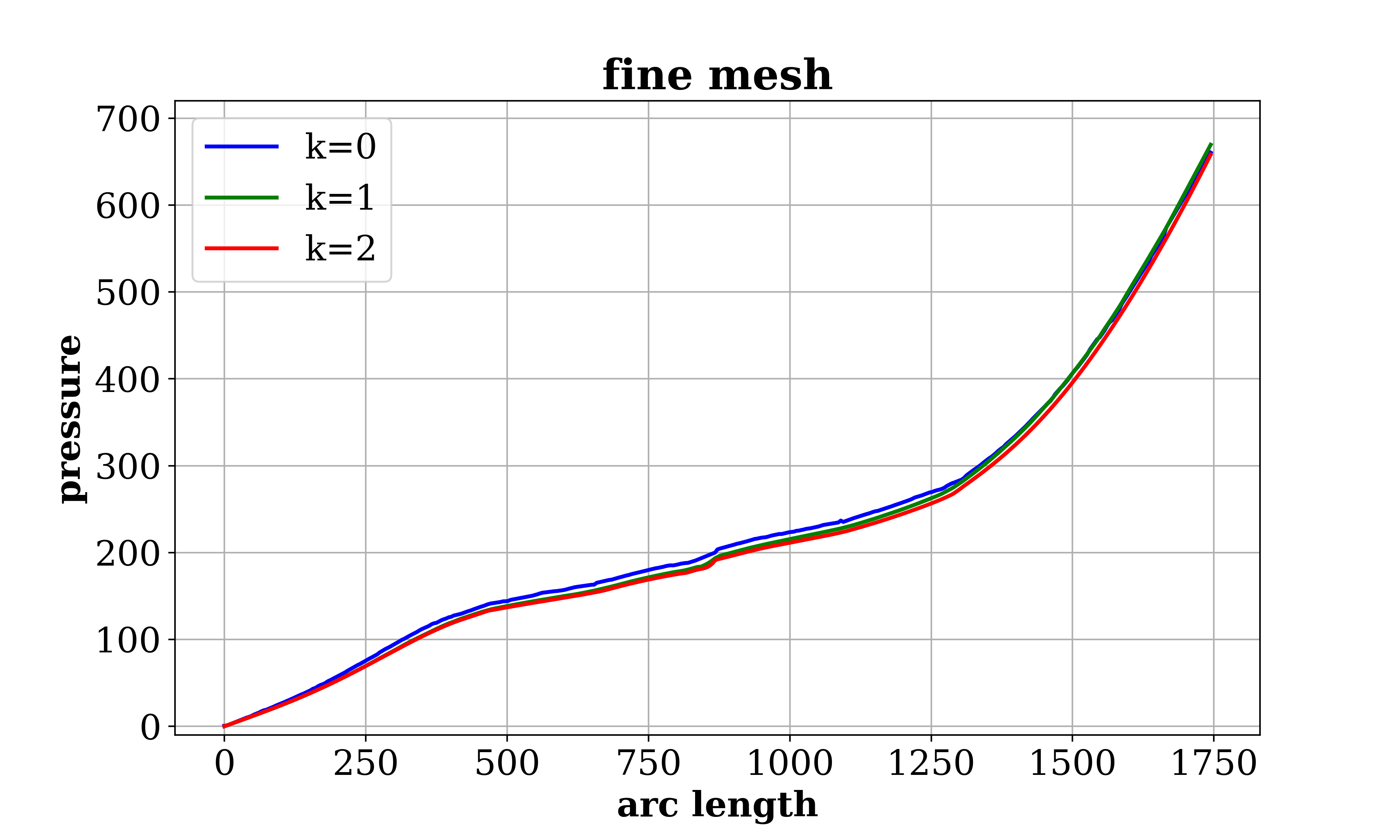}
\caption{\textbf{Example 6: Case (c).} Pressure along line $(-500, 100,-100)$--$(350, 1500, 500)$ (top) and line 
  $(350, 100,-100)$--$(-500, 1500, 500)$ (bottom).}
\label{fig:ex6-cut3}
\end{figure}

We now plot the pressure profile on the fine mesh for $k=2$ for the three cases in Figure~\ref{fig:ex6-cont}.
These contour plots are similar to the 2D results where the effects of conductive and blocking fractures are completely different as expected. 
% In particular, we notice that for case (c), the pressure profile is mainly determined by the two  blocking fractures (red in Figure~\ref{fig:ex6-msh}) which 
% effectively blocked most of the flow and makes the pressure profile for case (c) and case(b) to be quite similar.

\begin{figure}[ht]
\centering
\includegraphics[width=0.32\textwidth]{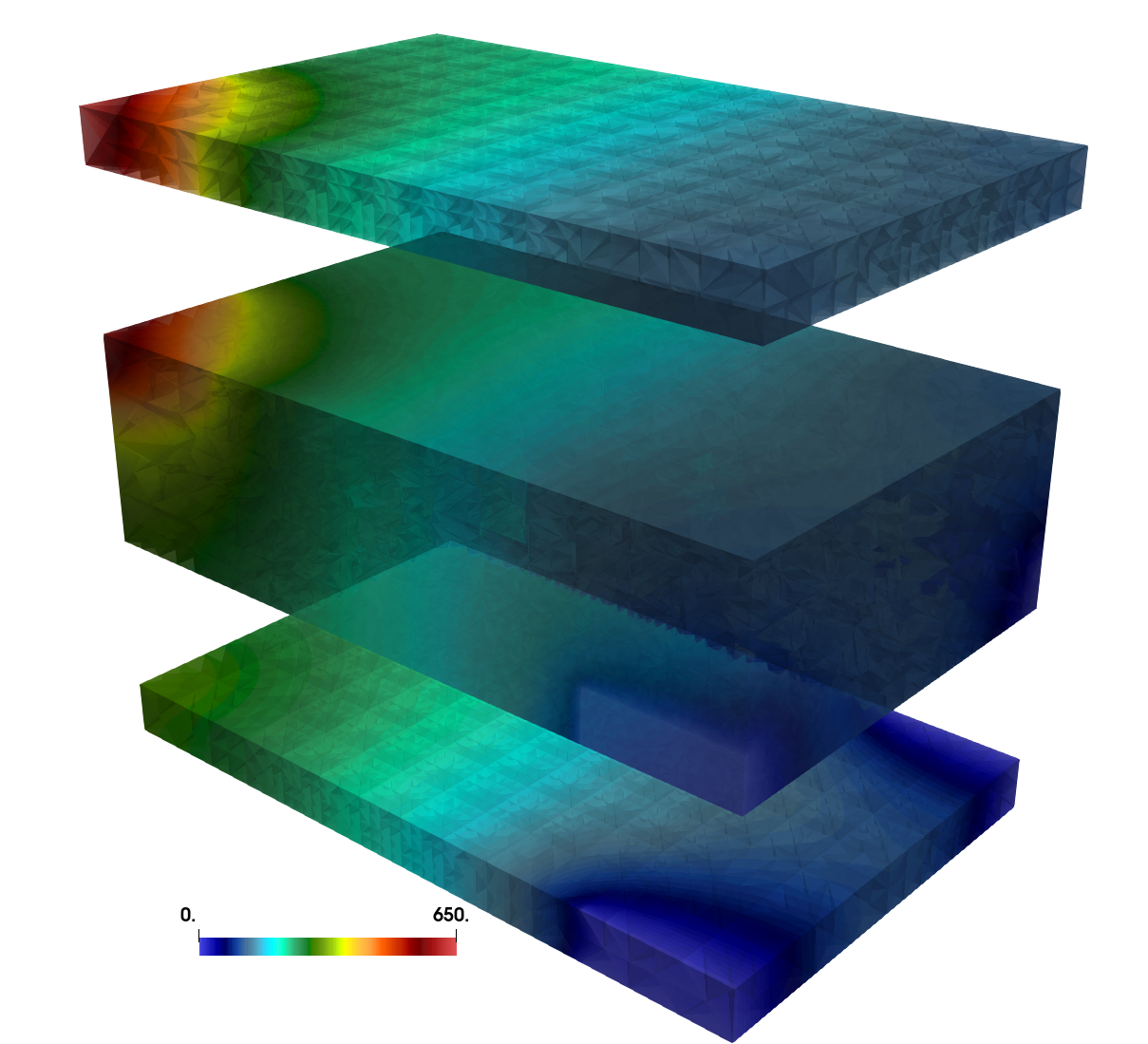}
\includegraphics[width=0.32\textwidth]{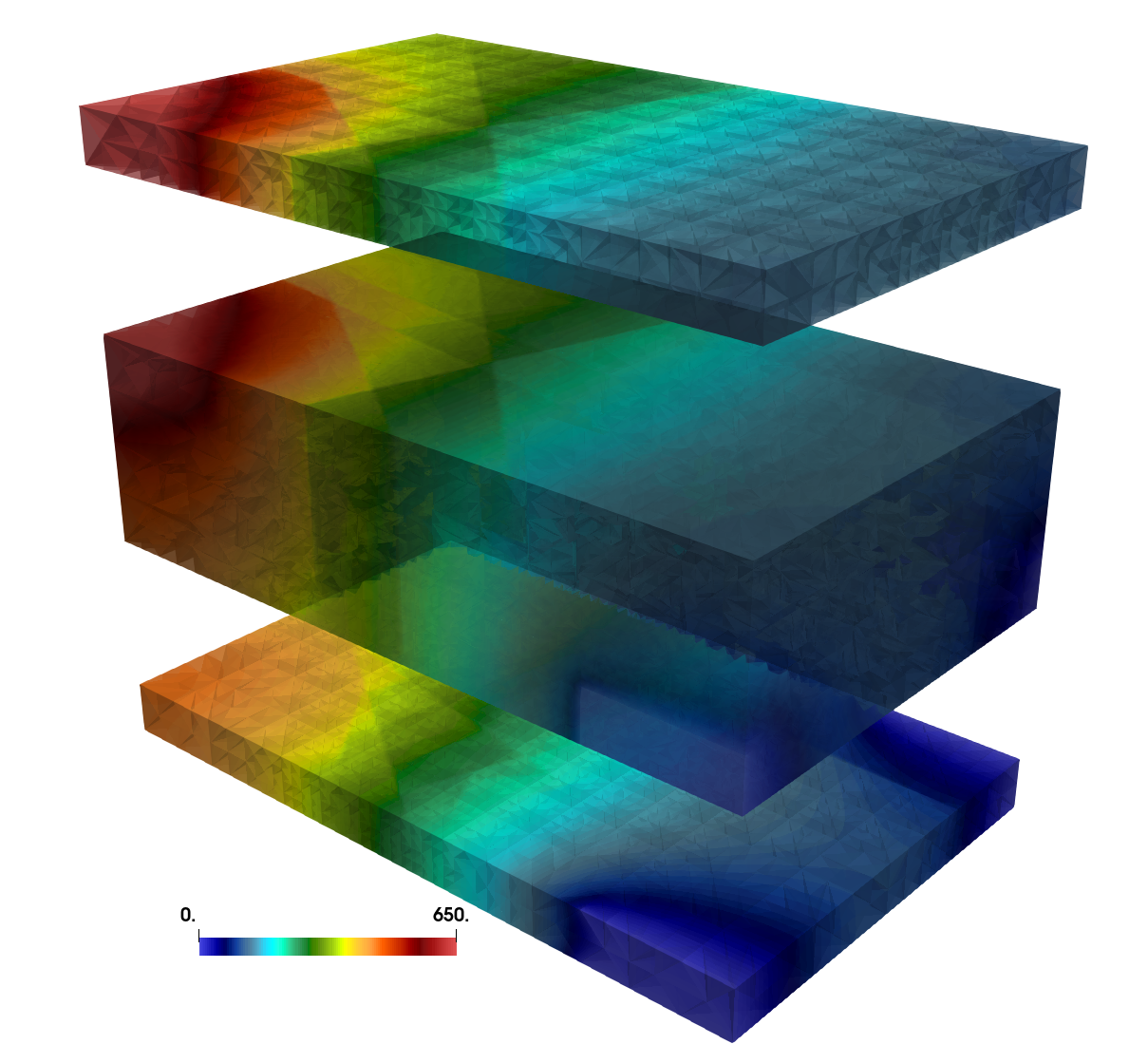}
\includegraphics[width=0.32\textwidth]{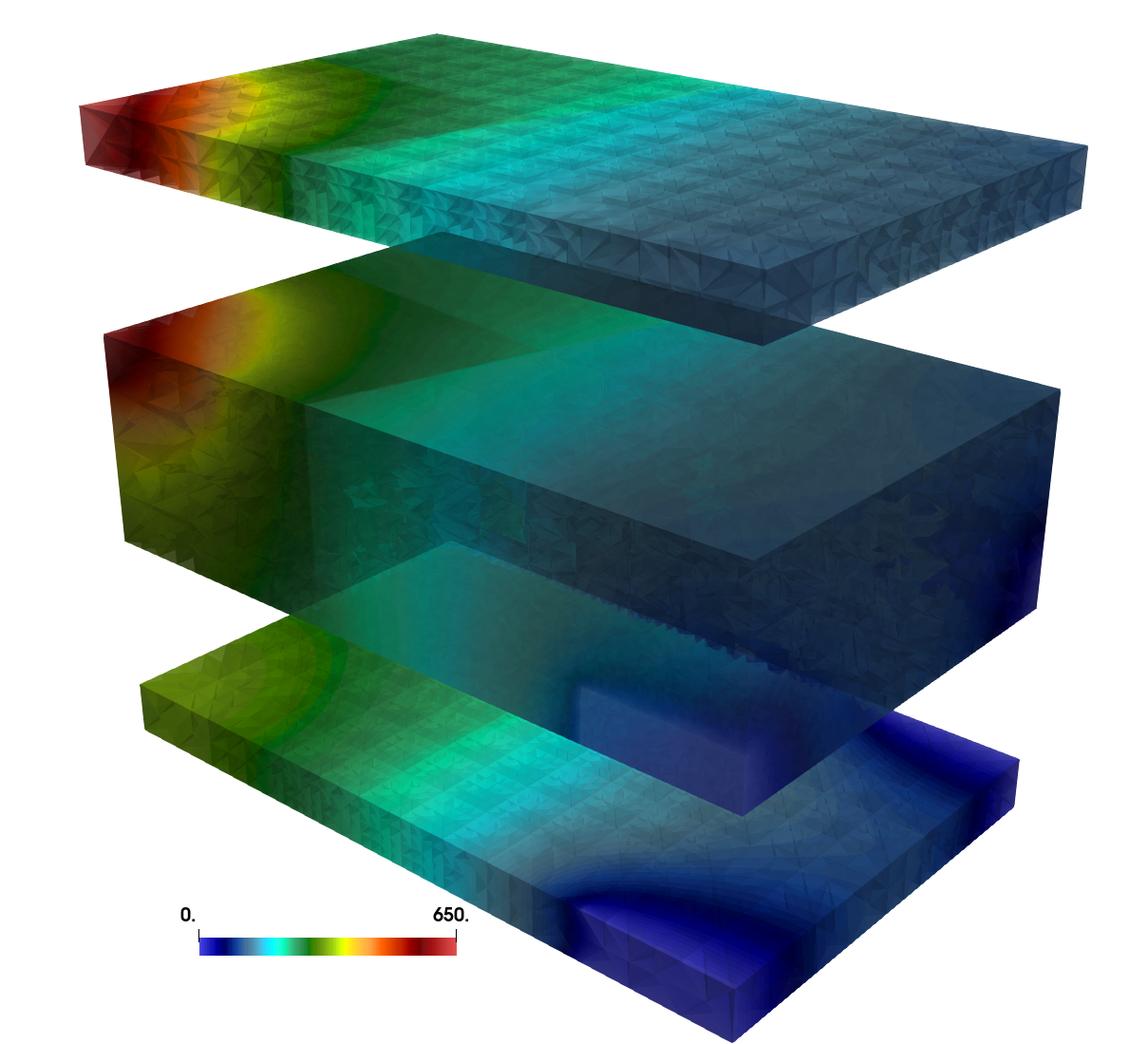}
\caption{\textbf{Example 6.} Pressure contour for 
$k=2$ on the fine mesh. 
Left: case (a). Middle: case (b). Right: case (c).
}
\label{fig:ex6-cont}
\end{figure}

Let us finally briefly comment on the computational cost on the fine mesh where the major bottleneck is the global linear system solve.
Here we use NGsolve's built-in parallel sparse Cholesky factorization to solve this global SPD linear system  on a 64-core server with two Two AMD EPYC 7532  processors which has 256G memory. 
For $k=0$, there are $439k$ global DOFs, and the linear system solver takes 30G memory and 4 seconds wall clock time; 
for $k=1$, there are $1.32$ million global DOFs, and the linear system solver takes 50G memory and 48 seconds wall clock time; 
for $k=2$, there are $2.64$ million global DOFs, and the linear system solver takes 100G memory and 285 seconds wall clock time.
More efficient solver like multigrid may significantly reduce the memory consumption and overall solver time. We will investigate this issue in our future work.

\section{Conclusion}
\label{sec:conclude}
We presented a novel HDG scheme on unfitted meshes for fractured porous media flow with both blocking and 
conductive fractures based on the RDFM using Dirac-$\delta$ functions approach to handle the fractures. Well-posedness of the method is established. Our scheme is relatively easy to implement comparing with most existing fractured porous media flow solvers in the literature that can simultanuously handel blocking and conductive fractures. 
In fact, we simply modify a regular porous media flow HDG solver by including two surface integrals related to the blocking and conductive fractures which are represented as (multi-)level set functions, and 
properly adjust the penalty parameters in the numerical flux on those fractured cells. 
No lower dimensional fracture modeling is needed in our approach. Besides the ease of using unfitted meshes in our scheme, we also maintain local conservation as typical of DG methodologies. Moreover, the resulting linear system can be solved efficiently via static condensation, which leads to a global coupled SPD linear system, and higher order pressure postprocessing is also available.  
The proposed HDG scheme is extensively tested against various benchmark examples in two- and three-dimensions. Satisfactory results are observed even when both blocking and conductive fractures co-exist in the computational domain.

Our future work includes the detailed study of the stabilization function on the performance of the scheme, and their variable-order and hybrid-mixed variants. We will also investigate robust preconditioning techniques for the global SPD linear system, and extend our solver to multiphase fractured porous media flows. 

% \section*{Appendix: sample code for single fracture in 3D}
% \lstinputlisting[language=Python]{mesh.py}

% \bibliography{dfm.bib}
\bibliographystyle{ieeetr}

\end{document}